\documentclass[letterpaper, 10pt,reqno]{amsart}

\usepackage[usenames,dvipsnames]{xcolor}

\usepackage[portrait,margin=3cm]{geometry}
\usepackage{mathrsfs}

\usepackage{amssymb,amsthm,amsfonts,amsbsy,latexsym}
\usepackage[reqno]{amsmath}

\usepackage{color}
\usepackage{graphicx}
\usepackage{bbm}
\usepackage{comment}
\usepackage{bm}
\usepackage{a4wide,graphicx, listings,lscape, epstopdf, units, textcomp, fancyhdr, url, soul, color,subcaption}
\usepackage[textsize=small]{todonotes}
\usepackage{enumitem}

\usepackage{mathtools}
\mathtoolsset{showonlyrefs}

\definecolor{MyDarkblue}{rgb}{0,0.08,0.50}
\definecolor{Brickred}{rgb}{0.65,0.08,0}

\usepackage{hyperref}
\hypersetup{
	colorlinks=true,       
	linkcolor=MyDarkblue,          
	citecolor=Brickred,        
	filecolor=red,      
	urlcolor=cyan           
}

\usepackage{orcidlink}

\usepackage{pgfplots}

\newtheorem*{theorem*}{Theorem}
\newtheorem{theorem}{Theorem}[section]
\newtheorem{lemma}[theorem]{Lemma}

\newtheorem{proposition}[theorem]{Proposition}
\newtheorem{corollary}[theorem]{Corollary}
\newtheorem{problem}[theorem]{Open Problem}

\newtheorem{conjecture}[theorem]{Conjecture}

\theoremstyle{definition}
\newtheorem{definition}[theorem]{Definition}

\newtheorem{remark}[theorem]{Remark}

\newenvironment{assK}{
	\textbf{Assumption $\boldsymbol{\cK}$. }}{}

\newenvironment{assKalpha}{
	\textbf{Assumption $\boldsymbol{\cK_\alpha}$. }}{}

\renewcommand{\P}{\mathbb{P}}

\newcommand{\Pv}{\mathbb{P}}

\newcommand{\eps}{\varepsilon}

\newcommand{\cA}{\mathcal{A}}\newcommand{\cB}{\mathcal{B}}
\newcommand{\cD}{\mathcal{D}}\newcommand{\cE}{\mathcal{E}}\newcommand{\cF}{\mathcal{F}}
\newcommand{\cI}{\mathcal{I}}
\newcommand{\cK}{\mathcal{K}}
\newcommand{\cO}{\mathcal{O}}
\newcommand{\cR}{\mathcal{R}}
\newcommand{\cS}{\mathcal{S}}\newcommand{\cU}{\mathcal{U}}

\newcommand{\Var}{{\rm Var}}

\newcommand{\e}{{\mathrm e}}

\newcommand{\R}{\mathbb{R}}
\newcommand{\N}{\mathbb{N}}
\newcommand{\Z}{\mathbb{Z}}

\newcommand{\dd}{\mathrm{d}}

\renewcommand{\emptyset}{\varnothing}

\newcommand*{\wt}{\widetilde}

\newcommand*{\be}{\begin{equation}}
	\newcommand*{\ee}{\end{equation}}
\newcommand*{\ba}{\begin{aligned}}
	\newcommand*{\ea}{\end{aligned}}
\newcommand*{\barr}{\begin{array}{c}}
	\newcommand*{\earr}{\end{array}}
\def \toinp    {\overset  \Pv \longrightarrow}
\def \toindis  {\overset {\mathrm{d}}{\longrightarrow}}
\def \toas     {\overset {\mathrm{a.s.}}{\longrightarrow}}

\newcommand*{\ind}{\mathbbm{1}}
\makeatletter
\def\namedlabel#1#2{\begingroup
	#2%
	\def\@currentlabel{#2}%
	\phantomsection\label{#1}\endgroup
}

\makeatother

\newcommand{\bes}{\begin{equation*}}
	\newcommand{\ees}{\end{equation*}}

\renewcommand{\P}[1]{\mathbb{P}\!\left(#1\right)}
\newcommand{\E}[1]{\mathbb{E}\left[#1\right]}

\renewcommand{\N}{\mathbb{N}}

\newcommand{\I}{\mathbb{I}}

\renewcommand{\th}{\mathrm{th}}
\newcommand{\cont}{\mathrm{cont}}

\numberwithin{equation}{section}

\renewcommand{\e}{\mathrm{e}}

\newcommand{\bp}{\mathrm{BP}}
\newcommand{\floor}[1]{\lfloor #1\rfloor}
\newcommand{\ceil}[1]{\lceil #1\rceil}
\newcommand{\Ps}[1]{\mathbb P_{\cS}\left(#1\right)}
\newcommand{\toinps}{\xrightarrow{\mathbb P_\cS}}

\newcommand{\ensymboldremark}{\hfill$\blacktriangleleft$}

\newcommand{\invisible}[1]{}

\makeatletter
\newcommand{\leqnomode}{\tagsleft@true\let\veqno\@@leqno}
\newcommand{\reqnomode}{\tagsleft@false\let\veqno\@@eqno}
\makeatother
\newlength{\tagmarginsep} 
\title[Lack of Persistence of the Maximum Degree in PAVD models]{Preferential Attachment Trees with Vertex Death: Lack of Persistence of the Maximum Degree}
\author{Markus Heydenreich\orcidlink{0000-0002-3749-7431}}
\author{Bas Lodewijks\orcidlink{0000-0001-5624-2410}}
\address{Universität Augsburg, Department of Mathematics, D-86135 Augsburg, Germany}
\email{markus.heydenreich@uni-a.de, bas.lodewijks@uni-a.de}
\date{\today} 

\begin{document}

	\begin{abstract}
		We consider an evolving random discrete tree model called \emph{Preferential Attachment with Vertex Death}, as introduced by Deijfen~\cite{Dei10}. Initialised with an alive root labelled $1$, at each step $n\geq1$ either a new vertex with label $n+1$ is introduced that attaches to an existing \emph{alive} vertex selected preferentially according to a function $b$, or an alive vertex is selected preferentially according to a function $d$ and \emph{killed}. 
		In this article we introduce a generalised concept of \emph{persistence} and lack thereof for evolving random graph models. Let $O_n$ be the smallest label among all alive vertices (the oldest alive vertex), and let $I_n$ be the label of the alive vertex with the largest degree (among all alive vertices). Persistence occurs when $I_n/O_n$ is tight, whereas lack of persistence occurs when $I_n/O_n$ diverges with $n$. 
		
		We study \emph{lack of persistence} in this article and we identify two regimes: the `{old are rich}' regime and the `{rich die young}' regime. In the `rich are old' regime, though the oldest alive vertices in the tree typically have the largest degrees, lack of persistence can occur subject to the non-summability condition $\sum_{i=0}^\infty 1/(b(i)+d(i))^2=\infty$, under which `lucky' vertices that are slightly younger than the oldest vertices can attain the largest degrees by step $n$. This generalises known results by Banerjee and Bhamidi~\cite{BanBha21} whilst also removing a technical assumption in their work. In contrast, lack of persistence always occurs in the ‘rich die young’ regime, without the need of the non-summability condition. This regime is entirely novel and cannot be observed in preferential attachment models without death. 
		Here, vertices can survive for an exceptionally long time by obtaining a low degree, whereas vertices with a large degree die much faster, causing lack of persistence to occur.  
		A main technique is an embedding of the discrete tree process into a Crump-Mode-Jagers branching process, and a higher-order analysis of the resulting birth-death mechanism based on moderate deviation principles with exponential tilting. 
	\end{abstract}
	
	\maketitle
	
	\tableofcontents
	
	\section{Introduction}\label{sec:intro}
	
	Since the late 1990s and early 2000s, the amount of research into random graphs as models for real-world networks has grown tremendously. The universal behaviour observed in many different real-world networks coming from distinct and unrelated contexts drove the motivation to understand the formation and structure of these networks from a more abstract and theoretical perspective. Within the fields of network science, statistical physics, and mathematics, this has yielded a large variety of random graph models, often with particular features that serve certain modelling purposes or aim to explain the underlying principles that potentially govern how real-world networks form. 
	
	Within this `random graph zoo', a distinction can be made between \emph{static} and \emph{evolving} random graphs. The former consists of models where a large random graph is considered to be a snapshot of a real-world network, but where there is no direct link or correlation between the random graph of size $n$ (e.g.\ with $n$ vertices) and that of size $n+1$. The configuration model is a clear example. Evolving random graphs, on the other hand, attempt to model the formation of a real-world network through time. Here, new vertices (and edges) are added to the graph sequentially, and new vertices (may) connect to already present vertices. The preferential attachment model is possibly the most famous example of a family of such evolving random graphs. 
	
	These preferential attachment models have received a wealth of attention over the years and are well-understood nowadays. Still, a clear gap between most of these models and networks that form in the real world is the fact that preferential attachment graphs allow for \emph{growth}  only. Vertices and edges are sequentially attached to the graph, causing the graph to grow. However, this is not particularly realistic, as nodes and bonds in real-world networks can be \emph{removed} in almost all contexts as well. As an example, people are born but also pass away; friendships are formed but can also be broken; scientific articles receive citations from other articles but can become irrelevant over time as the field progresses; users of social media can befriend or follow others, but also unfriend or unfollow them and even remove their accounts all together. 
	
	The presence of both growth and shrinking of a network is clearly represented in the case of population sizes, as displayed in Figure~\ref{fig:popsize}. The population sizes of different countries such as Austria, Bosnia and Herzegovina, Croatia, and Portugal are, at times, subject to contrasting trends, showing both growth and decrease in size. This becomes even clearer when looking at certain sub-groups of the population, such as females within the whole population, or females aged 25-29 within the female population, as exemplified in Figure~\ref{fig:propfem}. Overall, however, the total population of the world is increasing without much fluctuation in time.
	
	\begin{figure}[h]
		\includegraphics[width=0.5\textwidth]{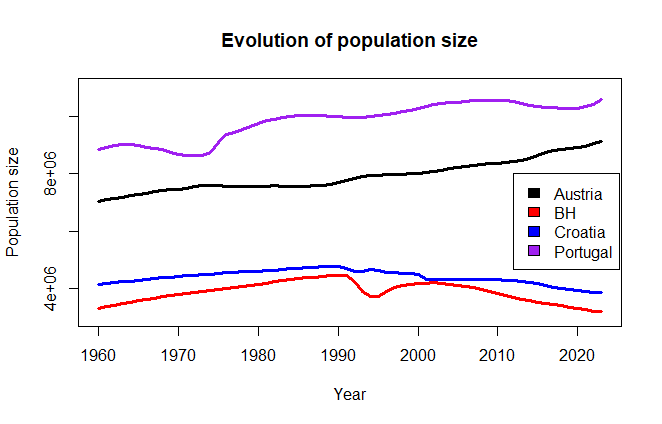}
		\caption{Population size of the countries Austria, Bosnia and Herzegovina, Croatia, and Portugal, from 1960 until 2023~\cite{Pop}. }\label{fig:popsize}
	\end{figure}
	
	\begin{figure}[h]
		\centering
		\begin{subfigure}{0.45\textwidth}
			\includegraphics[width=\textwidth]{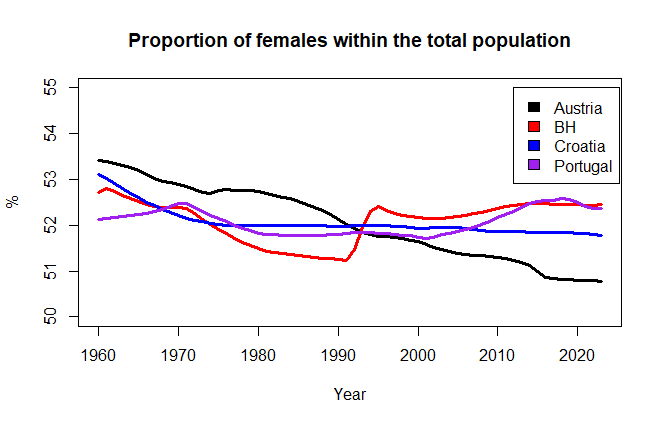}
		\end{subfigure} 
		~
		\begin{subfigure}{0.45\textwidth}
			\includegraphics[width=\textwidth]{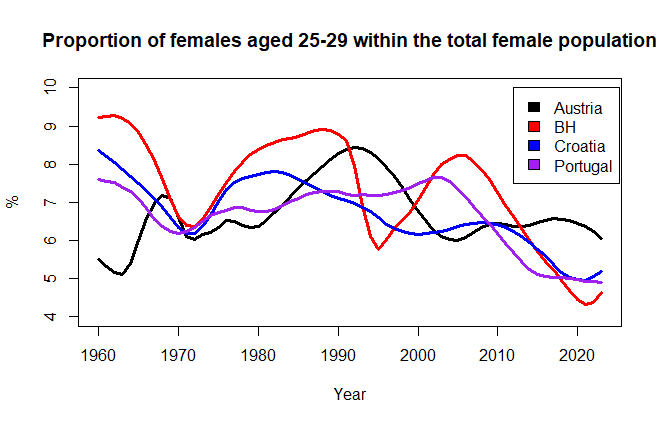}
		\end{subfigure}
		\caption{Proportion of females within the total population~\cite{Popfem} (left) and the proportion of females aged 25 to 29 years old within the total female population~\cite{Popfem2529} (right) of the countries Austria, Bosnia and Herzegovina, Croatia, and Portugal, from 1960 until 2023.}\label{fig:propfem}
	\end{figure}

	To more closely match this growth and shrinking of real-world networks, a limited amount of work has focused on models that allow for vertices and/or edges to also be removed from the network. Generally speaking, in such models at each step, 
	\begin{enumerate}[label=(\arabic*)] 
		\item A new vertex is added to the graph with probability $p_1$, which connects to $m$ already existing vertices (preferentially or uniformly at random).
		\item A new edge is added to the graph between existing vertices with probability $p_2$, where the existing vertices are selected preferentially and/or uniformly. 
		\item A vertex is selected preferentially or uniformly and is removed with probability $p_3$ (including the edges incident to it).
		\item An edge is selected uniformly and is removed with probability $p_4$.
	\end{enumerate}
	Examples include the models studied by Cooper et al.~\cite{CooFrieVer04} (with $m\in\N$) with further work of Lindholm and Vallier~\cite{LinVal11} and Vallier~\cite{Val13}, Chung and Lu~\cite{ChungLu04} (with $m=1$), Deo and Cami~\cite{DeoCami07} (with $m=1$, $p_2=p_4=0$, and vertex deletion is done anti-preferentially, i.e.\ favouring vertices with low degree), and Deijfen and Lindholm~\cite{DeiLind09} (with $m=1$, $p_3=0$). Another model of Britton and Lindholm~\cite{BritLind10}, with further work in~\cite{BritLin11,KuckSchu20}, assigns a random fitness value to each vertex, and edges are added between existing vertices with a probability proportional to the fitness values of the vertices. Work on duplication-divergence type models with edge deletion includes, among others, \cite{Thor15,HerPfaf19,BarLo21,LoReiZha25}. Additionally, there is some non-rigorous work on similar models~\cite{MooGhosNew06,SidMirEm23,SarRoy04}.
	
	The main focus in the analysis of these models is often the degree distribution, in particular for which parameter choices the power-law behaviour of the degree distribution is lost due to vertex/edge removal. Cooper et al.~\cite{CooFrieVer04} and Chung and Lu~\cite{ChungLu04} do study other properties such as typical distances and the diameter, but this is under rather limiting assumptions. 
	
	A different but related family of models is preferential attachment with \emph{ageing}. Here, vertices and edges are not removed, but vertices become less likely to make new connections over time as they age. The main motivation for these models comes from citation networks, where both `young' and well-cited papers are more likely to be cited than `old' and poorly-cited papers, and most papers cease to receive citations after some time. Most work focuses on fitting such models to datasets of citation networks~\cite{HajSen05,HajSen06,HazKulSkiDil17,WangSongBar13,WangYuYu08,WangYuYu09}, whilst the work of Garavaglia et al.\ studies the degree distribution of such an preferential attachment model with ageing and multiplicative fitness~\cite{GarHofWoe17}.

	The work on these models is \emph{limited} in scope for a number of reasons. First, it generally only considers the degree distribution of the model, and if further properties are studied, this is under restrictive conditions and assumptions.  Second, all models mentioned only study affine preferential attachment or uniform attachment for attaching edges. Third, in most models the choice to add or remove vertices and edges at each step is determined by fixed model parameters (the $p_i$) and does not depend on the evolution of the graph itself. 
	
	Some recent work partially addresses these concerns. Diaz, Lichev, and the second author study a model with uniform attachment and uniform vertex removal ($m$ is random, $p_2=p_4=0$), where the local weak limit, existence of a giant component, and the size and location of the maximum degree are studied~\cite{DiazLichLod22}. Bellin et al.~\cite{BelBlaKamKor23,BelBlaKamKor23II} study a model of uniform attachment and uniform vertex removal ($m=1$, $p_2=p_4=0$), but where vertices are not removed but are `frozen' and can then no longer make new connects, and where the choice to freeze vertices need not be random. This allows them to study the model in a `critical window' where they obtain precise result for the height of the tree and its scaling limit.
	
	Finally,  Deijfen~\cite{Dei10} studies a Preferential Attachment tree model with Vertex Death (PAVD). Here, death is equivalent to freezing as in the work of Bellin et al.\ but carries a different name. In this model, the probability of killing a vertex or adding a vertex, as well as to which alive vertex this new vertex connects itself, is dependent on the state of the tree. In particular, it depends on the in-degrees of the alive vertices in a general way. Deijfen studied the limiting degree distribution of the tree conditionally on survival, and shows with a number of examples how introducing death can yield novel behaviour compared to preferential attachment trees without death.  	
	
	In this article we focus on the PAVD model introduced and studied by Deijfen. To this end, let us provide a definition of the model. For a tree $T$, we naturally think of its edges as being directed towards the root. We then let $\deg_T(v)$ denote the in-degree of a vertex $v$ in $T$. For a sequence of trees $(T_n)_{n\in\N}$, we  write $\deg_n(v)$ for $\deg_{T_n}(v)$ for ease of writing.
	
	\begin{definition}[Preferential Attachment with Vertex Death]\label{def:pavd}
		Let $b\colon\N_0\to (0,\infty)$ and $d\colon \N_0\to [0,\infty)$ be two sequences. We recursively construct a sequence of trees $(T_n)_{n\in\N}$ and a sequence of sets of vertices $(\cA_n)_{n\in\N}$ as follows. We initialise $T_1$ as a single vertex labelled $1$ and $\cA_1=\{1\}$. For $n\geq 1$, conditionally on $T_n$ and $\cA_n$, if $\cA_n\neq \emptyset$, we select a vertex $i$ from $\cA_n$ with probability 
		\be 
		\frac{b(\deg_n(i))+d(\deg_n(i))}{\sum_{j\in \cA_n} b(\deg_n(j))+d(\deg_n(j))}.
		\ee 
		Then, conditionally on $T_n$ and $i$, we either \emph{kill} vertex $i$ with probability 
		\be 
		\frac{d(\deg_n(i))}{b(\deg_n(i))+d(\deg_n(i))}, 
		\ee 
		and set $T_{n+1}=T_n$ and $\cA_{n+1}=\cA_n\setminus\{i\}$, or otherwise construct $T_{n+1}$ from $T_n$ by introducing a new vertex $n+1$ which we connect by a directed edge to $i$ and set $\cA_{n+1}=\cA_n\cup\{n+1\}$.
		
		If $\cA_n=\emptyset$, we terminate the recursive construction, set $T_i=T_n$ for all $i>n$, and say that \emph{the tree has died}.
	\end{definition} 
	
	We see from the definition that the probabilities to both select and kill a vertex depend on the evolution of the tree itself, and that the sequences $b$ and $d$ are not restricted to the uniform case ($b$ and/or $d$ constant) or the affine case ($b(i)=b_1i+b_2$ and/or $d(i)=c_1i+c_2$).
	
	\textbf{Our contribution.} We further the theoretical knowledge of the PAVD model by analysing \emph{lack of persistence of the maximum degree}. Persistence (or lack thereof) of the maximum degree in evolving random graphs is the emergence of a \emph{fixed} vertex that attains the largest degree in the graph for all but finitely many steps. When there is lack of persistence, such a fixed vertex does not exist and the maximum degree changes hands infinitely often (i.e.\ the maximum degree is attained by different vertices infinitely often). Persistence of the maximum degree is also known as degree centrality and is often leveraged in network archaeology and root-finding algorithms (see e.g.~\cite{BanBha22,BriCalLug23,BubMosRac15} and references therein)  In this way, we thus address the three concerns raised earlier, as we analyse an intricate property that requires an in-depth understanding of the model in a very general setting under only mild assumptions. 
	
	The notion of persistence in evolving random graphs has been studied from different perspectives, starting with the work of Dereich and M\"orters~\cite{DerMor09}. Later, sufficient conditions under which persistence and the lack thereof hold have been weakened or omitted by Galashin~\cite{Gal13}, Banerjee and Bhamidi~\cite{BanBha21}, and Iyer~\cite{Iyer24}. In short (and omitting minor technical assumptions), using the PAVD formulation with $d\equiv 0$, so that there is no vertex death and we recover the classic preferential attachment model, persistence of the maximum degree occurs almost surely if and only if 
	\be \label{eq:sumcond}
	\sum_{i=0}^\infty\frac{1}{b(i)^2}<\infty.
	\ee 
	Heuristically, persistence occurs when high-degree vertices have a sufficiently strong advantage over low-degree vertices, so that they are more likely to make more connections and increase their in-degree further. This allows one vertex to attain the largest degree for all but finitely many steps, instead of the largest degree switching between vertices infinitely often. The transition between whether the advantage is strong enough or not lies exactly at the point when the series in~\eqref{eq:sumcond} goes from summable to non-summable. This is a \emph{robust} property, in the sense that changes to finitely many values $b(i)$ does not change the large-scale behaviour of the model.
	
	When vertices are killed and can no longer make new attachments afterwards, as in the PAVD model, the notion of persistence as stated above is not particularly interesting. Indeed, if a vertex is killed after a finite (random) number of steps almost surely, then lack of persistence always occurs. The largest degree in the first $\cO(1)$ many vertices is of order $\cO(1)$ as well. There are many more alive vertices after $n$ steps, however, one of which will be sufficiently lucky to obtain a degree larger than $\omega_n$, where $\omega_n$ tends to infinity at an arbitrarily slow rate. Only when vertices are never killed with positive probability is the question whether persistence does or does not occur interesting. However, this setting, as we shall see, is not too dissimilar from preferential attachment without death. 
	
	More interesting is the question whether the \emph{oldest alive individual}, that is, the alive individual with the smallest label, and the \emph{alive individual with the largest degree}, i.e.\ the \emph{richest} individual, have labels that are `close' or `far apart'. In the case of preferential attachment without death all individuals are alive, so that the oldest alive vertex is always the vertex with label $1$. The ratio of their labels thus equals the label of the maximum degree vertex, which converges in the case of persistence and does not (in fact, it tends to infinity) in the case of lack of persistence. Viewing persistence from this more general perspective thus allows for a more general problem that is non-trivial to study in general preferential attachment trees with death. 
	
	In this article we show when, under some mild technical assumptions, \emph{lack of persistence occurs}. In particular, we identify two regimes in which different behaviour can be observed, which we coin the \emph{rich are old} and \emph{rich die young} regimes. In the `rich are old' regime, behaviour is to some extent similar to persistence of the maximum degree in preferential attachment models without death, which can be viewed as a special case that we generalise to a larger class of models that include vertex death. Here, we show that persistence does not occur when the summability condition in~\eqref{eq:sumcond} is not met, and thus agrees with what is known for preferential attachment trees without death. The asymptotic behaviour of the labels of the oldest alive vertex and the vertex with the largest degree is distinct from classical preferential attachment, however. 
	
	The analysis of the `rich die young' regime is entirely novel and has not been studied previously, as it cannot be attained by preferential attachment models without death. Here we show, under mild assumptions on the regularity of the sequences $b$ and $d$, that \emph{lack of persistence always occurs} and that the condition in~\eqref{eq:sumcond} is no longer relevant. As the name of the regime already suggests, this stark contrast in behaviour in this regime is caused due to the fact that the values $d(i)$ are `too large' for all large $i$, which causes vertices who manage to obtain large degrees to be killed much faster than individuals with small degree. Such small degree individuals can avoid being selected for a long time in the first place, so that they do not increase their degree nor die, but consequently manage to stay alive for a larger amount of time compared to high-degree vertices. 
	
	Our analysis extends the methodology developed by Deijfen in~\cite{Dei10} using an embedding of the discrete tree process in a  continuous-time branching process known as a Crump-Mode-Jagers branching process. We use a more precise formulation of this embedding compared to Deijfen, which allows us to simplify certain proofs of results in~\cite{Dei10}, but mainly to develop novel and precise results regarding the behaviour of the branching process embedding that were unattainable previously. Furthermore, we extend ideas used by Banerjee and Bhamidi in~\cite{BanBha21} to precisely analyse the optimal window in which the vertex with the largest degrees are found in the `rich are old' regime. Here, we have to balance the likelihood that vertices attain a large degree \emph{and survives} with the growth rate of the tree. Taking into account the survival of vertices is novel for this model compared to preferential attachment without death and significantly complicates the analysis.
	
	Finally, the analysis in the `rich die young' regime uses similar ideas as for the `rich are old' regime, though we carry this out in greater generality in this regime. In particular, we show that the transition between the `rich are old' and `rich die young' regimes is \emph{not robust}, in the sense that changing even a single value of $b(i)$ and/or $d(i)$ can take the behaviour from one to the other regime, entirely altering the behaviour of persistence of the model. This is in stark contrast with the robustness of the summability condition in~\eqref{eq:sumcond}. 
	
	\textbf{Structure of the paper. } We state the main results and the necessary assumptions in Section~\ref{sec:model}. The methodology used in the analysis in presented in Section~\ref{sec:embed}, followed by a heuristic explanation of the results in Section~\ref{sec:heur}. We collect some preliminary results and tools for the proofs of the main results in Section~\ref{sec:prelim}. These are used in Sections~\ref{sec:old} through~\ref{sec:mainproofs}. Section~\ref{sec:old} studies the label of the oldest vertex, Section~\ref{sec:max} focusses on the largest degree and the label of the vertex that attains it in the `rich are old' regime, and Section~\ref{sec:maxrdy} focusses on the `rich die young' regime. These results are then combined to prove the main results in Section~\ref{sec:mainproofs}. Finally, we discuss the case of asymptotically constant functions $b$ and $d$ in Section~\ref{sec:const} and conclude with a discussion of our work and possibilities for future directions in Section~\ref{sec:disc}.
	
	\textbf{Notation. } 	Throughout the paper we use the following notation: we let $\N:=\{1,2,\ldots\}$ denote the natural numbers, set $\N_0:=\{0,1,\ldots\}$ to include zero and let $[t]:=\{i\in\N: i\leq t\}$ for any $t\geq 1$. For $x,y\in\R$, we let $\lceil x\rceil:=\inf\{n\in\Z: n\geq x\}$ and $\lfloor x\rfloor:=\sup\{n\in\Z: n\leq x\}$, and let $x\wedge y:=\min\{x,y\}$ and $x\vee y:=\max\{x,y\}$. For sequences $(a_n)_{n\in\N},(b_n)_{n\in\N}$ such that $b_n$ is positive for all $n$ we say that $a_n=o(b_n)$ and $a_n=\mathcal{O}(b_n)$ if $\lim_{n\to\infty} a_n/b_n=0$ and if there exists a constant $C>0$ such that $|a_n|\leq Cb_n$ for all $n\in\N$, respectively. We write $a_n=\Theta(b_n)$ if $a_n=\cO(b_n)$ and $b_n=\cO(a_n)$. For random variables $X,(X_n)_{n\in\N}$, and $Y$ we let $X_n\toindis X, X_n\toinp X$ and $X_n\toas X$ denote convergence in distribution, probability and almost sure convergence of $X_n$ to $X$, respectively. We say that $(X_n)_{n\in\N}$ is a tight sequence of random variables when for any $\eps>0$ there exists $K_\eps>0$ such that $\P{X_n>K_\eps}<\eps$ for all $n\in\N$. Further, $X\preceq Y$ (resp.\ $X\succeq Y$) denotes that $X$ is stochastically dominated by $Y$ (resp.\ $X$ stochastically dominates $Y$). For a sequence $(\cE_t)_{t\in \I}$ of events, where $\I=\N$ or $\I=[0,\infty)$, we say that $\cE_t$ holds with high probability when $\lim_{t\to\infty}\P{\cE_t}=1$.  Finally, for an event $\cS$ such that $\P{\cS}>0$ we let  $\Ps{\cdot}:=\mathbb{P}(\cdot\, |\,\cS)$ and  $\mathbb E_\cS{}{\cdot}:=\E{\cdot\,|\,\cS}$ denote the conditional probability measure and conditional expectation, respectively. Then, we let $\xrightarrow{\mathbb P_\cS-\mathrm{a.s.}}$ respectively $\overset{\mathbb  P_\cS}{\longrightarrow}$ denote almost sure convergence and convergence in probability with respect to the probability measure $\mathbb P_\cS$.

	\section{Results}\label{sec:model}
	
	In this section we present the main results and the necessary assumptions that we use. Recall the PAVD model, as in Definition~\ref{def:pavd}. We let 
	\be \label{eq:surv}
	\cS:=\bigcap_{n\in\N}\{\cA_n\neq \emptyset\}
	\ee 
	denote the event that the process \emph{survives}, that is, the construction is never terminated. We assume throughout that $\P{\cS}>0$. We let $\mathbb P_{\cS}$ denote the probability measure $\mathbb P$, conditionally on $\cS$. Upon $\cS$, the quantities of interest are the \emph{oldest alive individual} (i.e.\ with smallest label) and the \emph{alive individual with the largest degree} in $T_n$. We thus define
	\be 
	O_n:=\min \cA_n \qquad\text{and}\qquad \ I_n:=\min\{i\in\cA_n: \deg_n(i)\geq \deg_n(j) \text{ for all }j\in \cA_n\}.
	\ee 
	We are interested in the long-term behaviour of $I_n$ and $O_n$. In particular, whether
	\be\tag{P}\label{eq:pers} 
	\Big(\frac{I_n}{O_n}\Big)_{n\in\N}\text{ is a tight sequence of random variables with respect to $\mathbb P_\cS$}, 
	\ee 
	or whether 
	\be \tag{NP}\label{eq:nopers}
	\frac{I_n}{O_n}\xrightarrow{\mathbb P_\cS} \infty. 
	\ee
	We say \emph{persistence} occurs when~\eqref{eq:pers} holds, whereas \emph{lack of persistence} occurs when~\eqref{eq:nopers} holds. In this paper, we focus on showing that lack of persistence occurs for a large class of models.
	
	We define the following quantities.  Let $(E_i)_{i\in\N_0}$ be a sequence of independent exponential random variables, where $E_i$ has rate $b(i)+d(i)$, and let $(B_i)_{i\in\N_0}$ be a sequence of independent Bernoulli random variables (also independent of the $E_i$), such that 
	\be 
	\P{B_i=1}=\frac{b(i)}{b(i)+d(i)}, \qquad\text{for }i\in\N_0. 
	\ee 
	Finally, define 
	\be \label{eq:D}
	S_k:=\sum_{i=0}^{k-1} E_i, \qquad  \ D:=\inf\{i\in\N_0: B_i=0\}, \qquad \text{and}\qquad \cR:=\sum_{i=1}^D\delta_{S_i}, 
	\ee 
	where $\delta$ is a Dirac measure. Note that $\cR$ is an empty point process when $D=0$. We let  $\mu$ denote the density of the point process $\cR$, i.e.\  
	\be 
	\mu(t):=\lim_{\eps \downarrow 0}\eps^{-1}\P{\cR((t,t+\eps))\geq 1}, \qquad t\geq 0. 
	\ee 
	Also, let $\widehat\mu(\lambda)$ denote the Laplace transform of $\mu$, for $\lambda\geq0$. That is, 
	\be \label{eq:rhohat}
	\widehat \mu(\lambda):=\int_0^\infty \e^{-\lambda t}\mu(t)\, \dd t=\sum_{k=1}^\infty \prod_{i=0}^{k-1}\frac{b(i)}{b(i)+d(i)+\lambda}. 
	\ee 
	The explicit expression for $\widehat\mu$ in terms of the sequences $b$ and $d$ follows from~\cite[Proposition $1.1$]{Dei10}. A shorter and new proof of this fact is provided in Section~\ref{prop:rhoexpl}. Then, we define 
	\be \label{eq:undlambda}
	\underline \lambda:=\inf\{\lambda>0: \widehat \mu(\lambda)<\infty\}. 
	\ee 
	Furthermore, we define the following sequences. For $k\in\N$,
	\be \ba \label{eq:seqs}
	\varphi_1(k)&:= \sum_{i=0}^{k-1}\frac{1}{b(i)+d(i)}, \qquad & \varphi_2(k):=\sum_{i=0}^{k-1}\Big(\frac{1}{b(i)+d(i)}\Big)^2,& \\
	\rho_1(k)&:=\sum_{i=0}^{k-1}\frac{d(i)}{b(i)+d(i)},\qquad & \rho_2(k):=\sum_{i=0}^{k-1}\Big(\frac{d(i)}{b(i)+d(i)}\Big)^2.&
	\ea\ee 
	We extend the domain of $\varphi_1,\varphi_2,\rho_1$, and $\rho_2$ to $\R_+$ by linear interpolation so that, for example, $\varphi_1(t)=\int_0^t1/(b(\lfloor x\rfloor)+d(\lfloor x\rfloor))\,\dd x$. In particular, this implies that $\varphi_1$ and $\varphi_2$ are strictly increasing and thus invertible. In the case that $d(i)$ converges to some limit $d^*\in[0,\infty)$, we define the sequence $\alpha$ as
	\be \label{eq:alpha}
	\alpha(k):=\rho_1(k)-d^*\varphi_1(k)=\sum_{i=0}^{k-1}\frac{d(i)-d^*}{b(i)+d(i)}, \qquad k\in \N. 
	\ee 
	Again, extend the domain of $\alpha$ to $\R_+$ by linear interpolation. Note that $\alpha\equiv \rho_1$ when $d^*=0$, that is, when $d$ converges to zero. In the case that $\lim_{k\to\infty}\rho_1(k)$ exists (irrespective of whether $d$ converges to $d^*=0$ or not), we also define $\alpha(k):= \rho_1(k)$.  We then define 
	\be \label{eq:Ks}
	\cK(t):=\varphi_2\big(\varphi_1^{-1}(t)\big), \qquad \text{and}\qquad \ \cK_\alpha(t):=\alpha\big(\varphi_1^{-1}(t)\big), \qquad \text{for }t\geq 0. 
	\ee  
	Note that $\E{S_k}=\varphi_1(k)$ and $\Var(S_k)=\varphi_2(k)$, where $S_k$ is as in~\eqref{eq:D}. Heuristically, $\cK(t)$ denotes the variance of the sum $S_k$, where $k=\varphi_1^{-1}(t)$ is such that $\E{S_k}=t$.  Similarly, $\cK_\alpha(t)$ quantifies the difference between $\rho_1(\varphi_1^{-1}(t))$ and $d^*t$, and allows us to precisely quantify probabilities of the form $\P{D\geq \varphi_1^{-1}(t)}$ (see Lemma~\ref{lemma:Dtail}). 
	
	Before we present our main results, we state the following assumptions.
	
	\subsection{Assumptions} 
	
	Recall the Laplace transform $\widehat \mu$ of the density $\mu$ of the point process $\cR$, as in~\eqref{eq:rhohat}, and recall $\underline \lambda$ from~\eqref{eq:undlambda}. We assume that 
	\be \tag{Ma}\label{ass:C1}
	\widehat \mu(\lambda^*)=1\text{ has a solution }\lambda^*\in(0,\infty)\qquad\text{and}\qquad\  
	\lambda^* >\underline \lambda.
	\ee
	Note that this implicitly implies that $\underline \lambda<\infty$. The solution $\lambda^*$ to the equation $\widehat\mu(\lambda^*)=1$ is known as the \emph{Malthusian parameter}. 
	
	Then, we have the assumptions
	\begin{align} 
		\lim_{k\to\infty}\varphi_1(k)&=\sum_{i=0}^\infty \frac{1}{b(i)+d(i)}=\infty,\tag{N-E}\label{ass:A1} \\
		\lim_{k\to\infty}\varphi_2(k)&=\sum_{i=0}^\infty \frac{1}{(b(i)+d(i))^2}=\infty, \tag{D-V}\label{ass:varphi2}\\
		\lim_{k\to\infty}\rho_1(k)&=\sum_{i=0}^\infty \frac{d(i)}{b(i)+d(i)}=\infty.\tag{F-D}\label{ass:A2}
	\end{align}
	Throughout the paper we assume that Assumption~\eqref{ass:A1} ``Non-Explosion" holds. In fact, Assumption~\eqref{ass:C1} combined with the continuity of the mapping $\lambda\mapsto \widehat\mu(\lambda)$ implies Assumption~\eqref{ass:A1}. It guarantees that the tree process is `not exploding', in the sense that there does not exists a unique vertex that obtains an infinite degree, nor a unique path from the root of infinite length. Assumption~\eqref{ass:varphi2} ``Diverging-Variance" implies that the variance of $S_k$ tends to infinity with $k$, and is crucial for some of the main results, similar to the summability condition in~\eqref{eq:sumcond}. Finally, Assumption~\eqref{ass:A2} ``Finite-Degree" implies that the random variable $D$, as in~\eqref{eq:D}, is finite almost surely (see Lemma~\ref{lemma:Dlifedistr}), and we may or may not assume that~\eqref{ass:A2} is satisfied. 
	
	Furthermore, we have the following technical assumptions for the functions $\cK$ and $\cK_\alpha$, as in~\eqref{eq:Ks}.
	
	\begin{assK}\label{ass:K}
		The function $\cK$ satisfies
		\be\label{ass:K1}
		\lim_{\eps\downarrow 0}\limsup_{t\to\infty} \frac{\cK((1+\eps)t)}{\cK(t)}=1.
		\ee 
	\end{assK}

	\begin{assKalpha}\label{ass:Kalpha}
		Suppose Assumption~\eqref{ass:varphi2} is satisfied. For $\lambda^*>0,d^*\geq 0$, there exists a function $r\colon(0,\infty)\to\R_+$ such that
		\be\label{ass:Kalphas}
		\cK_\alpha\big(\tfrac{\lambda^*}{\lambda^*+d^*}t-\tfrac{1}{\lambda^*+d^*}\cK_\alpha(r(t))\big)-\cK_\alpha(r(t))=o(\cK(r(t))).
		\ee 
	\end{assKalpha}
	
	\begin{remark} 
		Assumption~\hyperref[ass:Kalpha]{$\cK_\alpha$}  can be interpreted as an assumption for $\rho_1\circ \varphi_1^{-1}$ in the case that $d$ converges to zero (as $\alpha\equiv \rho_1$ in that case). When $\rho_1$ converges, so that $\alpha\equiv \rho_1$, it follows that $\cK_\alpha$ converges, so that Assumption~\hyperref[ass:Kalpha]{$\cK_\alpha$} is trivially satisfied.\ensymboldremark
	\end{remark}
	
	\begin{remark}
		Assumption~\hyperref[ass:Kalpha]{$\cK_\alpha$} is satisfied when $r(t)=\frac{\lambda^*}{\lambda^*+d^*}t-x(t)$, where $x(t)$ satisfies $x(t)=\frac{1}{\lambda^*+d^*}\cK_\alpha(\frac{\lambda^*}{\lambda^*+d^*}t-x(t))+\cO(\cK(\tfrac{\lambda^*}{\lambda^*+d^*}t))$.\ensymboldremark
	\end{remark}
	
	\subsection{Main results} \label{sec:results}
	
	We present our results in two parts, which reflect two distinct regimes in which behaviour of the PAVD model is rather different. We denote these regimes as the \emph{rich are old} and the \emph{rich die young} regimes. In short, lack of persistence occurs in the \emph{rich are old} regime when Assumption~\eqref{ass:varphi2} is satisfied, i.e.\ when the function $\varphi_2$ tends to infinity. In the \emph{rich die young} regime, however, lack of persistence, under mild assumptions on the sequences $b$ and $d$, \emph{always} occurs. These two regimes arise due to different \emph{survival strategies} of vertices. In the `rich are old' regime, vertices are able to live for a long time by increasing their degree, so that old alive vertices tend to have large degrees. In the `rich die young' regime however, as the name suggests, vertices with large degrees are more likely to be killed and thus do not live long, whereas vertices with small degree can manage to survive for a much longer time, which causes lack of persistence to occur. 
	
	To distinguish these regimes, let us define the quantity 
	\be \label{eq:R}
	R:=\inf_{i\in\N_0}(b(i)+d(i)). 
	\ee
	
	\textbf{The rich are old. } In the `rich are old' regime, we assume that either Assumption~\eqref{ass:A2} is not satisfied, or that there exists $d^*\in[0,R)$, such that
	\be \label{eq:dconv}
	\lim_{i\to\infty} d(i)=d^*. 
	\ee 
	Here, it is crucial that the limit $d^*$ is strictly smaller than $R$. Our first result concerns choices of birth and death rates such that the model behaves almost as preferential attachment without death (i.e.\ $d\equiv 0$). That is, it considers a family of models of which preferential attachment without death is a special case, where Assumption~\eqref{ass:A2} is not satisfied. Here, a positive proportion of vertices is never killed, and it are these vertices that determine the behaviour of $O_n$ and $I_n$. 
	
	\begin{theorem}\label{thrm:asPA}
		Consider the PAVD model in Definition~\ref{def:pavd}. Suppose that $b$ and $d$ are such that Assumption~\eqref{ass:A1} is satisfied, but Assumption~\eqref{ass:A2} is not. Then, there exists an almost surely finite random variable $O$ such that 
		\be \label{eq:Otasconv}
		O_n\longrightarrow O\qquad \mathbb P_\cS\mathrm{-a.s.}
		\ee 
		Additionally, suppose that $b$ tends to infinity and suppose that Assumption~\eqref{ass:C1} and Assumption~\hyperref[ass:K]{$\cK$} are satisfied. When, moreover, Assumption~\eqref{ass:varphi2} is satisfied,
		\be \label{eq:Inplusmax}
		\frac{\log I_n}{\cK\big(\frac{1}{\lambda^*}\log n\big)} \toinps \frac{(\lambda^*)^2 }{2}, \qquad \text{and} \qquad \ \frac{1}{\cK\big(\frac{1}{\lambda^*}\log n\big)}\Big(\varphi_1(\max_{v\in \cA_n}\deg_n(v))-\frac{1}{\lambda^*}\log n \Big)\toinps \frac{\lambda^*}{2}.
		\ee
		In particular, persistence does not occur in the sense of~\eqref{eq:nopers}.
	\end{theorem} 
	
	\begin{remark}
		Theorem~\ref{thrm:asPA} \emph{generalises} Theorem $4.12$ in~\cite{BanBha21}, which concerns itself with lack of persistence in the case $d\equiv 0$, i.e.\ preferential attachment trees without death. Furthermore, we also prove this more general results \emph{without} the need of Assumption C$3$ in~\cite{BanBha21}. That is, we do not need to assume that there exist constants $t',D>0$ such that $\cK(3t)\leq D\cK(t)$ for all $t\geq t'$.\ensymboldremark
	\end{remark}
	
	The second result in the `rich are old' regime considers the case that $d$ converges to $d^*\in[0,R)$ such that Assumption~\eqref{ass:A2} is satisfied. This implies that any vertex is eventually killed almost surely, which causes $O_n$ to grow with $n$ and yields a different first-order term for $I_n$.
	
	\begin{theorem}[Converging death sequences, smaller than $R$]\label{thrm:conv}
		Consider the PAVD model, as in Definition~\ref{def:pavd}.  Suppose that $b$ and $d$ are such that Assumptions~\eqref{ass:A1} and \eqref{ass:C1}  are satisfied. Recall $R$ from~\eqref{eq:R} and suppose that $d$ satisfies~\eqref{eq:dconv} with $d^*\in[0,R)$, and that $b$ tends to infinity. Let Assumptions~\hyperref[ass:K]{$\cK$} and~\hyperref[ass:Kalpha]{$\cK_\alpha$} be satisfied. When Assumptions~\eqref{ass:A2} and~\eqref{ass:varphi2} are satisfied, then
		\begin{align}
			\frac{1}{\cK\big(\tfrac{1}{\lambda^*+d^*}\log n\big)}\Big(\log O_n-\frac{d^*}{\lambda^*+d^*}\log n-\frac{\lambda^*}{\lambda^*+d^*}\cK_\alpha\big(r\big(\tfrac{1}{\lambda^*}\log n\big)\big)\Big)&\toinps 0,\label{eq:Onconv}
			\intertext{and}
			\frac{1}{\cK\big(\tfrac{1}{\lambda^*+d^*}\log n\big)}\Big(\log I_n-\frac{d^*}{\lambda^*+d^*}\log  n-\frac{\lambda^*}{\lambda^*+d^*}\cK_\alpha \big(r\big(\tfrac{1}{\lambda^*}\log n\big)\big)\Big)&\toinps \frac{\lambda^*(\lambda^*+d^*)}{2}, \label{eq:Inconv}
			\intertext{and}
			\frac{\varphi_1(\max_{v\in \cA_n}\deg_n(v))-\frac{1}{\lambda^*+d^*}\log n+\frac{1}{\lambda^*+d^*}\cK_\alpha\big(r\big(\frac{1}{\lambda^*}\log n\big)\big)}{\cK\big(\tfrac{1}{\lambda^*+d^*}\log n\big)}&\toinps \frac{\lambda^*-d^*}{2}.\label{eq:maxdegconv}
		\end{align}
		In particular, persistence does not occur in the sense of~\eqref{eq:nopers}.
	\end{theorem}
	
	\begin{remark}
		Though we present Theorems~\ref{thrm:asPA} through~\ref{thrm:conv} as separate cases to aid the reader and highlight different behaviour, they can be viewed as \emph{one general result}. Indeed, the results in~\eqref{eq:Inplusmax}, \eqref{eq:Onconv}, \eqref{eq:Inconv}, and~\eqref{eq:maxdegconv} can be proved with the same methodology with only minor differences in the proof to distinguish between whether Assumption~\eqref{ass:A2} is satisfied (as in Theorem~\ref{thrm:conv}) or not (as in Theorem~\ref{thrm:asPA}). Furthermore, the results in~\eqref{eq:Inconv} and~\eqref{eq:maxdegconv} hold for \emph{any} $d^*\geq 0$. The restriction $d^*<R$ is necessary only for~\eqref{eq:Onconv}.\ensymboldremark
	\end{remark}
	
	\begin{remark}\label{rem:d0ass}
		When $\lim_{i\to\infty}d(i)=0$ such that $\alpha=\rho_1=o(\varphi_2)$, the condition that Assumption~\hyperref[ass:Kalpha]{$\cK_\alpha$} is satisfied can be omitted.\ensymboldremark
	\end{remark}
	\begin{remark}
		The result in~\eqref{eq:Onconv} yields that the third order asymptotic behaviour of $\log O_n$ is $o(\cK(\tfrac{1}{\lambda^*+d^*}\log n))$. It could be possible to obtain a more precise result, but this would require additional assumptions on the functions $\cK$ and $\cK_\alpha$, and is not necessary for our purposes here.\ensymboldremark
	\end{remark}
	
	\begin{remark}\label{rem:incldead}
		The results presented in~\eqref{eq:Inplusmax}, \eqref{eq:Inconv}, and~\eqref{eq:maxdegconv} all remain true if we consider the label of the vertex with the largest degree among \emph{all} vertices (i.e.\ both alive and dead). This can be verified by modifying parts of the proofs in later sections. It shows that vertices can obtain a large degree only if they also manage to stay alive for a long time, and do not get lucky due to another vertex with a possibly even larger degree dying.\ensymboldremark
	\end{remark} 
	
	\textbf{Rich die young.} We then consider the `rich die young' regime. Here, we can prove a general result regarding lack of persistence with only minor assumptions on the sequences $b$ and $d$. However, we do not provide precise asymptotic results for $I_n$, as more assumptions on the sequence $d$ (as in~\eqref{eq:dconv}, for example) would be required.  

	Before stating the result, we introduce the sequence $(\overline d(i))_{i\in\N_0}$, defined as 
	\be 
	\overline d(i):=\sup_{j\leq i}d(i)\qquad\text{for }i\in\N_0. 
	\ee 
	\begin{theorem}[The rich die young]\label{thrm:biggerR}
		Consider the PAVD model, as in Definition~\ref{def:pavd}.  Suppose that $b$ and $d$ are such that Assumptions~\eqref{ass:A1} and \eqref{ass:C1} are satisfied. Recall $R$ from~\eqref{eq:R}, and suppose that $d$ satisfies $\liminf_{i\to\infty}d(i)\geq R$. Then, 
		\be \label{eq:OnR}
		\frac{\log O_n}{\log n}\toinps \frac{R}{\lambda^*+R}. 
		\ee 
		Additionally, suppose $\liminf_{i\to\infty}d(i)>R$,  that $b$ tends to infinity so that $b(k)=\cO(k)$ and $\overline d(k)=o(b(k))$, and that $b(k)$ and $\overline d(k)$ are both regularly varying with a non-negative exponent. Then, conditionally on the event $\cS$, persistence does not occur in the sense of~\eqref{eq:nopers}.
	\end{theorem}
	
	We stress that Assumption~\eqref{ass:varphi2} is stated in Theorems~\ref{thrm:asPA} and~\ref{thrm:conv}, whilst it is not included in Theorem~\ref{thrm:biggerR}. Indeed, work on persistence of the maximum degree in preferential attachment models without death (see~\cite{DerMor09,Gal13,BanBha21,Iyer24}) shows that Assumption~\eqref{ass:varphi2} is necessary and sufficient for lack of persistence in a large family of models. Here, we prove that it is a sufficient condition for a large family of models in the `rich are old' regime. We intend to address the necessary part, as well as generalise Theorem~\ref{thrm:asPA}, in future work.  
	
	In the `rich die young' regime, however, lack of persistence occurs regardless of whether Assumption~\eqref{ass:varphi2} is satisfied or not. This is essentially due to the following reason: When we have $\liminf_{i\to\infty}d(i)>R$, the optimal survival strategy for a vertex is as follows. Since there exist a finite number of indices $I_1,I_2,\ldots, I_j\in\N$ such that $b(I_\ell)+d(I_\ell)=\inf_{i\in\N_0}(b(i)+d(i))=R$ for all $\ell\in[j]$, a vertex that survives for a long time creates $I_\ell$ connections (for some $\ell\in[j]$) and is then no longer selected up to step $n$. This strategy is how the oldest individual in $T_n$ with label $n^{R/(\lambda^*+R)+o(1)}$ survives up to step $n$. Individuals that produce a large number of children, say $M$ many, are selected and then killed with probability proportional to $d(M)$, which is larger than $R$ (if $M$ is large) since $\liminf_{i\to\infty}d(i)>R$. Vertices with high degree are hence less likely to survive up to step $n$ and must therefore have been introduced much later than the oldest alive vertex.

	\subsection{Special case: converging birth and death functions} 
	
	To conclude this section, we consider the special case of \emph{converging} birth and death sequences as an interesting side-case. This case is not covered by any of the results presented so far. In particular, the case that $b$ and $d$ are constant represents a \emph{uniform} attachment tree with uniform vertex death, as studied in~\cite{Dei10,BelBlaKamKor23,BelBlaKamKor23II}.  
	
	\begin{theorem}[Converging birth and death sequences]\label{thrm:const}
		Consider the PAVD model, as in Definition~\ref{def:pavd}. Suppose that  $\lim_{i\to\infty}d(i)=1$ and $\lim_{i\to\infty}b(i)=c>1$, such that Assumption~\eqref{ass:C1} is satisfied and recall $R$ from~\eqref{eq:R}. Then,
		\be
		\frac{\log O_n}{\log n}\toinps \frac{\min\{1,R\}}{\lambda^*+\min\{1,R\}}, \quad  \frac{\log I_n}{\log n}\toinps 1-\frac{\lambda^*}{2(\lambda^*+1)\log2}, \quad\text{and}\quad  \frac{\max_{v\in \cA_n}\deg_n(v)}{\log n}\toinps \frac{1}{\log 2}. 
		\ee 
		In particular, as $1- \frac{\lambda^*}{2(\lambda^*+1)\log 2}>\frac{1}{\lambda^*+1}$ for any $\lambda^*>0$, there is no persistence in the sense of~\eqref{eq:nopers}.
	\end{theorem}  
	
	\begin{remark}[Constant birth and death rates]\label{rem:const}
		In the case that $b\equiv c$ and $d\equiv 1$, the above result holds with $\lambda^*=c-1$ and $\min\{1,R\}=\min\{1,1+c\}=1$. Letting $c$ tend to infinity informally yields the results for the random recursive tree (for which $O_n\equiv 1$, $\log(I_n)/\log n\toas 1-1/(2\log 2)$, and $\max_{v\in\cA_n}\deg_n(v)/\log n\toas 1/\log 2$, see~\cite[Theorem $4.14$]{BanBha21} and~\cite[Theorem $2.3$]{Lod24}), in which no death occurs and connections are made uniformly at random. Indeed, for large values of $c$, new vertices are introduced much more often than that vertices are killed, so that this approximates the case of no death and uniform connections.\ensymboldremark
	\end{remark}
	
	We do not formally prove Theorem~\ref{thrm:const}, but rather in Section~\ref{sec:const} provide an overview of the alterations required to the proofs of Theorems~\ref{thrm:conv} and~\ref{thrm:biggerR} to yield the desired results.

	\textbf{Examples. } 
	
	Whether the PAVD model is in the `rich are old' or `rich die young' regime subtly depends on the sequences $d$ and $b$. By changing only a small number of values in either sequences, it is  possible to go from one regime to the other. For example, consider the three models
	\be \ba \label{eq:ex}
	({}&\text{RaO}) & b&=(1,2,3,\ldots)&&\text{and } d=(1,2,\tfrac32,\tfrac32,\ldots),\\ 	({}&\text{RdY}1)& b&=(1,2,3,\ldots)&&\text{and } d=(\tfrac14,2,\tfrac32,\tfrac32,\ldots),\\
	({}&\text{RdY}2) & b&=(\tfrac14,2,3,\ldots)&&\text{and } d=(1,2,\tfrac32,\tfrac32,\ldots).
	\ea\ee 
	In the first example we have $R=2$ and $d^*=\frac32$, so that this case belongs to the `rich are old' regime. In the second and third line we have changed one instance of $d$ and $b$, respectively, so that $R=\frac54$ and $d^*=\frac32$. These cases thus belong to the `rich die young' regime. Whilst this change is rather subtle (for example, the limiting degree distribution remains the same, see Section~\ref{sec:embed}), the consequences for persistence are more drastic. It becomes much harder (resp.\ easier) for vertices to survive for a long time, in particular if they acquire a large degree, if one moves from the `rich are old' to the `rich die young' regime and vice versa.
	
	Assumption~\hyperref[ass:K]{$\cK$} is satisfied for any $b$ and $d$ such that $b+d=b_f\cdot b_r$, where $b_f$ is bounded away from zero and infinity, and $b_r$ is regularly varying with exponent in $[0,1)$, or $b_r$ is affine; see~\cite[Appendix A]{BanBha21}.
	
	Assumption~\hyperref[ass:Kalpha]{$\cK_\alpha$} is trivially satisfied when $K_\alpha=\alpha\circ \varphi_1^{-1}$ converges, with $r(t)=\frac{\lambda^*}{\lambda^*+d^*}t$. That is, when $b(i)^{-1}(d(i)-d^*)$ is summable in $i$. When $\cK_\alpha$ does not converge, then Assumption~\hyperref[ass:Kalpha]{$\cK_\alpha$} is satisfied for $b(i)=(i+1)^\beta$ with $\beta\in(0,1/2)$ and $d(i)=d^*+i^{-\gamma}$ with $\gamma\in(0,1-\beta]$. 
	
	Finally, Assumption~\eqref{ass:C1} is satisfied when 
	\begin{itemize} 
		\item $\E{D}\in(1,\infty)$, or
		\item $\E{D}=\infty$, $b$ tends to infinity,   $C_1:=\limsup_{i\to\infty}\frac{b(i)}{i}<\infty$, and $\widehat \mu(C_1)>1$, or 
		\item  $\E{D}=\infty$, $b$ tends to infinity,   $\limsup_{i\to\infty}\frac{b(i)}{i}=\infty$, $C_2:=\limsup_{i\to\infty}\frac{b(i)}{d(i)i}<1$, and $\widehat \mu(C_2)>1$. 
	\end{itemize}

	\subsection{Related and future work} 
	
	The results presented in this section demonstrate a distinction between the `rich are old' and `rich die young' regimes. Whilst the `rich are old' regime already received a significant amount of attention for preferential attachment models without death, the `rich die young' regime can only occur when introducing vertex death. However, the distinction between these regimes is rather subtle, as minor changes in the birth or death sequence can imply a switch from one regime to the other (see the examples in~\eqref{eq:ex}), whereas these regimes are not observed in the analysis of e.g.\ the degree distribution, as studied in the work of Deijfen~\cite{Dei10}. 
	
	The tools we use to prove the main results combine a more in-depth understanding of the PAVD model and a continuous-time embedding of the process into a Crump-Mode-Jagers branching process, as used for preferential attachment trees without death by Banerjee and Bhamidi in~\cite{BanBha21}. In particular, we study the lifetime and degree distribution of the root vertex and derive moderate deviation bounds for the event that a vertex survives for a long time and obtains a large degree.
	
	\emph{Persistence. } This paper investigates lack of persistence of the maximum degree in PAVD models. A logical continuation is to also determine when persistence can occur. We aim to address this in future research. The results presented here are obtain by studying the quantities $O_n$ and $I_n$ separately. We derive precise asymptotic expressions for these quantities and use that $I_n$ is much larger than $O_n$ to conclude there is lack of persistence. To prove persistence, however, this would require very tight control of both $O_n$ and $I_n$, which may only work under constraining assumptions. A novel perspective of this problem, recently developed by Iyer in~\cite{Iyer24} for preferential attachment trees without death, could perhaps be generalised to the setting with death to study the occurrence of persistence. 
	
	\section{Methodology: Embedding in CMJ branching process}\label{sec:embed}
	
	To prove the results presented in Section~\ref{sec:results}, we make use of a technique where one `embeds' a random discrete process into a continuous-time branching process, a Crump-Mode-Jagers (CMJ) branching process to be exact. The  random discrete structure is then equal in distribution to the CMJ branching process, when viewed at certain stopping times, whilst the CMJ branching process in continuous time provides more analytical advantages compared to the discrete process.
	
	\textbf{Ulam-Harris tree. } To define the relevant CMJ branching process, we first introduce some definitions and notation. We let $\cU_\infty$ denote the Ulam-Harris tree. That is,
	\be\label{eq:UH}
	\cU_\infty:=\{\varnothing\}\cup\bigcup_{k=1}^\infty \N^k.
	\ee 
	For $v=v_1v_2\cdots v_k\in \cU_\infty$, we view $v$ as the $v_k^{\th}$ child of $v_1\cdots v_{k-1}$ (where $v_1\cdots v_{k-1}$ denotes $\varnothing$ when $k=1$). Similarly, we view $j\in \N$ as the $j^{\th}$ child of $\varnothing$. We assign to each individual $v\in\cU_\infty$  an i.i.d.\ copy $\cR^{(v)}$ of $\cR$, as in~\eqref{eq:D} (constructed by i.i.d.\ copies $S_k^{(v)}$, $( E_i^{(v)})_{i\in\N_0}$, and $D^{(v)}$ of $S_k$, $(E_i)_{i\in\N_0}$, and $D$, respectively). Furthermore, we set 
	\be \label{eq:lifetime}
	L^{(v)}:=S^{(v)}_{D^{(v)}+1}, \qquad\qquad v\in \cU_\infty, 
	\ee 
	as the \emph{lifetime} of the individual $v$. We then define the branching process $(\bp(t))_{t\geq 0}$ as follows. At time $t=0$, $\bp(0)$ consists of a single alive individual $\varnothing$. The individual $\varnothing$ gives birth to children according to $\cR^{(\varnothing)}$. That is, the position of the $j^{\th}$ point in $\cR^{(\varnothing)}$ denotes the birth-time of child $j$ of $\varnothing$. Then, $\varnothing$ dies after $L^{(\varnothing)}$ time. Each such child $v$, once born, produces children according to $\cR^{(v)}$ translated by the birth-time of $v$ and lives for $L^{(v)}$ amount of time, after which it dies, independently of all other individuals. This process continues either forever, in which case we say that the branching process \emph{survives}, or until all individuals have died, in which case we say that the branching process has \emph{died}.
	
	If we let $\sigma_v$ denote the birth-time of the individual $v\in \cU_\infty$, then $\sigma_\varnothing:=0$  and, for any $v=v_1\cdots v_k\in \cU_\infty$ such that $D^{(v_1\cdots v_{i-1})}\geq v_i$ for all $i=1,\ldots,k$,  
	\be 
	\sigma_v:=\sum_{j=1}^kS_{v_j}^{(v_1\cdots v_{j-1})}=\sum_{j=1}^k\sum_{i=0}^{v_j-1}E_i^{(v_1\cdots v_{j-1})}.
	\ee 
	For other individuals $v\in \cU_\infty$, we set $\sigma_v=\infty$, as these individuals are never born. We define 
	\be \label{eq:At}
	\cA^\cont_t:=\{u\in \cU_\infty: \sigma_u \leq t, \sigma_u+L^{(u)}>t\}, \qquad t\geq 0, 
	\ee 
	as the set of individuals alive at time $t$. 
	
	We coin the branching process $\bp$ the continuous-time preferential attachment model with vertex death (CTPAVD).  When the sequence $d$ is the all-zero sequence, i.e.\ $d\equiv 0$, we refer to the PAVD (resp.\ CTPAVD) model as a Preferential Attachment (PA) model (resp.\ continuous-time preferential attachment (CTPA) model). To highlight that we speak of a PA model, we may also phrase this as `preferential attachment \emph{without} death'.
	
	\textbf{Characteristics. } To quantify the growth of the branching process, one can count statistics related to the process. These are known as \emph{characteristics} and track, for example, the number of births, the number of alive individuals, or the number of individual with a certain number of children in the process up to time $t$. Let $(\chi(t))_{t\in \R}$ be a real-valued stochastic process, such that $\chi(t)=0$ for $t\leq 0$. We say that $\chi(t)$ \emph{scores} an individual of age $t$, and we count the branching process $\bp$ with the random characteristic $\chi$ by setting
	\be 
	Z^\chi_t:=\sum_{u\in \cU_\infty} \chi(t-\sigma_u). 
	\ee 
	Examples of commonly used characteristics are
	\be\ba  \label{eq:char}
	(1)& \ \ \chi_{\mathrm b}(t):=\ind_{\{t\geq 0\}}, \qquad\qquad &&(2)\ \  \chi_{\mathrm a}(t):=\ind_{\{0\leq t<L\}},\\ 
	(3)& \ \ \chi_{k,\mathrm b}(t):=\ind_{\{t\geq 0,\cR(t)=k\}} &&(4)\ \ \chi_{k,\mathrm a}(t):=\ind_{\{0\leq t<L, \cR(t)=k\}}.
	\ea \ee 
	Then,  $Z^{\chi_{\mathrm b}}_t$ counts the individuals born up to time $t$, and $Z^{\chi_{\mathrm a}}_t=|\cA^\cont_t|$ counts individuals \emph{alive} at time $t$. Further,  $Z^{\chi_{k,\mathrm b}}_t$ and $Z^{\chi_{k,\mathrm a}}_t$ count the number of individuals  with $k\in\N_0$ children born up to time $t$, and the number of alive individuals with $k\in\N_0$ children born up to time $t$, respectively.  To simplify notation we write $Z^{\mathrm b}_t$, $Z^{\mathrm a}_t$, and $Z^{k,\mathrm b}_t,Z^{k,\mathrm a}_t$ to denote these quantities. We then let $(N(t))_{t\geq 0}$ be the stochastic process which counts the number of births and deaths in $\bp$, that is, 
	\be \label{eq:Nt}
	N(t):=Z^{\mathrm b}_t+(Z^{\mathrm b}_t-Z^{\mathrm a}_t), \qquad t\geq0.
	\ee 
	When Assumption~\eqref{ass:A1} is satisfied, the branching process $\bp$ does not \emph{explode}, i.e.\ $N(t)<\infty$ almost surely for all $t\geq0$. Under Assumption~\eqref{ass:A1}, Assumption~\eqref{ass:C1} is then a necessary and sufficient condition for the branching process $\bp$ to be super-critical. We abuse notation to redefine the event $\cS$ (from~\eqref{eq:surv}) as 
	\be \label{eq:S}
	\cS:=\{Z^{\mathrm a}_t>0\text{ for all }t\geq 0\}=\{\bp\text{ survives}\}.
	\ee 
	It is clear that the branching process $\bp$ is super-critical when $\cS$ holds with strictly positive probability. For super-critical CMJ branching processes with Malthusian parameter $\lambda^*>\underline \lambda $, the branching process grows at rate $\e^{\lambda^* t}$ (see~\cite{Ner81}). That is 
	\be \label{eq:maltconv}
	|\bp(t)|\e^{-\lambda^*t}\toas W,
	\ee 
	where $\P{W>0}=\P{\cS}$, i.e.\ the limit $W$ is non-trivial conditionally on survival. Furthermore, we have the following scaling limit for characteristics. 
	
	\begin{theorem}[Nerman~\cite{Ner81}]\label{thrm:nerman} Consider a super-critical CMJ branching process with Malthusian parameter $\lambda^*>\underline \lambda$ and let $\chi$ and $\psi$ be two characteristics such that 
		\be 
		\E{\sup_{t\geq 0}\e^{-\lambda t}\chi(t)}<\infty \quad\text{for some }\lambda<\lambda^*,
		\ee 
		and likewise for $\psi$. Then, conditionally  on $\cS$, 
		\be 
		\frac{Z^\chi_t}{Z^\psi_t}\toas \frac{\widehat \chi(\lambda^*)}{\widehat \psi(\lambda^*)}, 
		\ee  
		where 
		\be 
		\widehat \chi(\lambda):=\int_0^\infty \e^{-\lambda t}\E{\chi(t)}\,\dd t, 
		\ee 
		and likewise for $\psi$. 
	\end{theorem}
	
	As a direct corollary, the limits 
	\be \label{eq:charlim}
	\lim_{t\to\infty}\frac{Z^\mathrm b_t}{Z^\mathrm a_t}=\frac{\widehat \chi_\mathrm b(\lambda^*)}{\widehat \chi_\mathrm a(\lambda^*)}, \qquad \lim_{t\to\infty}\frac{Z^{k,\mathrm b}_t}{Z^\mathrm b_t}=\frac{\widehat \chi_{k,\mathrm b}(\lambda^*)}{\widehat \chi_\mathrm b(\lambda^*)}=:p_k^\mathrm b, \quad\text{and}\quad \ \lim_{t\to\infty}\frac{Z^{k,\mathrm a}_t}{Z^\mathrm a_t}=\frac{\widehat \chi_{k,\mathrm a}(\lambda^*)}{\widehat \chi_\mathrm a(\lambda^*)}=:p_k^\mathrm a,
	\ee 
	all exist conditionally on $\cS$ almost surely. Here, $p_k^\mathrm a$ and $p_k^\mathrm b$ are called the \emph{limiting offspring distribution} of \emph{alive} and \emph{all} individuals, respectively, and represent the limiting proportion of alive (resp.\ all) individuals that have $k$ children in $\bp$. When, $d\equiv 0$, then $p^\mathrm a_k$ and $p^\mathrm b_k$ are equal.
	
	\textbf{Continuous-time embedding. } With $N(t)$ as in~\eqref{eq:Nt}, we   define the stopping times $(\tau_n)_{n\in\N}$ as
	\be\label{eq:taun} 
	\tau_n:=\inf\{t\geq 0: N(t)=n\}, \qquad n\in\N. 
	\ee 
	That is, $\tau_n$ denotes the time at which the $n^{\th}$ birth or death event occurs. We then have the following correspondence between the discrete tree process $(T_n, \cA_n)_{n\in\N}$ and the CMJ branching process $(\bp(\tau_n), \cA_{\tau_n}^{\cont})_{n\in\N}$ viewed at the sequence of stopping times.
	
	\begin{proposition}[Embedding of PAVD in CTPAVD]\label{prop:embed}
		Let $(T_n,\cA_n)_{n\in\N}$ be the sequence of trees and sets of alive vertices of a PAVD model and $(\bp(t),\cA_t^\cont)_{t\geq 0}$ be a CTPAVD model and the set of alive individuals, both constructed with the same birth and death sequences $b$ and $d$, respectively. Then, 
		\be 
		(T_n,\cA_n)_{n\in\N}\overset d= (\bp(\tau_n),\cA^\cont_{\tau_n})_{n\in\N}. 
		\ee  
	\end{proposition}
	
	The proposition follows from the memoryless property of exponential random variables and properties of minima of exponential random variables. The technique of embedding evolving discrete structures was originally pioneered by Athreya and Karlin~\cite{ArthKar68} and has found fruitful applications in a wide variety of discrete (combinatorial) models, such as P\'olya urns, discrete randomly growing trees such as preferential attachment trees and uniform attachment trees, evolving simplicial complexes, and many more. 
	
	The description of the branching process $\bp$ in which we embed the PAVD tree is used by Deijfen~\cite{Dei10} to determine an explicit description for the limiting offspring distribution of alive individuals $p^{\mathrm a}_k$, as defined in~\eqref{eq:charlim}. The description of the branching process used here is somewhat different (though equivalent), but provides an analytical advantage. In particular, defining the offspring and lifetime of an individual as an i.i.d.\ copy of $D$, see~\eqref{eq:D}, and $L$, defined in~\eqref{eq:lifetime}, respectively, is a new perspective which allows us to analyse these quantities to a greater extent.
	
	\textbf{Persistence in BP. } We can view the concept of persistence, and the lack thereof, from a continuous-time perspective as well. To this end, we define 
	\be \label{eq:offspring}
	\deg^{(u)}(t):=\min\{\sup\{k\in\N_0: S_k^{(u)}\leq t-\sigma_u\},D^{(u)}\},\qquad u\in \cU_\infty, t\geq 0,
	\ee 
	as the number of children $u$ has produced by time $t$. Here, we set the supremum equal to zero in case $t-\sigma_u\leq 0$.  As such, the minimum equals zero when $u$ is never born and thus never produces children, or $u$ is not born yet. With~\eqref{eq:offspring}, we then define
	\be\label{eq:contdef}
	O_t^\cont:=\!\!\min_{v\in \cA_t^\cont}\!\!\sigma_v\quad\text{and}\quad 	I_t^\cont:=\min\big\{\sigma_v:v\in \cA^\cont_t, \deg^{(v)}(t)\geq\deg^{(u)}(t)\text{ for all }u\in \cA^\cont_t\big\}.
	\ee
	Here, $O_t^\cont$ denotes the birth-time of the oldest alive individual at time $t$  and $I_t^\cont$ denotes the birth-time of the oldest alive individual that has the largest number of children at time $t$. Proposition~\ref{prop:embed} implies that 
	\be \label{eq:OtItequiv}
	\{(O_n,I_n): n\in\N\}\overset d= \{(N(O^\cont_{\tau_n}),N(I^\cont_{\tau_n})): n\in\N\}.
	\ee 
	Under Assumption~\eqref{ass:C1} it follows that $N(t)$, as defined in~\eqref{eq:Nt}, satisfies that $N(t)\e^{-\lambda^*\!t}$ converges almost surely (similar to~\eqref{eq:maltconv}), and hence that $\tau_n-\frac{1}{\lambda^*}\log n$ converges almost surely. As a result, we can, approximately, relate the quantities in~\eqref{eq:OtItequiv} via 
	\be \label{eq:corr}
	O_n\approx \exp\big( \lambda^* O^\cont_{\log(n)/\lambda^*}\big), \qquad\text{and},\qquad \ I_n\approx \exp\big( \lambda^* I^\cont_{\log(n)/\lambda^*}\big).
	\ee 
	This relation is made precise in Section~\ref{sec:mainproofs}. In particular, this indicates that persistence and the lack of persistence, in the sense of~\eqref{eq:pers} respectively~\eqref{eq:nopers} are implied when 
	\be\label{eq:persnoperscont} 
	(I^\cont_t-O^\cont_t)_{t\geq 0} \text{ is a tight sequence of r.v.'s, and }I^\cont_t-O^\cont_t\toinps \infty, \text{ respectively.}
	\ee 
	In the sections that follow, we study the CMJ branching process $\bp$, and the quantities $O^\cont_t$ and $I^\cont_t$. To prove the main results, we make the correspondence in~\eqref{eq:corr} precise, to translate results for the CMJ branching process back to the discrete process $(T_n,\cA_n)_{n\in\N}$.
	
	\textbf{Laplace transform of the density of the offspring point process. } We conclude this section by providing a shorter and simplified proof of a result of Deijfen~\cite[Proposition $1.1$]{Dei10} regarding the explicit expression of the Laplace transform $\widehat\mu$ in~\eqref{eq:rhohat} that leverages the description of the point process $\cR$, as in~\eqref{eq:D}, in terms of $(S_i)_{i\in\N}$ and $D$.
	
	\begin{proposition}[Proposition $1.1$~\cite{Dei10}]\label{prop:rhoexpl}
		For any $\lambda\geq 0$, 
		\be 
		\widehat\mu(\lambda)=\sum_{k=1}^\infty \prod_{i=0}^{k-1}\frac{b(i)}{\lambda+b(i)+d(i)}. 
		\ee 
	\end{proposition}
	
	\begin{proof}
		First, we observe that $\mu(t)\,\dd t =\E{ \cR(\dd t)}$. As a result, by the definition of $\widehat \mu$ in~\eqref{eq:rhohat} and using Fubini's theorem, 
		\be 
		\widehat\mu(\lambda) =\int_0^\infty \e^{-\lambda t} \E{\cR(\dd t)}=\E{\int_0^\infty \e^{-\lambda t} \cR(\dd t)}=\E{\sum_{k=1}^D \e^{-\lambda S_k}}=\E{\sum_{k=1}^\infty \ind_{\{D\geq k\}}\e^{-\lambda S_k}}. 
		\ee 
		By taking the summation out of the expected value, using that $D$ is independent of the $(S_k)_{k\in\N}$, and that $S_k$ is a sum of independent exponentials, we arrive at 
		\be 
		\widehat\mu(\lambda)=\sum_{k=1}^\infty \P{D\geq k}\prod_{i=0}^{k-1}\frac{b(i)+d(i)}{\lambda+b(i)+d(i)}=\sum_{k=1}^\infty \prod_{i=0}^{k-1}\frac{b(i)}{\lambda+b(i)+d(i)},
		\ee 
		which concludes the proof.
	\end{proof}
	
	\section{Heuristic explanation of the main results}\label{sec:heur}
	
	In this section we provide a short and heuristic explanation of the main results in Section~\ref{sec:results}. We refer to intermediate results proved in later sections to aid the reader in finding the rigorous proofs of claims and statements made here. 
	
	\textbf{The rich are old vs.\ the rich die young. } Recall the random variable $L:=S_{D+1}$ from~\eqref{eq:lifetime} as the lifetime of an individual in the CMJ branching process BP. We can interpret $L$ as follows: An individual, say $v$, has two independent exponential clocks, a `birth' and `death' clock. When $v$ has $i\in\N_0$ many children, the birth and death clock ring at rate $b(i)$ and $d(i)$, respectively. If the birth clock rings before the death clock, $v$ gives birth to a child and the rates are updated. If the death clock rings before the birth clock, $v$ dies.  When $v$ has $i$ children, the first clock to ring, rings at rate $b(i)+d(i)$ (as it is a minimum of two exponential times). Furthermore, the first index $i$ such that the death clock rings before the birth clock is equal in distribution to $D$, as in~\eqref{eq:D}, which is independent of the time the clocks need to ring.
	
	Now, suppose that $d\equiv d^*$ for some $d^*>0$ and that $v$ has no children. If the death clock rings first, $v$ has died after an exponential time with rate $d^*$. If the birth clock rings first, we resample the birth and death clocks with adjusted rates. However, as $d\equiv d^*$, resampling the death clock is equal in distribution to letting it continue to run, by the memoryless property. We can thus repeat the above case distinction, so that $v$ dies after a rate $d^*$ exponential time. See Lemma~\ref{lemma:stochdomexp}. In short, 
	\be \label{eq:lifetimeconst}
	d\equiv d^* \ \Longrightarrow \ L\sim \text{Exp}(d^*).
	\ee  
	Now, suppose that $d$ is such that $d(i)$ converges to $d^*$ as $i$ tends to infinity. Though the above argument no longer exactly holds, one could imagine that the lifetime should, asymptotically, be exponentially distributed with rate $d^*$. This is not quite correct, however, as there are two strategies for an individual to survive for a long time:
	\begin{enumerate}
		\item The birth clock rings many times before the death clock does. As the rate of the death clock gets closer to $d^*$, an individual approximately lives for an exponentially distributed time with rate $d^*$. 
		\item The individual gives birth to $I$ children, for some $I\in\N_0$, after which neither the birth clock nor death clock ring for a long time. An individual thus approximately lives for an exponentially distributed time with rate $b(I)+d(I)$.
	\end{enumerate} 
	Now, recall that $R:=\inf_{i\in\N_0}(b(i)+d(i))$. If $d^*< R$, strategy $(a)$ yields the best way of surviving for a long time. Indeed, as $b(I)+d(I)\geq R> d^*$ for any $I\in\N_0$, strategy $(b)$ is significantly less likely to occur. On the other hand, if $d^*>R$ there exists an index $I$ such that $b(I)+d(I)=R<d^*$, so that strategy $(b)$ becomes the more likely scenario and an individual approximately lives for an exponentially distributed time with rate $R$. See Lemma~\ref{lemma:lifetime}. We note that $d\equiv d^*$ implies that $R\geq d^*$, so that strategy $(b)$ is always sub-optimal when the death rates are constant. In short, 
	\be\label{eq:lifetimeconv}
	\lim_{i\to\infty}d(i)=d^* \ \Longrightarrow \ L\approx \text{Exp}(\min\{d^*,R\}).
	\ee 
	Furthermore, strategy $(a)$ entails that an individual produces a large number of children, whereas strategy $(b)$ requires that an individual only produces a small number of children. Combined, we thus conclude that strategy $(a)$ being optimal corresponds to the `rich are old' regime, in which the oldest individuals survive for an exceptionally long time by producing a large offspring. Similarly, strategy $(b)$ corresponds to the `rich die young' regime, in which the oldest individuals survive for an exceptionally long time by only producing a small offspring, and individuals with a large offspring (which thus follow strategy $(a)$) survive for a much shorter amount of time and must therefore be born long after the oldest alive individuals to be able to survive.
	
	\textbf{The oldest alive individual. }  To find the birth-time of the oldest alive individual in $\bp(t)$, we set $M:=\min\{d^*,R\}$ and let $s=s(t)\ll t$. Suppose an individual is born at time $s$. This individual is alive at time $t$ if its lifetime is at least $t-s$, which occurs with probability approximately $\exp(-M(t-s))$ by~\eqref{eq:lifetimeconv}. Furthermore, the Malthusian parameter $\lambda^*$ in Assumption~\eqref{ass:C1} determines the growth-rate of the branching process $\bp$, see~\eqref{eq:maltconv}. This implies that the number of individuals born `around' time $s$ is of the order $\exp(\lambda^* s)$. Combined, a second moment method yields that there exists an individual born at time $s$ that is alive at time $t$ when $s$ solves
	\be \label{eq:Otcontasymp}
	\exp\big(\lambda^*s-M(t-s)\big)\geq 1, \qquad\text{so that}\qquad O_t^\cont=\frac{M}{\lambda^*+M}t=\begin{cases}
		\frac{d^*}{\lambda^*+d^*}t &\mbox{if } d^*<R, \\
		\frac{R}{\lambda^*+R}t&\mbox{if } d^*\geq R.
	\end{cases}
	\ee   
	Together with the approximate relation between $O_t^\cont$ and $O_n$ in~\eqref{eq:corr} this yields the first-order term in the scaling limit for $O_n$ in Theorems~\ref{thrm:conv} and~\ref{thrm:biggerR}. In the `rich die young' regime, the assumption that $d$ converges to $d^*>R$ can be weakened to $\liminf_{i\to\infty}d(i)>R$, for which a similar argument can be applied. 
	
	When $d^*=0$, i.e.\ when $d$ converges to zero, we distinguish two sub-cases. \\$(i)$ If Assumption~\eqref{ass:A2} is satisfied, then a similar reasoning as strategy $(a)$ can be used to show that the lifetime of individuals has sub-exponential tails, see Lemma~\ref{lemma:lifetime2nd}. The above argument for the oldest alive individual then yields that $O_t^\cont=o(t)$, and this can be made more precise to yield the growth-rate of $O_n$ in Theorem~\ref{thrm:conv} for $d^*=0$, where we then see that the first-order term in the scaling limit, as in~\eqref{eq:Otcontasymp}, disappears. \\	
	$(ii)$ If Assumption~\eqref{ass:A2} is not satisfied, it is readily checked that $D=\infty$ and thus $S_{D+1}=\infty $ (under Assumption~\eqref{ass:A1}) with positive probability, see Lemmas~\ref{lemma:Dlifedistr} and~\ref{lemma:stochdomexp}, respectively. As a result $O_t^\cont$ equals the birth-time of the first individual that has an infinite lifetime for all $t$ large, which is finite almost surely. Equivalently, $O_n$ equals the label of the first individual that never dies for all $n$ large, almost surely, as in Theorem~\ref{thrm:asPA}.
	
	\textbf{Lack of persistence in the `rich are old' regime. }  We argued that in the `rich die young' regime the oldest individuals survive by producing a small number of children, whereas individuals that produce a large number of children live significantly shorter. This, at least heuristically, implies lack of persistence in the `rich die young' regime. In the `rich are old' regime, however, we argued that the oldest individuals are able to survive for a long time by producing a large offspring. Lack of persistence then occurs under the \emph{additional} condition that Assumption~\eqref{ass:varphi2} is satisfied, i.e.\ when $\varphi_2$ tends to infinity. This is a generalisation of some known results in the literature (see~\cite{DerMor09,Gal13,Iyer24}, and~\cite{BanBha21} in particular), which we discuss now. 
	
	In the `rich are old' regime, the oldest individual is born at time $O_t^\cont\approx \frac{d^*}{\lambda^*+d^*}t$ by~\eqref{eq:Otcontasymp}, so that it has $t-O_t^\cont\approx \frac{\lambda^*}{\lambda^*+d^*}t$ amount of time to produce children. We recall $\varphi_1$ and $\varphi_2$ from~\eqref{eq:seqs} and $S_k$ from~\eqref{eq:D}, and note that $\E{S_k}=\varphi_1(k)$ and $\Var(S_k)=\varphi_2(k)$. In expectation, the oldest individual thus produces approximately $\varphi_1^{-1}(t-O_t^\cont)\approx \varphi_1^{-1}(\frac{\lambda^*}{\lambda^*+d^*}t)$ many children by time $t$.
	
	To show there are younger individuals (i.e.\ that are born later) that obtain a larger offspring by time $t$, we need to balance two things. $(i)$ Younger individuals typically produce fewer children by time $t$, as they have less time at their disposal. $(ii)$ There are many more younger individuals that survive until time $t$, as the branching process grows exponentially fast and younger individuals need to survive for a smaller amount of time. The aim is thus to find a time in which many individuals are born, among which one is sufficiently lucky to produce a large number of children, much faster than is typical. Let us make this idea more precise now. 
	
	Suppose another individual is born at time 
	\be 
	s=s(t):=\frac{d^*}{\lambda^*+d^*}t+\frac{1}{\lambda^*+d^*}\cK_\alpha\big(\tfrac{\lambda^*}{\lambda^*+d^*}t\big)+x_1\cK\big(\tfrac{\lambda^*}{\lambda^*+d^*}t\big),
	\ee 
	where $x_1>0$ is a constant and we recall $\cK_\alpha$ and $\cK$ from~\eqref{eq:Ks}. An individual has at least
	\be 
	k=k(t):=\varphi_1^{-1}\big(\tfrac{\lambda^*}{\lambda^*+d^*}t-\tfrac{1}{\lambda^*+d^*}\cK_\alpha\big(\tfrac{\lambda^*}{\lambda^*+d^*}t\big)+x_2\cK\big(\tfrac{\lambda^*}{\lambda^*+d^*}t\big) \big)
	\ee 
	children, for some constant $x_2>0$, and survives until time $t$ when $D\geq k$, $S_k\leq t-s$, and $L=S_{D+1}>t-s$.  Indeed, $D\geq k$ means that the individual produces at least $k$ children before its death, $S_k\leq t-s$ means that at least $k$ children are born by time $t$, and $S_{D+1}>t-s$ implies that the individual is alive at time $t$ (as the individual is born at time $s$). We claim that we can omit the last requirement that $S_{D+1}>t-s$. First, because it simplifies the heuristic explanation, but also because we have chosen $k$ and $s$ in such a way that $S_k$ is typically `very close to $t-s$', so that surviving after producing $k$ children is not `too unlikely' and therefore does not influence the result `too much'. One can make this heuristic reasoning precise and forms the ground for Remark~\ref{rem:incldead}. 
	
	As $D$ and $S_k$ are independent, we thus obtain
	\be \ba 
	\mathbb P(\text{Individual born at time }s\text{ has $\geq k$ children and is alive at time }t)\leq \P{D\geq k}\P{S_k<t-s}.
	\ea \ee
	Then,
	\begin{align*}
		\E{S_k}&=\varphi_1(k)= \frac{\lambda^*}{\lambda^*+d^*}t-\frac{1}{\lambda^*+d^*}\cK_\alpha\big(\tfrac{\lambda^*}{\lambda^*+d^*}t\big)+x_2\cK\big(\tfrac{\lambda^*}{\lambda^*+d^*}t\big),
		\intertext{and} \Var(S_k)&=\varphi_2(k)=\cK\big(\tfrac{\lambda^*}{\lambda^*+d^*}t-\tfrac{1}{\lambda^*+d^*}\cK_\alpha\big(\tfrac{\lambda^*}{\lambda^*+d^*}t\big)+x_2\cK\big(\tfrac{\lambda^*}{\lambda^*+d^*}t\big)\big)\approx  \cK\big(\tfrac{\lambda^*}{\lambda^*+d^*}t\big), 
	\end{align*}
	where we use~\eqref{ass:K1} in Assumption~\hyperref[ass:K]{$\cK$} in the last step. We can thus rewrite 
	\be \ba 
	\P{S_k<t-s}&=\P{S_k<\frac{\lambda^*}{\lambda^*+d^*}t-\frac{1}{\lambda^*+d^*}\cK_\alpha\big(\tfrac{\lambda^*}{\lambda^*+d^*}t\big)-x_1\cK\big(\tfrac{\lambda^*}{\lambda^*+d^*}t\big)}\\ 
	&\approx\P{\frac{S_k-\E{Sk}}{\sqrt{\Var(S_k)}}<-(x_1+x_2)\sqrt{\cK\big(\tfrac{\lambda^*}{\lambda^*+d^*}t\big)}}.
	\ea \ee 
	Since $\varphi_2$, and hence $\cK$, tends to infinity, we can use a moderate deviation principle, see Lemma~\ref{lemma:mdp}, to estimate this probability by 
	\be \label{eq:mdpapprox}
	\P{S_k<t-s}\approx\exp\Big(-\frac12(x_1+x_2)^2 \cK\big(\tfrac{\lambda^*}{\lambda^*+d^*}t\big)\Big).
	\ee 
	Finally, we can approximate the tail distribution of $D$, see Lemma~\ref{lemma:Dtail}, by 
	\be 
	\P{D\geq k}=\prod_{i=0}^{k-1}\frac{b(i)}{b(i)+d(i)}\approx \exp\big(-\rho_1(k)-\tfrac12 \rho_2(k)\big)=\exp\big(-d^*\varphi_1(k)-\alpha(k)-\tfrac12\rho_2(k)\big), 
	\ee 
	where we recall $\rho_1$ and $\rho_2$ from~\eqref{eq:seqs} and $\alpha$ from~\eqref{eq:alpha}. Also using the definition of $\cK_\alpha$ in~\eqref{eq:Ks}, that $\rho_2(k)\approx (d^*)^2 \varphi_2(k)$ when $d$ converges to $d^*$, and the choice of $k$, this approximately equals 
	\be \ba  \exp\Big({}&-\frac{\lambda^*d^*}{\lambda^*+d^*}t+\frac{d^*}{\lambda^*+d^*}\cK_\alpha\big(\tfrac{\lambda^*}{\lambda^*+d^*}t\big)-\cK_\alpha\big(\tfrac{\lambda^*}{\lambda^*+d^*}t-\tfrac{1}{\lambda^*+d^*}\cK_\alpha\big(\tfrac{\lambda^*}{\lambda^*+d^*}t\big)+x_2\cK\big(\tfrac{\lambda^*}{\lambda^*+d^*}t\big)\big)\\ 
	&-\big(x_2d^*+\tfrac12 (d^*)^2\big) \cK\big(\tfrac{\lambda^*}{\lambda^*+d^*}t\big) \Big).
	\ea\ee 
	Using Assumption~\hyperref[ass:Kalpha]{$\cK_\alpha$}, we can simplify this to 
	\be 
	\P{D\geq k}\approx \exp\Big(-\frac{\lambda^*d^*}{\lambda^*+d^*}t-\frac{\lambda^*}{\lambda^*+d^*}\cK_\alpha\big(\tfrac{\lambda^*}{\lambda^*+d^*}t\big) -\big(x_2d^*+\tfrac12(d^*)^2\big)\cK\big(\tfrac{\lambda^*}{\lambda^*+d^*}t\big) \Big).
	\ee 
	Combining this with~\eqref{eq:mdpapprox}, we thus arrive at 
	\be \ba 
	\mathbb P({}&\text{Individual born at time }s(t)\text{ has $\geq k(t)$ children and is alive at time }t)\\ 
	&\approx \exp\Big(-\frac{\lambda^*d^*}{\lambda^*+d^*}t-\frac{\lambda^*}{\lambda^*+d^*}\cK_\alpha\big(\tfrac{\lambda^*}{\lambda^*+d^*}t\big) -\big(\tfrac12(x_1+x_2)^2+x_2d^*+\tfrac12(d^*)^2\big)\cK\big(\tfrac{\lambda^*}{\lambda^*+d^*}t\big) \Big).
	\ea\ee 
	We then use that the exponential growth-rate of the branching process $\bp$, as in~\eqref{eq:maltconv}, implies that `around' time $s$ roughly $\exp(\lambda^*\!s)$ many individuals are born. As a result, the number of individuals born `around time $s(t)$' that is alive at time $t$ with at least $k(t)$ children is approximately
	\be \ba 
	\exp{}&\Big(\big[\lambda^*x_1 -\big(\tfrac12(x_1+x_2)^2+x_2d^*+\tfrac12(d^*)^2\big)\big]\cK\big(\tfrac{\lambda^*}{\lambda^*+d^*}t\big) \Big)\\ 
	&=\exp\Big(\big[(\lambda^*+d^*)x_1 -\tfrac12(x_1+x_2+d^*)^2\big]\cK\big(\tfrac{\lambda^*}{\lambda^*+d^*}t\big) \Big).
	\ea \ee 
	This expression is maximised for $x_1=\lambda^*-x_2$ and equals one when $x_1=\frac{\lambda^*+d^*}{2}$ and $x_2=\frac{\lambda^*-d^*}{2}$. This leads to the desired birth-time of the individual with the largest offspring, and the corresponding size of the offspring, as in Theorems~\ref{thrm:asPA} and~\ref{thrm:conv}.

	\section{Preliminaries: offspring and the remaining lifetime distribution}\label{sec:prelim}
	
	In this section, we collect a number of preliminary results that are  that we leverage in the analysis of $O^\cont_t$ and $I^\cont_t$ later on. We state the results here only and provide their proofs in Appendix~\ref{app:proofs}. 
	
	Recall the random variables $D$ and $S_k$ from~\eqref{eq:D}. An individual produces $D$ many children during its lifetime. We thus coin $D$ the \emph{offspring}. $S_k$ is the time it requires to produce $k\in\N$ many children and we coin $L-S_k$ the \emph{remaining lifetime}. That is, the amount of time an individual lives after producing $k$ children. We also let $\deg(t)$ denote the number of children of the root at time $t$, $D$ its offspring, $L$ its lifetime, and $S_k$ the time for the root to produce $k$ children provided $k\leq D$.
	
	The following corollary is immediate from the definitions in~\eqref{eq:lifetime} and~\eqref{eq:offspring}.
	
	\begin{corollary}[Degree and lifetime distribution of PAVER]\label{cor:deglifedistr}
		Jointly, 
		\be 
		(L,\deg(t))\overset \dd= \big(S_{D+1},\min\{\sup\{k\in\N_0: S_k\leq t\},D\}\big).
		\ee 
	\end{corollary}
	
	\subsection{Offspring}
	To study the distribution of the offspring $D$, we first introduce the following class of random variables. 
	
	\begin{definition}[Inhomogeneous geometric random variable]\label{def:inhomgeom}
		Let $(p_i)_{i\in\N_0}$ and $(q_i)_{i\in\N_0}$ be two sequences of non-negative real numbers such that $p_i+q_i>0$ for all $i\in\N_0$. We say that a random variable $G=G((p_i)_{i\in\N_0},(q_i)_{i\in\N_0})$ is an \emph{inhomogeneous geometric random variable} characterised by the sequences $(p_i)_{i\in\N_0}$ and $(q_i)_{i\in\N_0}$ when 
		\be 
		\P{G=k}=\frac{q_k}{p_k+q_k}\prod_{i=0}^{k-1}\frac{p_i}{p_i+q_i}, \qquad \text{for }k\in \N_0\qquad\text{and} \qquad \P{G=\infty}=\prod_{i=0}^\infty \frac{p_i}{p_i+q_i}. 
		\ee 
		In case $p_i=p$ for all $i\in\N_0$, we say that $G$ is characterised by $p$ and $(q_i)_{i\in\N_0}$ (and similarly if $q_i=q$ for all $i\in\N_0$).
	\end{definition}
	
	With this definition at hand, we have the following result.
	
	\begin{lemma}\label{lemma:Dlifedistr}
		The random variable $D$ is an inhomogeneous geometric random variable, characterised by the sequences $(b(i))_{i\in\N_0}$ and $(d(i))_{i\in\N_0}$. Moreover,  $\P{D=\infty}=0$ if and only if Assumption~\eqref{ass:A2} is satisfied. 
	\end{lemma}
	
	Let us further distinguish between different behaviour of the lifetime and offspring based on Assumptions~\eqref{ass:A1} and~\eqref{ass:A2}. First, when Assumption~\eqref{ass:A2} is satisfied, $D$ is finite almost surely, as follows from Lemma~\ref{lemma:Dlifedistr}. So, irrespective of whether Assumption~\eqref{ass:A1} holds or not, $L=S_{D+1}$ is finite almost surely under Assumption~\eqref{ass:A2}, so individuals live for an almost surely finite time. 
	
	When, instead, Assumption~\eqref{ass:A2} does not hold, an individual can give birth to an infinite number of children with positive probability. In this case, we make two further distinctions, based on whether Assumption~\eqref{ass:A1} is satisfied. If Assumption~\eqref{ass:A1} holds, then $S_\infty=\infty$ almost surely. Indeed, for any $x>0$, 
	\be \label{eq:Sinf}
	\P{S_\infty \leq x}\leq \e^x \E{\e^{-S_\infty}}=\e^x \prod_{i=0}^\infty \frac{b(i)+d(i)}{1+b(i)+d(i)}\leq \e^x \frac{1}{1+\sum_{i=0}^\infty (b(i)+d(i))^{-1}}=0.
	\ee 
	Hence, when Assumption~\eqref{ass:A1} is satisfied but Assumption~\eqref{ass:A2} is not, an individual has an infinite lifetime with positive probability, during which it produces an infinite offspring. If Assumption~\eqref{ass:A1} is also not met, then 
	\be 
	\E{S_\infty}=\E{\sum_{i=0}^\infty E_i}=\sum_{i=0}^\infty \frac{1}{b(i)+d(i)}<\infty, 
	\ee 
	Hence, each individual has a finite lifetime almost surely, but with positive probability an individual produces an infinite offspring during its lifetime. This is summarised in Table~\ref{table:deglife}. Throughout the paper, we assume that Assumption~\eqref{ass:A1} is satisfied, so that no individual can produce an infinite offspring in finite time.
	
	\begin{table}[h]
		\centering
		\begin{tabular}{|l|c|c|}
			\hline
			& \eqref{ass:A1} holds & \eqref{ass:A1} does not hold \\ 
			\hline
			\eqref{ass:A2} does not hold & $S_\infty=\infty$ a.s., $\P{D=\infty}\in(0,1)$ & $S_\infty<\infty$ a.s., $\P{D=\infty}\in(0,1)$ \\
			\hline
			\eqref{ass:A2} holds & \multicolumn{2}{c|}{$D$ and $ S_{D+1}$ are finite a.s.} \\ 
			\hline
		\end{tabular}
		\caption{An overview of the behaviour of the random variables $D$ and $S_{D+1}$, based on Assumptions~\eqref{ass:A1} and~\eqref{ass:A2}.}\label{table:deglife} 
	\end{table} 
	
	To conclude this subsection, we provide bounds for the tail distribution of $D$ in terms of the sequences $\rho_1$ and $\rho_2$, defined in~\eqref{eq:seqs}.  Observe that for any $k\in\N_0$ and irrespective of whether Assumption~\eqref{ass:A2} holds,
	\be \label{eq:tailD}
	\P{D\geq k}=\prod_{i=0}^{k-1}\frac{b(i)}{b(i)+d(i)}.
	\ee

	\begin{lemma}[Tail bounds for the distribution of $D$]\label{lemma:Dtail}
		Let $D$ be as in~\eqref{eq:D}. Then,
		\be 
		\P{D\geq k}\leq \e^{-\rho_1(k)-\frac12 \rho_2(k)}. 
		\ee 
		Furthermore, assume that $b$ and $d$ are such that Assumption~\eqref{ass:A2} is satisfied. When $\rho_2$ diverges and $d=o(b)$, 
		\be 
		\P{D\geq k}=\e^{-\rho_1(k)-(\frac12+o(1))\rho_2(k)}.
		\ee 
		If, instead, Assumption~\eqref{ass:A2} is not satisfied $($i.e.\ $\lim_{k\to\infty}\rho_2(k)$ exists$)$, 
		\be 
		\P{D\geq k}=\e^{-\rho_1(k) -\cO(1)}.
		\ee 
	\end{lemma} 	
	
	\subsection{Remaining lifetime  and large degrees}
	The lifetime equals $L=S_{D+1}$. We first study distributional properties of the \emph{remaining lifetime}, defined as 
	\be 
	S_{D+1}-S_k \text{ conditionally on }D\geq k \qquad \text{for }k\in \N_0. 
	\ee 
	The remaining lifetime is equal to $L$ when $k=0$, but equals the amount of time an individual has left to live after producing $k$ children when $k>0$. Conditionally on $\{D\geq k\}$, let $D_k:=D-k$ denote the \emph{excess offspring}. It is clear that $D_k$ has distribution
	\be \label{eq:Dk}
	\P{D_k=\ell}=\P{D=k+\ell\,|\, D\geq k}=\frac{d(k+\ell)}{b(k+\ell)+d(k+\ell)}\prod_{i=k}^{k+\ell-1}\frac{b(i)}{b(i)+d(i)}, \qquad\text{for } \ell\in\N_0. 
	\ee 
	That is, $D_k$ is an inhomogeneous geometric random variable, characterised by the sequences $(b(k+\ell))_{\ell\in\N_0}$ and $(d(k+\ell))_{\ell\in\N_0}$.  Upon $\{D\geq k\}$ we can write $D=k+D_k$. Hence, the remaining lifetime after giving birth to $k$ children equals 
	\be 
	\sum_{i=k}^{k+D_k}E_i. 
	\ee 
	Regarding this random variable, we have the following results. 
	
	\begin{lemma}[Remaining lifetime]\label{lemma:stochdomexp}
		Suppose that $b$ and $d$ are such that Assumption~\eqref{ass:A1} is satisfied. Suppose there exists $I\in\N_0$ and $d^*\in(0,\infty)$ such that $d(i)=d^*$ for all $i\geq I$. Then,
		\be \label{eq:Lexp}
		\sum_{i=k}^{k+D_k}E_i\sim \mathrm{Exp}(d^*) \quad \text{for any }k\geq I.
		\ee 
		In particular, when $I=0$ so that $d\equiv d^*$, then $L \sim \mathrm{Exp}(d^*)$. Fix $k\in \N_0,\lambda>0$, and let $E_\lambda \sim \mathrm{Exp}(\lambda)$. Suppose  that $d(i)\geq \lambda$ $($resp.\ $d(i)\leq \lambda)$ for all $i\geq k$. Then,
		\be \label{eq:remainstochdom}
		\sum_{i=k}^{k+D_k}E_i\preceq E_\lambda \qquad \bigg(\text{resp.\ }\sum_{i=k}^{k+D_k}E_i\succeq E_\lambda\bigg).
		\ee   
		Finally, when $b$ and $d$ are such that Assumption~\eqref{ass:A1} is satisfied but Assumption~\eqref{ass:A2} is not,
		\be \label{eq:infremlife}
		\P{\sum_{i=k}^{k+D_k}E_i=\infty}>0\qquad\text{for any }k\in\N_0.
		\ee 
	\end{lemma} 

	When $d$ converges to some value $d^*\in(0,\infty)$, the above result suggests that  the remaining lifetime, after having giving birth to a large number of children, is roughly  exponentially distributed with rate $d^*$. Similarly, when $d$ converges to zero, the remaining lifetime has sub-exponential tails and even equals infinity with positive probability when Assumption~\eqref{ass:A2} is not satisfied. We make this intuition precise and build on this even further when studying the \emph{entire} lifetime in the next section. For now, we discuss the limiting distribution of the remaining lifetime when $d$ diverges, in which case the remaining lifetime has super-exponential tails. 
	
	\begin{lemma}[Remaining lifetime for diverging $d$]\label{lemma:remlifediv}
		Suppose that $b$ and $d$ both tend to infinity, such that $d$ is increasing, $d=o(b)$, and $b(k)=o(kd(k))$, and that Assumption~\eqref{ass:A1} is satisfied. Furthermore, assume that both $b(k)$ and $d(k)$ are regularly varying in $k$ with a non-negative exponent. Let $t_k$ be such that $\liminf_{k\to\infty}t_k>0$ and $t_k=o(kd(k)/b(k))$. Then,
		\be \label{eq:unifconv}
		\P{d(k)\sum_{i=k}^{k+D_k}E_i>t_k}=\e^{-(1+o(1))t_k}.
		\ee  
		In particular, with $E$ a rate-one exponential random variable, 
		\be \label{eq:convddiv}
		d(k)\sum_{i=k}^{k+D_k}E_i \toindis E, \qquad\text{as }k\to\infty. 
		\ee 
	\end{lemma} 
	
	\begin{remark}
		It is possible to relax the assumption that $d$ is increasing by using similar quantities $\overline d_k$ and $\underline d_k$ throughout the proof, and assuming that $d(k)=(1+o(1))\inf_{i\geq k}d(i)$ in the case that $d$ is slowly varying (i.e.\ regularly varying with exponent $0$). We used the assumption on $d$ to ease notation and aid the reader, as well as since applications of Lemma~\ref{lemma:remlifediv} require only increasing death rates (see Section~\ref{sec:maxrdy}).\ensymboldremark
	\end{remark}
	
	\begin{remark}[Linear birth rates]\label{rem:linb1}
		We can weaken the assumption that $t_k=o(kd(k)/b(k))$ somewhat in the case that $b(k)=\cO(k)$. Namely, we can assume that $t_k=\cO(d(k))$. This weaker assumption is required in the proof of Proposition~\ref{prop:divdeathhighdeg}. See Remark~\ref{rem:linb} in Appendix~\ref{app:proofs} for more details on how to adjust the proof.\ensymboldremark
	\end{remark}
	
	So far, we have either looked at the distribution of the offspring $D$ or the remaining lifetime separately. We are also interested in events that combine both quantities, in particular how likely (or unlikely) it is for an individual to survive for a long time and obtain a large offspring.  By Corollary~\ref{cor:deglifedistr}, we can write the event that the root has degree at least $k$ at time $t$ and is alive at time $t'\geq t$ as 
	\be \label{eq:degliferewrite}
	\{d(t)\geq k, L> t'\}=\{D\geq k, S_k\leq t, S_{D+1}>t'\}. 
	\ee		
	We then have  the following result.
	
	\begin{lemma}[Survival with a high degree]\label{lemma:survdeg}
		$(i)$ Suppose that Assumptions~\eqref{ass:A1} and~\eqref{ass:A2} are satisfied. Suppose that there exist $x\geq 0$ and $K=K(x)\in\N_0$ such that $d(i)\geq x$ for all $i\geq K$. Then, for all $k\geq K$ and all $t'\geq t\geq0$,
		\be 
		\P{D\geq k, S_k\leq t, S_{D+1}>t'}\leq \e^{-x(t'-t)}\P{D\geq k}\E{\ind_{\{S_k\leq t\}}\e^{x(S_k-t)}}.
		\ee 
		$(ii)$ Suppose that Assumptions~\eqref{ass:A1} and~\eqref{ass:A2} are satisfied. Suppose that there exist $x\geq 0$ and $K=K(x)\in\N$ such that $d(i)\leq x$ for all $i\geq K$. Then, for all $k\geq K$ and all $t'\geq t\geq0$,  
		\be 
		\P{D\geq k, S_k\leq t, S_{D+1}>t'}\geq \e^{-x(t'-t)}\P{D\geq k}\E{\ind_{\{S_k\leq t\}}\e^{x(S_k-t)}}.
		\ee 
		$(iii)$ Suppose that Assumption~\eqref{ass:A1} is satisfied but Assumption~\eqref{ass:A2} is not. Then, 
		\be 
		\P{D\geq k, S_k\leq t, S_{D+1}>t'}=\Theta(\P{S_k\leq t}) \qquad \text{for any $k\in\N_0$ and $t'\geq t\geq 0$,} 
		\ee 
		where the constants in the $\Theta$ notation are independent of $k,t'$, and $t$.
	\end{lemma} 
	
	To use the bounds presented in Lemma~\ref{lemma:survdeg}, the asymptotic behaviour of the expected values is required (for a suitable choice of $t$). To this end, we present the following result from~\cite{BanBha21}, regarding moderate deviations of sums of independent exponential random variables. 
	
	\begin{lemma}[Moderate deviation principle, Lemma $7.10$~\cite{BanBha21}]\label{lemma:mdp} Let $f\colon \N_0\to(0,\infty)$ and assume $f(k)$ tends to infinity with $k$. Let $(E_i)_{i\in\N_0}$ be a sequence of independent exponential random variables, where $E_i$ has rate $f(i)$, and set $S_k:=\sum_{i=0}^{k-1}E_i$. Assume that $\lim_{k\to\infty}\Var(S_k)=\infty$. Then, for $z\geq 0$, 
		\begin{align*}			\lim_{k\to\infty}\frac{1}{\Var(S_k)}\log\P{S_k-\E{S_k}\geq z\Var(S_k)}&=-\frac{z^2}{2}, 
			\intertext{and} 
			\lim_{k\to\infty}\frac{1}{\Var(S_k)}\log\P{S_k-\E{S_k}\leq -z\Var(S_k)}&=-\frac{z^2}{2}.
		\end{align*}
	\end{lemma}
	
	\begin{remark}
		We shall use Lemma~\ref{lemma:mdp} throughout with $f\equiv b+d$, so that $\E{S_k}=\varphi_1(k)$ and $\Var(S_k)=\varphi_2(k)$. \ensymboldremark
	\end{remark}
	
	We extend this result by including exponential weighting or tilting of the sum $S_k$. We phrase the following result in terms of the sequences $b$ and $d$ for ease of writing, since we apply this in combination with Lemma~\ref{lemma:survdeg} in the analysis in Section~\ref{sec:max}.
	
	\begin{proposition}[Moderate deviation principle with exponential tilting]\label{prop:mdptilt} Suppose that $b$ and $d$ are such that $b$ diverges and that Assumptions~\eqref{ass:A1} and~\eqref{ass:varphi2} are satisfied. Then, for $z> 0$ and $y,\theta \in \R$, 
		\be \label{eq:mdptiltmin}
		\lim_{k\to\infty}\frac{1}{\varphi_2(k)}\log\E{\ind_{\{S_k\leq \varphi_1(k)-z\varphi_2(k)\}}\e^{\theta(S_k-(\varphi_1(k)-y\varphi_2(k)))}}=\theta (y-z)-\frac{z^2}{2}.
		\ee 
	\end{proposition}

\subsection{Functional inequalities} 

We conclude this section with a few elementary results regarding the functions $\varphi_1, \varphi_2$, and $\rho_1, \rho_2$, as well as $\alpha$ and $\cK$ and $\cK_\alpha$, as introduced in Section~\ref{sec:model}. 

\begin{lemma}[Functional inequalities]\label{lemma:func} 
	\begin{enumerate} 	
		\item Suppose that Assumption~\eqref{ass:A1} is satisfied and $b$ tends to infinity. Then, $\varphi_2=o(\varphi_1)$ and  $\cK(t)=o(t)$. Similarly, suppose that Assumption~\eqref{ass:A2} is satisfied and $d=o(b)$. Then, $\rho_2=o(\rho_1)$. 
		\item Suppose that Assumption~\eqref{ass:A1} is satisfied and that $\overline d:=\limsup_{i\to\infty}d(i)<\infty$. Then, $\limsup_{k\to\infty} \rho_1(k)/\varphi_1(k)\leq \overline d$. In particular, when $d$ converges to zero, then $\rho_1=o(\varphi_1)$. Similarly, if $\underline d:=\limsup_{i\to\infty}d(i)>0$, then $\liminf_{k\to\infty}\rho_1(k)/\varphi_1(k)\geq \underline d$. In particular, if $d$ tends to infinity, then $\varphi_1=o(\rho_1)$.
		\item Suppose that Assumption~\eqref{ass:varphi2} is satisfied and $\lim_{i\to\infty}d(i)=d^*\in[0,\infty)$. Then, $\rho_2(k)=((d^*)^2+o(1))\varphi_2(k)$. 
		\item Suppose that Assumption~\eqref{ass:A1} is satisfied and $\lim_{i\to\infty}d(i)=d^*\in[0,\infty)$. Then, $\cK_\alpha(t)=o(t)$. Moreover, for any function $s\colon (0,\infty)\to\R_+$ such that $s$ tends to infinity and $s(t)=o(t)$, we have $\cK_\alpha(t)-\cK_\alpha(t-s(t))=o(s(t))$.
	\end{enumerate}
\end{lemma}

\section{The oldest alive individual}\label{sec:old}

In this section, we study the quantity $O_t^\cont$; the birth-time of the oldest alive individual in $\bp(t)$ conditionally on survival. To this end, we first obtain precise asymptotic distributional properties of the lifetime $L$ of individuals in the branching process. We then study the growth-rate of the entire branching process, to finally prove when the oldest  individual alive at time $t$ is born. 

\subsection{Lifetime distribution} \label{sec:lifetime}

In the previous section, we obtained results for the remaining lifetime after giving birth to a number of children. Here, we use those results to now focus on the \emph{entire} lifetime. 

The first result of Lemma~\ref{lemma:stochdomexp} already shows that the lifetime is exponentially distributed when the death rates are constant (and non-zero). We now provide some further asymptotic results when the death rates are not constant. Here we observe that the lifetime distribution undergoes a phase transition between the `rich are old` and the `rich die young' regimes. This is made precise in the following result.

\begin{lemma}[Asymptotic exponential lifetime distribution]\label{lemma:lifetime}
	Assume that the sequences $b$ and $d$ are such that Assumption~\eqref{ass:A1} is satisfied. Recall  $R$ from~\eqref{eq:R}. If $\liminf_{i\to\infty}d(i)\geq R$, then
	\be \label{eq:LlimR}
	\lim_{t\to\infty}\frac1t\log\P{L>t}=-R.
	\ee  
	Moreover, when $\liminf_{i\to\infty}d(i)>R$,
	\be \label{eq:Lprecise}
	\liminf_{t\to\infty} \big[\log(\P{L>t})+Rt\big]>-\infty\qquad\text{and}\qquad \ \limsup_{t\to\infty}\frac{\log(\P{L>t})+Rt}{\log t}<\infty.
	\ee 
	If $\overline d:=\limsup_{i\to\infty}d(i)<R$ and $b$ tends to infinity, then 
	\be \label{eq:Lliminfd}
	\liminf_{t\to\infty}\frac1t\log\P{L>t}\geq -\overline d.
	\ee  
	In particular, if $\lim_{i\to\infty}d(i)=d^*\in[0,R)$ exists and $b$ tends to infinity, 
	\be \label{eq:Llimdstar}
	\lim_{t\to\infty}\frac1t\log\P{L>t}= -  d^*.
	\ee 
\end{lemma} 

\begin{remark}
	Whilst~\eqref{eq:LlimR} holds in a slightly more general setting, we need the more precise results in~\eqref{eq:Lprecise} in later analysis. We included both results, as their proof are relatively short.	Furthermore, the bounds in~\eqref{eq:Lprecise} are optimal, in the sense that the rescaling is of the correct order in both the liminf and limsup.\ensymboldremark
\end{remark} 

In the case that $\lim_{i\to\infty}d(i)=d^*<R$, we can also provide higher-order asymptotic behaviour. 

\begin{lemma}[Third-order asymptotic lifetime distribution]\label{lemma:lifetime2nd}
	Suppose that $b$ and $d$ are such that Assumptions~\eqref{ass:A1} and~\eqref{ass:A2} are satisfied, and recall $R$ from~\eqref{eq:R}. Also suppose that $b$ diverges and that $\lim_{i\to\infty}d(i)=d^*\in[0,R)$. Recall the functions $\cK$ and $\cK_\alpha$ from~\eqref{eq:Ks}, and suppose that Assumptions~\hyperref[ass:K]{$\cK$} and~\hyperref[ass:Kalpha]{$\cK_\alpha$} and Assumption~\eqref{ass:varphi2} are satisfied. Then, 
	\be \label{eq:Llim2nd}
	\lim_{t\to\infty}\frac{\log(\P{L>t})+d^*t+\cK_\alpha(t)}{\cK(t)}=0.
	\ee 
\end{lemma}

Finally, we consider the behaviour of the lifetime when Assumption~\eqref{ass:A1} is satisfied but Assumption~\eqref{ass:A2} is not. This is a direct application of~\eqref{eq:infremlife} in Lemma~\ref{lemma:stochdomexp} when $k=0$.

\begin{corollary}[Infinite lifetime]\label{cor:inflifetime}
	Suppose that $b$ and $d$ are such that Assumption~\eqref{ass:A1} is satisfied, but Assumption~\eqref{ass:A2} is not. Then, $\P{L=\infty}>0$. 
\end{corollary} 

Lemma~\ref{lemma:lifetime} provides several general results in case Assumption~\eqref{ass:A1} holds. The `rich die young' and `rich are old' regimes correspond to the results in~\eqref{eq:LlimR} and~\eqref{eq:Lprecise}, and the results in~\eqref{eq:Lliminfd} and~\eqref{eq:Llimdstar}, respectively, and we make the heuristic explanation of~\eqref{eq:lifetimeconv} in Section~\ref{sec:heur} precise in this Lemma. The result in~\eqref{eq:Llimdstar} is then strengthened to a third order asymptotic expansion under additional assumptions in Lemma~\ref{lemma:lifetime2nd}. Here, Assumption~\eqref{ass:varphi2} is crucial, whereas Assumptions~\hyperref[ass:K]{$\cK$} and~\hyperref[ass:Kalpha]{$\cK_\alpha$} are used for technical reasons. This precise result is necessary to study the behaviour of $O_t^\cont$, among others. Finally, Corollary~\ref{cor:inflifetime} provides necessary and sufficient conditions for individuals to have an infinite lifetime with positive probability, which is used in the proof of Theorem~\ref{thrm:asPA}.

We now prove the two lemmas. 

\begin{proof}[Proof of Lemma~\ref{lemma:lifetime}]		
	We start by proving~\eqref{eq:LlimR} and~\eqref{eq:Lprecise}, and first prove a general lower bound. Suppose there exists $I\in\N_0$ such that $b(I)+d(I)=R$.  Since $L\overset \dd= S_{D+1}$ by Corollary~\ref{cor:deglifedistr}, 
	\be 
	\P{L>t}=\P{S_{D+1}>t}\geq \P{D\geq I-1,S_I>t}\geq \P{D\geq I-1}\P{E_I>t}=\P{D\geq I-1}\e^{-Rt}. 
	\ee 
	Note that $\P{D\geq I-1}$ has strictly positive probability, since $b(i)>0$ for all $i\in\N_0$. Hence, there exists $C>0$ such that for all $t\geq 0$, 
	\be \label{eq:preciseLB}
	\log(\P{L>t})+Rt\geq -C. 
	\ee 		
	When $\liminf_{i\to\infty}d(i)>R$, the existence of such an index $I$ is guaranteed, which yields the first bound in~\eqref{eq:Lprecise}. The lower bound in~\eqref{eq:preciseLB} also implies that
	\be \label{eq:Lliminf}
	\liminf_{t\to\infty}\frac1t\log\P{L>t}\geq -R,
	\ee 
	which yields the desired lower bound for~\eqref{eq:LlimR} in the case the index $I$ exists.		
	
	Now suppose such an index $I$ does not exist. It follows that $b(j)+d(j)>R:=\inf_{i\in\N_0}(b(i)+d(i))$ for all $j\in\N_0$, so that $\liminf_{i\to\infty}b(i)+d(i)=R$ must hold.  Hence, for any $\eps>0$ there exists $I'\in\N_0$ such that $b(I')+d(I')\leq R+\eps$. Then, by the same argument,
	\be 
	\P{L>t} \geq \P{D\geq I'-1}\P{E_{I'}>t}\geq \P{D\geq I'-1}\e^{-(R+\eps)t}. 
	\ee 
	Again, the probability on the right-hand side is strictly positive, so that 
	\be
	\liminf_{t\to\infty}\frac1t\log\P{L>t}\geq -(R+\eps).
	\ee 
	As $\eps$ is arbitrary, the desired lower bound follows. We remark that the bound in~\eqref{eq:Lliminf} is thus true in general and that we have not used the assumption that $\liminf_{i\to\infty}d(i)\geq R$ for the proof (only for the stronger bound in~\eqref{eq:preciseLB} do we use that $\liminf_{i\to\infty}d(i)>R$), as it is indeed not necessary. Though true in general, the lower bound in~\eqref{eq:Lliminf} need not be sharp in case this assumption is not met. The assumption \emph{is} necessary for the upper bound, which we prove now.
	
	For an upper bound, let $I\in\N_0$ be such that $d(i)>R-\eps$ for all $i\geq I$, which exists since $\liminf_{i\to\infty}d(i)\geq R$. Then,
	\be \ba 
	\P{L>t}&\leq \P{S_{I}>t}+\P{S_{I}\leq t, S_{D+1}>t}=  \P{S_{I}>t}+\P{D\geq I,S_{I}\leq t, S_{D+1}>t}.
	\ea \ee 
	We start by bounding the first probability on the right-hand side. Without loss of generality we can assume that $I\neq0$. As $b(i)+d(i)\geq R$ for all $i\in\N_0$,  let $G(I,R)$ denote a sum of  $I$ i.i.d.\ exponential random variables with rate $R$, to bound
	\be 
	\P{S_{I}>t}\leq \P{G(I,R)>t}=\int_{t}^\infty \frac{R^{I}}{(I-1)!} s^{I-1}\e^{-R s}\, \dd s=\frac{1+o(1)}{(I-1)!}(Rt)^{I-1}\e^{-Rt}.
	\ee 
	We then bound the second probability on the right-hand side. By the choice of $I$ and using Lemma~\ref{lemma:survdeg}, we arrive at the upper bound
	\be  
	\P{D\geq I, S_{I}\leq t, S_{D+1}>t}\leq  \e^{-(R-\eps)t}\E{\e^{(R-\eps)S_{I}}}. 
	\ee 
	Combined, we thus arrive at
	\be \label{eq:Lub}
	\P{L>t}\leq \frac{1+o(1)}{(I-1)!}(Rt)^{I-1}\e^{-Rt}+\e^{-(R-\eps)t}\E{\e^{(R-\eps)S_{I}}}. 
	\ee 
	By the definition of $R$ and since $I$ is finite, the expected value equals a constant $C>0$, independent of $t$. We thus conclude that 
	\be \label{eq:Llimsup}
	\limsup_{t\to\infty} \frac 1t\log \P{L>t}\leq -(R-\eps), 
	\ee 
	Since $\eps$ is arbitrary, combining this with~\eqref{eq:Lliminf} yields~\eqref{eq:LlimR}. 
	
	We then prove the second bound in~\eqref{eq:Lprecise}. As we assume $\liminf_{i\to\infty}d(i)>R$, we observe that we can choose $\eps=1/t$ to depend on $t$ and that the index $I\in\N_0$ such that $b(i)+d(i)>R-\eps=R-1/t$ for all $i\geq I$ still exists and is \emph{independent} of $t$. Then, the upper bound in~\eqref{eq:Lub} changes to
	\be \ba 
	\P{L>t}&\leq\frac{1+o(1)}{(I-1)!}(Rt)^{I-1}\e^{-Rt}+\e^{-Rt+1}\E{\e^{(R-1/t)S_{I}}}\\ 
	&=\frac{1+o(1)}{(I-1)!}(Rt)^{I-1}\e^{-Rt}+\e^{-Rt+1}(Rt)^{|[I-1]_R|}\prod_{\substack{i=0\\ i\not\in [I-1]_R}}^{I-1}\frac{b(i)+d(i)}{b(i)+d(i)-R+1/t},
	\ea \ee 
	where 
	\be\label{eq:I-1R}
	[I-1]_R:=\{i\in\{0,\ldots, I-1\}: b(i)+d(i)=R\}.
	\ee 
	The product on the right-hand side can be bounded from above by a constant $C>0$, since $\liminf_{i\to\infty} d(i)>R$ and hence 
	\be\label{eq:largerRinf} 
	\inf\{b(i)+d(i): i\in \N_0\setminus [I-1]_R\}>R.
	\ee
	As $|[I-1]_R|\leq I$, we thus arrive at the upper bound
	\be 
	\limsup_{t\to\infty}\frac{\log(\P{L>t})+Rt}{\log t}\leq I, 
	\ee 
	which concludes the proof of~\eqref{eq:Lprecise}.
	
	We continue by proving~\eqref{eq:Lliminfd}. We let $I=I(\eps)\in \N$ be such that $d(i)\leq \overline d+\eps$ for all $i\geq I$. Then,
	\be 
	\P{L>t}\geq \P{S_{D+1}>t, D\geq I}=\mathbb P\bigg(S_{I}+\!\!\sum_{i=I}^{I+D_{I}}\!\!\!E_i>t\bigg)\P{D\geq I}\geq \mathbb P\bigg( \sum_{i=I}^{I+D_{I}}\!\!\!E_i>t\bigg)\P{D\geq I}.
	\ee 
	Since $b(i)>0$ for all $i\in\N_0$, we have $\P{D\geq I}>0$. Then, by Lemma~\ref{lemma:stochdomexp}, we can bound the other term on the right-hand side from below by $\exp(-(\overline d+\eps)t)$, so that 
	\be 
	\liminf_{t\to\infty}\frac1t \log\P{L>t}\geq -(\overline d+\eps). 
	\ee 
	As $\eps$ is arbitrary, we arrive at~\eqref{eq:Lliminfd}. 
	
	We finally prove~\eqref{eq:Llimdstar}. It remains to prove an upper bound, as~\eqref{eq:Lliminfd} now holds with $\overline d=d^*$. For $d^*=0$, an upper bound readily follows from the fact that $\log\P{L>t}\leq 0$. For $d^*>0$, the argument that leads to the upper bound in~\eqref{eq:Llimsup} is still valid when applying the Chernoff bound with a parameter $\theta\in(0,d^*-\eps)$ and any $\eps\in(0,  d^* )$, in the sense that there exists  $I\in\N_0$ such that $d(i)>d^*-\eps$ for all $i\geq I$. The remainder of the argument can be repeated with this $I$ and by substituting $d^*-\eps$ for $R-\eps$. Hence, in this case, 
	\be \label{eq:limsupdstar}
	\limsup_{t\to\infty}\frac1t \log \P{L>t}\leq- d^*.
	\ee 
	Combined with the lower bound in~\eqref{eq:Lliminfd}, this proves~\eqref{eq:Llimdstar} and concludes the proof.
\end{proof} 

\begin{proof}[Proof of Lemma~\ref{lemma:lifetime2nd}]
	\textbf{Lower bound. } We fix $x>0$ and set $k:=\ceil{\varphi_1^{-1}(t-x\cK(t))}$. If $\rho_2$ tends to infinity, by Lemma~\ref{lemma:Dtail},
	\be \label{eq:Llb}
	\P{L>t}\geq \P{D\geq k}\P{S_k>t}= \e^{-\rho_1(k)-(\frac12+o(1))\rho_2(k)}\P{S_k>t}. 
	\ee 
	In case $\rho_2$ converges instead, then $\rho_2(k)=o(\varphi_2(k))$ since $\varphi_2$ tends to infinity. Hence, we can use the lower bound in~\eqref{eq:Llb} with $\rho_2(k)$ replaced by $o(\varphi_2(k))$. Now, by the choice of $k$ and using Assumption~\hyperref[ass:K]{$\cK$} and $(a)$ in Lemma~\ref{lemma:func}, for any $\eps>0$ and all $t$ sufficiently large, 
	\be 
	t\leq \varphi_1(k)+x\cK(t)=\varphi_1(k)+x\varphi_2(\varphi_1^{-1}(t))\leq \varphi_1(k)+(x+\eps/3)\varphi_2(k).
	\ee 
	Hence,
	\be \label{eq:Skbound}
	\P{S_k>t}\geq \P{S_k>\varphi_1(k)+(x+\tfrac\eps3)\varphi_2(k)}\geq \exp\big(-(\tfrac12x^2+\eps)\varphi_2(k)\big),
	\ee 
	where the last step follows from Lemma~\ref{lemma:mdp} for $\eps$ small and $k$ sufficiently large. Since $d$ converges to $d^*$, we have $\rho_2(k)=( (d^*)^2+o(1))\varphi_2(k)$ by $(c)$ in Lemma~\ref{lemma:func}. Combined with~\eqref{eq:Skbound} in~\eqref{eq:Llb}, we obtain the lower bound
	\be 
	\P{L>t}\geq \exp\big(-\rho_1(k)-\tfrac12 (x^2+(d^*)^2+2\eps+o(1))\varphi_2(k)\big).
	\ee 
	We now write $\rho_1(k)=d^*\varphi_1(k)+\alpha(k)$ and use that, since $k=\ceil{\varphi_1^{-1}(t-x\cK(t))}$, it follows that $|\alpha(k)-\cK_\alpha(t)|=o(\cK(t))$ by Assumption~\hyperref[ass:Kalpha]{$\cK_\alpha$} and that $\cK(t)=(1+o(1))\varphi_2(k)$ by Assumption~\hyperref[ass:K]{$\cK$} and $(a)$ in Lemma~\ref{lemma:func}. This yields, by the choice of $k$,
	\be \ba 
	\exp{}&\big(-d^*\varphi_1(k)-\cK_\alpha(t)-\tfrac12 (x^2+(d^*)^2+2\eps+o(1))\cK(t)\big)\\ 
	&= \exp\big(-d^*t- \cK_\alpha(t)-\big[\tfrac12(x-d^*)^2+\eps+o(1)\big]\cK(t)\big)
	\ea \ee 
	We maximise the lower bound by choosing $x=d^*$, so that we arrive at 		
	\be \label{eq:Ldzero2nd}
	\liminf_{t\to\infty} \frac{\log(\P{L>t})+d^*t+\cK_\alpha(t)}{\cK(t)}\geq -\eps. 
	\ee 
	As $\eps$ is arbitrarily, this yields the desired lower bound. 
	
	\textbf{Upper bound. } We fix $\eps,\zeta>0$ small and define
	\be 
	M:=\Big\lceil \frac1\zeta \Big(\frac{(1-\eps)t}{\cK(t)}\Big)\Big\rceil, \quad x_i=i\zeta, \quad \text{and}\quad k_i:=\lceil \varphi_1^{-1}(t-x_i\cK(t))\rceil, 
	\ee  
	where $i\in\{0,1,\ldots, M\}$. Then, we partition $[0,\infty)$ into the intervals $[k_{i+1},k_i)$ for $i\in\{0,\ldots ,M\}$ and $[k_0,\infty)$ (with $k_{M+1}:=0$) to obtain the upper bound
	\be \label{eq:Lubsplit}
	\P{L>t}\leq \P{D\geq k_0}+\P{S_{k_M}>t}+\sum_{i=0}^{M-1}\P{D\geq k_{i+1},S_{k_i}>t}.
	\ee 
	We bound each term separately. 
	
	\textbf{First term. } By using Lemma~\ref{lemma:Dtail}, we arrive at 
	\be 
	\P{D\geq k_0}\leq \e^{-\rho_1(k_0)-\tfrac12\rho_2(k_0)}. 
	\ee 
	We use that $\rho_1$ and $\rho_2$ are increasing, the definition of $\alpha$ in~\eqref{eq:alpha}, and the fact that $d$ converges to $d^*$, so that $\rho_2(k)=((d^*)^2+o(1))\varphi_2(k)$ by $(c)$ in Lemma~\ref{lemma:func}, to bound the exponent from above by 
	\be \label{eq:firstbound}
	\exp\big(-\rho_1(\varphi_1^{-1}(t))-\tfrac12\rho_2(\varphi_1^{-1}(t))\big)=\exp\big(-d^*t -\cK_\alpha(t)-(\tfrac12(d^*)^2+o(1))\cK(t)\big).
	\ee 
	\textbf{Second term. } A Chernoff bound with $\theta\in(d^*,R)$ such that $\theta(1-\eps)>d^*$ (which is possible for $\eps$ sufficiently small)  combined with the inequality $\log(1+x)\leq x-x^2/2+x^3/3$ yields
	\be \ba
	\P{S_{k_M}\!>t}&\leq \e^{-\theta t}\prod_{i=0}^{k_M-1}\frac{b(i)+d(i)}{b(i)+d(i)-\theta}\\
	&\leq \exp\bigg(\!\!-\theta t+\!\!\sum_{i=0}^{k_M-1}\!\bigg[\frac{\theta}{b(i)+d(i)-\theta}-\frac{1}{2}\Big(\frac{\theta}{b(i)+d(i)-\theta}\Big)^2+\frac{1}{3}\Big(\frac{\theta}{b(i)+d(i)-\theta}\Big)^3\bigg]\bigg).
	\ea \ee 
	Since
	\be \ba \label{eq:sumineq}
	\sum_{i=0}^{k_M-1}{}&\frac{\theta}{b(i)+d(i)-\theta}-\frac{\theta^2}{2}\sum_{i=0}^{k_M-1}\Big(\frac{1}{b(i)+d(i)-\theta}\Big)^2\\ 
	={}&\theta \varphi_1(k)+\frac{\theta^2}{2}\varphi_2(k)-\frac{\theta^4}{2}\sum_{i=0}^{k_M-1}\frac{1}{(b(i)+d(i))^2(b(i)+d(i)-\theta)^2}\leq \theta\varphi_1(k)+\frac{\theta^2}{2}\varphi_2(k), 
	\ea \ee 
	we arrive at the upper bound
	\be 
	\P{S_{k_M}>t}\leq \exp\big(-\theta t +\theta \varphi_1(k_M)+(\tfrac12\theta^2+o(1))\varphi_2(k_M)\big). 
	\ee 
	By the choice of $k_M\leq \varphi_1^{-1}(\eps t)\leq\varphi_1^{-1}(t)$ and the fact that $\varphi_2=o(\varphi_1)$ since $b$ diverges, and hence $\cK(t)=o(t)$, we arrive at the upper bound
	\be \label{eq:secondbound}
	\exp\big(-((1-\eps)\theta+o(1))t\big). 
	\ee 
	Since $\cK_\alpha(t)=o(t)$ by $(d)$ in Lemma~\ref{lemma:func}, it follows that this lower bound is much smaller than $\exp(-d^*t-\cK_\alpha(t))$.
	
	\textbf{Summands. } Combining the bounds of the first and second term in each of the summands in~\eqref{eq:Lubsplit} (with again $\theta\in(d^*,R)$), we arrive at 
	\be 
	\P{D\geq k_{i+1}, S_{k_i}>t}\leq \exp\big(-\rho_1(k_{i+1})-\tfrac12\rho_2(k_{i+1})-\theta t+\theta \varphi_1(k_i)+(\tfrac12 \theta^2+o(1))\varphi_2(k_i)\big).
	\ee 
	Again using that $\rho_1$ and $\rho_2$ are increasing and that $\varphi_1(k_i)=t-x_i\cK(t)+o(1)$ as $b$ diverges and $\varphi_2(k_{i+1})=(1+o(1))\varphi_2(k_i)=(1+o(1))\cK(t-x_i\cK(t))$ by Assumption~\hyperref[ass:K]{$\cK$} and $(a)$ in Lemma~\ref{lemma:func}, where the $o(1)$ is uniform in $i$, we obtain the upper bound 
	\be 
	\exp\big(-d^*t -\cK_\alpha(t-x_{i+1}\cK(t))+\big[\tfrac12 \theta^2-\tfrac12 (d^*)^2+o(1)\big]\cK(t-x_i\cK(t))-(\theta-d^*)x_i\cK(t)+\zeta d^*\cK(t)\big).
	\ee 
	We distinguish between two cases, for which we first introduce the following functions. We define 
	\be 
	\wt\alpha(k)=\sum_{i=0}^{k-1}\frac{|d(i)-d^*|}{b(i)+d(i)}\qquad \text{for }k\in\N_0,
	\ee 
	and extend $\wt\alpha$ to $\R_+$ by linear interpolation. It is clear that $\wt \alpha(x)\geq\alpha(x)$ for all $x\geq 0$. Moreover, we let $\cK_{\wt\alpha}(t):=\wt\alpha(\varphi_1^{-1}(t))$ and $q(t):=C_1(\cK_{\wt\alpha}(t)+\cK(t))/t$, where $C_1>2(\zeta(\theta-d^*)(1-\eps))^{-1}$ is a constant. We note that $\lim_{t\to\infty}q(t)=0$ since $d$ converges to $d^*$ and $b$ tends to infinity. We then have the following two cases:
	\be 
	(a)\quad i< \frac{(1-\eps)t}{\cK(t)}q(t),\qquad (b)\quad \frac{(1-\eps)t}{\cK(t)}q(t)\leq i\leq M.
	\ee 
	\textbf{Case (a). } As $\theta\in(d^*,R)$, we use that $x_i\geq 0$ and that $\cK$ is increasing to obtain the upper bound
	\be 
	\exp\big(-d^*t -\cK_\alpha(t-x_{i+1}\cK(t))+\big[\tfrac12 \theta^2-\tfrac12 (d^*)^2+o(1)\big]\cK(t)-(\theta-d^*)x_i\cK(t)+\zeta d^*\cK(t)\big).
	\ee 
	By the bound on $i$, it follows that $x_i\cK(t)=o(t)$. Hence, by  $(d)$ in Lemma~\ref{lemma:func} we have $\cK_\alpha(t-x_{i+1}\cK(t))=\cK_\alpha(t)+o(x_i\cK(t))$, where we use that $x_{i+1}=x_i+\zeta$ in the final step. As a result, we can rewrite the upper bound to obtain 
	\be 
	\exp\big(-d^*t -\cK_\alpha(t)+\big[\tfrac12 \theta^2-\tfrac12 (d^*)^2+o(1)\big]\cK(t)-(\theta-d^*+o(1))x_i\cK(t)+\zeta d^*\cK(t)\big).
	\ee 
	Finally, as $x_i$ grows linearly in $i$, we have for some large constant $C_2>0$, 
	\be 
	\sum_{i=0}^{\lfloor(1-\eps)tq(t)/\cK(t)\rfloor}\!\!\!\!\!\!\!\!\P{D\geq k_{i+1},S_{k+i}>t}\leq C_2\exp\big(-d^*t -\cK_\alpha(t)+\big[\tfrac12 \theta^2-\tfrac12 (d^*)^2+\zeta d^*+o(1)\big]\cK(t)\big).
	\ee 
	\textbf{Case (b). } With $\theta\in(d^*,R)$, since $-\cK_\alpha(t-x_{i+1}\cK(t))\leq \cK_{\wt\alpha}(t)$, and as $x_i\geq 0$, 
	\be \ba 
	\exp{}&\big(-d^*t -\cK_\alpha(t-x_{i+1}\cK(t))+\big[\tfrac12 \theta^2-\tfrac12 (d^*)^2+o(1)\big]\cK(t-x_i\cK(t))-(\theta-d^*)x_i\cK(t)+\zeta d^*\cK(t)\big)\\ 
	&\leq \exp\big(-d^*t +\cK_{\wt\alpha}(t)+\big[\tfrac12 \theta^2-\tfrac12 (d^*)^2-(\theta-d^*)x_i+\zeta d^*+o(1)\big]\cK(t)\big)
	\ea \ee 
	As a result, since $x_i=i\zeta $ increases linearly in $i$, 
	\be \ba 
	\sum_{i=\lceil(1-\eps)tq(t)/\cK(t)\rceil}^{M-1}\!\!\!\!\!\!\!\!\!\!\!\!\!\!{}&\P{D\geq k_{i+1}, S_{k_i}>t}\\ 
	\leq{}& \exp\big(-d^*t +\cK_{\wt\alpha}(t)-(\theta-d^*)\zeta (1-\eps) q(t)t+\big[\tfrac12\theta^2-\tfrac12(d^*)^2+\zeta d^*+o(1)\big]\cK(t)\big).
	\ea\ee 
	We now use the specific choice of the function $q$. For any $\theta\in(d^*,R)$ and since $\cK_{\wt \alpha}(t)\geq \cK_\alpha(t)$ and $C_1>2(\zeta(\theta-d^*)(1-\eps))^{-1}$, we thus arrive at 
	\be 
	\sum_{i=\lceil(1-\eps)tq(t)/\cK(t)\rceil}^{M-1}\!\!\!\!\!\!\!\!\!\!\!\!\!\!\P{D\geq k_{i+1},S_{k_i}>t}\leq \exp\big(-d^*t -\cK_\alpha(t)+\big[\tfrac12\theta^2-\tfrac12(d^*)^2+\zeta d^*-2+o(1)\big]\cK(t)\big). 
	\ee 
	We combine this with the bound in Case $(a)$ and the bounds for first and second term in~\eqref{eq:firstbound} and~\eqref{eq:secondbound}, respectively. We observe that the main contribution of these four bounds comes from Case $(a)$. Using all four bounds in~\eqref{eq:Lubsplit}, we finally obtain 
	\be 
	\limsup_{t\to\infty} \frac{\log(\P{L>t})-d^*+\cK_\alpha(t)}{\cK(t)}\leq \frac12\big(\theta^2-(d^*)^2\big)+\zeta d^*. 
	\ee 
	The upper bound can be made arbitrarily small by choosing $\zeta$ close to zero and $\theta$ close to $d^*$. Together with the lower bound in~\eqref{eq:Ldzero2nd}, we arrive at the desired result.
\end{proof} 

\subsection{Growth-rate of the branching process}

In this section we provide some results for the growth-rate of the branching process $\bp$.  We require a precise understanding of the growth-rate of the branching process to be able to analyse the quantities $O_t^\cont$ and $I_t^\cont$ in later Sections. We introduce the following notation. First, recall that $\cS$ denotes the event that the branching process survives and that
\be \label{eq:Ps} 
\Ps{\cdot}:=\P{\cdot\,|\,\cS} \quad\text{and}\quad \mathbb E_{\cS}[\cdot]:=\E{\cdot\,|\,\cS}
\ee 
denote the conditional probability measure and its corresponding expected value, under the event $\cS$. For $0\leq s<t<\infty$, we further define 
\be 
\cB(s,t):=\{u\in \cU_\infty: \sigma_u\in [s,t]\}
\ee 
as the set of individuals born in the time interval $[s,t]$. Also, recall $\cA_t^\cont$ from~\eqref{eq:At} as the set of individuals alive at time $t$, and $\cR$ from~\eqref{eq:D} as the point process that governs the production of offspring of individuals. We have the following results regarding the growth-rate of  $\cA^\cont_t$ and $\cB(s,t)$. The first is similar to the well-known Kesten-Stigum theorem and is a direct consequence of results for general CMJ branching processes in~\cite{Don72,Ner81}; the second uses a similar approach as the proof of~\cite[Lemma $7.12$]{BanBha21}.

\begin{proposition}[Growth rate of branching process conditionally on survival]\label{prop:growth}
	Suppose that $b$ and $d$ are such that Assumptions~\eqref{ass:A1} and \eqref{ass:C1} are satisfied, and recall $\cS$ from~\eqref{eq:S} as the event that the process $\mathrm{BP}$ survives. There exists a non-negative random variable $W$, such that 
	\be \label{eq:Aconv}
	\e^{-\lambda^*\!t} |\cA^\cont_t|\toas W. 
	\ee 
	Moreover, with
	\be 
	\widehat\cR^{\lambda^*}(t):=\int_0^t \e^{-\lambda^* u}\,\cR(\dd u), 
	\ee 
	the following are equivalent:
	\begin{enumerate}
		\item $\E{ \widehat\cR^{\lambda^*}(\infty)\log^+ \widehat\cR^{\lambda^*}(\infty)}<\infty$, 
		\item $\E{W}>0$,
		\item $\E{\e^{-\lambda^*\!t}|\cA^\cont_t|}\to \E{W}$ as $t\to\infty$, 
		\item $W>0$ almost surely on $\cS$. 
	\end{enumerate}
\end{proposition}

Proposition~\ref{prop:growth} is a direct result of~\cite[Theorem $5.4$]{Ner81} and~\cite{Don72} (see also~\cite[Proposition $1.1$]{Ner81} for a condensed version). We have the following useful corollary. 

\begin{corollary}\label{cor:growth}
	Suppose that $b$ and $d$ are such that Assumptions~\eqref{ass:A1} and \eqref{ass:C1} are satisfied, and recall the event $\cS$ that the process $\mathrm{BP}$ survives from~\eqref{eq:S}. Let $r,s,u=r(t),s(t),u(t)\geq0$ be such that $s(t)-r(t)\to\infty$ and $u(t)\to\infty$ as $t\to\infty$. Then,  
	\begin{align}  
		\lim_{t\to\infty}\Ps{|\cB(r,s)|\leq \e^{\lambda^*\!s-u}}&=0,\label{eq:lbB}
		\intertext{and} 
		\lim_{t\to\infty}\P{|\cB(r,s)|\geq \e^{\lambda^*\!s+u}}&=0.\label{eq:ubB}
		\intertext{Finally, }
		\lim_{M\to\infty} \P{\sup_{t\geq 0} |\cB(0,t)|\e^{-\lambda^*\!t}\geq M}&=0,\label{eq:supB} 
		\intertext{and} 
		\lim_{M\to\infty} \Ps{\inf_{t\geq 0} |\cB(0,t)|\e^{-\lambda^*\!t}\leq 1/M}&=0.\label{eq:infB}
	\end{align}
\end{corollary}

\begin{proof} 
	We observe that the number of births between time $r$ and time $s$ equals the number of alive individuals at time $s$ minus the number of individuals alive at time $r$ plus the number of individuals born between time $r$ and $s$ that have died before time $s$. That is,
	\be
	|\cB(r,s)|=|\cA_s^\cont|-|\cA_r^\cont|+|\{u\in \cU_\infty: \sigma_u\in [r,s], L_u<s-\sigma_u\}|\geq |\cA_s^\cont|-|\cA_r^\cont|. 
	\ee 
	Consequentially,
	\be \ba 
	\Ps{|\cB(r,s)|\leq \e^{\lambda^*\!s-u}}&\leq \Ps{|\cA_s^\cont|-|\cA_r^\cont|\leq \e^{\lambda^*\!s-u}}\\ 
	&=\Ps{|\cA_s^\cont|\e^{-\lambda^*\!s} -|\cA_r^\cont|\e^{-\lambda^*\!r}\e^{-\lambda^*(s-r)}\leq \e^{-u}}.
	\ea \ee 
	By~\eqref{eq:Aconv}, it follows that $|\cA_s^\cont|\e^{-\lambda^*\!s}$ converges almost surely to $W$, which is $\mathbb P_\cS$-almost surely positive. Similarly, $|\cA_r^\cont|\e^{-\lambda^*\!r}$  is a tight sequence of random variables. Hence, since $s-r$ tends to infinity with $t$, it follows that $|\cA_r^\cont|\e^{-\lambda^*\!r}\e^{-\lambda^*(s-r)}$ converges to zero $\mathbb P_\cS$-almost surely. As $u$ tends to infinity with $t$, the right-hand side of the event in the probability tends to zero with $t$ and we arrive at~\eqref{eq:lbB}. 
	
	In a similar manner, we observe that $|\cB(r,s)|\leq |\cB(0,s)|$. It follows from~\cite[Theorem $5.4$]{Ner81} that $|\cB(0,s)|\e^{-\lambda^*\!s}$ converges almost surely (or, equivalently, from~\eqref{eq:charlim}, since $|\cB(0,s)|=Z^{\mathrm b}_s$ and $|\cA^\cont_s|=Z^{\mathrm a}_s$), so that 
	\be 
	\lim_{t\to\infty}\P{|\cB(r,s)|\geq \e^{\lambda^*\!s+u}}\leq \lim_{t\to\infty}\P{|\cB(0,s)|\e^{-\lambda^*\!s}\geq \e^{u}}=0,
	\ee 
	since $u$  tends to infinity with $t$. 
	
	Finally, we prove~\eqref{eq:supB} and~\eqref{eq:infB}. Fix $\eps>0$. First, by a union bound, 
	\be 
	\P{\sup_{t\geq 0} |\cB(0,t)|\e^{-\lambda^*\!t}\geq M}\leq \P{\sup_{t\in[0,\log(M)/(2\lambda^*)]}\!\!\!\!\!\!\!\! |\cB(0,t)| \geq M}+\P{\sup_{t\geq \log(M)/(2\lambda^*)}\!\!\!\!\!\!\!\! |\cB(0,t)|\e^{-\lambda^*\!t}\geq M},
	\ee
	where we bound the exponential term in the first probability on the right-hand side from above by one. As $|\cB(0,t)|$ is increasing in $t$, we can bound this term from above further by 
	\be 
	\P{|\cB(0,\log(M)/(2\lambda^*))|\geq M}=\P{|\cB(0,\log(M)/(2\lambda^*))|\geq \e^{2\lambda^*\log(M)/(2\lambda^*)}}. 
	\ee 
	It thus follows from~\eqref{eq:ubB} (with $r=0$, $s=\log(M)/(2\lambda^*)$, and $u=\log(M)/2$) that the right-hand side is at most $\eps/3$ for all $M$ large. For the second probability, we first take $M_0=M_0(\eps)>0$ large enough so that $\P{W\geq M_0}<\eps/2$. Then, we take $M\geq M_0$ large such that
	\be \ba 
	\P{\sup_{t\geq \log(M)/(2\lambda^*)}\!\!\!\!\!\!\!\! |\cB(0,t)|\e^{-\lambda^*\!t}\geq M}&\leq \P{\sup_{t\geq \log(M)/(2\lambda^*)}\!\!\!\!\!\!\!\! |\cB(0,t)|\e^{-\lambda^*\!t}\geq M_0}\\ 
	&\leq \P{W\geq M_0}+\eps/3<2\eps/3. 
	\ea \ee 
	Here, the second step follows from the fact that $\inf_{t\geq \log(M)/(2\lambda^*)}|\cB(0,t)|\e^{-\lambda^* \! t}\toas W$ as $M$ tends to infinity. As $\eps$ is arbitrary, this yields the desired result and proves~\eqref{eq:supB}. 
	
	The proof of~\eqref{eq:infB} uses a similar approach, where we now use that $\cB(0,t)\geq 1$ for all $t\geq 0$, so that $|\cB(0,s)|\e^{-\lambda^* \!s}\geq \e^{-\lambda^* T}$ for all $s\leq T$. As a result, with $T=\log(M)/\lambda^*$,
	\be
	\Ps{\inf_{t\geq 0} |\cB(0,t)|\e^{-\lambda^*\!t}\leq 1/M}=\Ps{\inf_{t\geq \log(M)/\lambda^*}\!\!\!\!\! |\cB(0,t)|\e^{-\lambda^*\!t}\leq 1/M}.
	\ee  
	Then, fix $\eps>0$ and take $M_0=M_0(\eps)>0$ large so that $\Ps{W\leq 1/M_0}\leq \eps/2$. We note that this is possible, conditionally on survival, since then $W>0$ almost surely by $(d)$. Now, take $M\geq M_0$ large so that
	\be \ba 
	\Ps{\inf_{t\geq \log(M)/\lambda^*}\!\!\!\!\! |\cB(0,t)|\e^{-\lambda^*\!t}\leq 1/M}&\leq \Ps{\inf_{t\geq \log(M)/\lambda^*}\!\!\!\!\! |\cB(0,t)|\e^{-\lambda^*\!t}\leq 1/M_0}\\ 
	&\leq \Ps{W\leq 1/M_0}+\eps/2\leq \eps. 
	\ea \ee  
	Here, the second step follows from the fact that $\inf_{t\geq \log(M)/\lambda^*} |\cB(0,t)|\e^{-\lambda^*\!t}\xrightarrow{\mathbb P_\cS-\mathrm{a.s.}} W$ as $M$ to infinity. As $\eps$ is arbitrary, this proves~\eqref{eq:infB} and concludes the proof.
\end{proof} 

\subsection{Non-survival of old individuals}\label{sec:oldproof}

We conclude this section by combining the results from the previous sub-sections to prove asymptotic results for $O_t^\cont$. In essence, we show that all individuals that are `too old' have all died by time $t$ with high probability and that there exists `young enough' individuals that survive up to time $t$. To formalise this intuition, we introduce the following notation. Fix $d^*\geq 0, \lambda^*>0$, and $p>0$, and define 
\be \label{eq:Ft}
F_{p,t}:=\big[t-p\cK\big(\tfrac{\lambda^*}{\lambda^*+d^*}t\big),t+p\cK\big(\tfrac{\lambda^*}{\lambda^*+d^*}t\big)\big].
\ee 
We observe that the behaviour of $O_t^\cont$ is different in the `rich are old' and the `rich die young' regimes, since the lifetime of an individual has a different (asymptotic) distribution in these regimes. This is made precise in the following result.

\begin{proposition}\label{prop:oldage}
	Suppose that the sequences $b$ and $d$ are such that Assumptions~\eqref{ass:A1}, \eqref{ass:A2}, and~\eqref{ass:C1} are satisfied. Recall $R$ from~\eqref{eq:R} and the event $\cS$ from~\eqref{eq:S} that the branching process $\bp$ survives. Assume that $\liminf_{i\to\infty} d(i)> R$. Then, there exists $K_0>0$ such that for any $K>K_0$ there exists $p>0$ sufficiently small so that 
	\be \label{eq:ROt}
	\lim_{t\to\infty}\Ps{\forall s\in[t-p\log t, t+p\log t]: \Big| O^{\cont}_s-\frac{R}{\lambda^*+R}t\Big|\leq K\log t}=1.
	\ee 		
	Now, assume that $b$ diverges and that $\lim_{i\to\infty} d(i)=d^*\in[0,R)$. Also, suppose that Assumptions~\hyperref[ass:K]{$\cK$} and~\hyperref[ass:Kalpha]{$\cK_\alpha$} and Assumption~\eqref{ass:varphi2} are satisfied. Then, for any $\eps>0$ there exists $p>0$ such that 
	\be \label{eq:Otconv}
	\lim_{t\to\infty}\Ps{\forall s\in F_{p,t}:\Big|\cK\big(\tfrac{\lambda^*}{\lambda^*+d^*}t\big)^{-1}\big(O_s^\cont-\tfrac{d^*}{\lambda^*+d^*}t-\tfrac{1}{\lambda^*+d^*}\cK_\alpha(r(t))\big)\Big|<\eps}=1.
	\ee  		
	Finally, if Assumption~\eqref{ass:A2} is not satisfied, 
	\be \label{eq:Otas}
	O_t^\cont\xrightarrow{\mathbb P_\cS-\mathrm{a.s.}} O', 
	\ee 
	for some almost surely finite random variable $O'$.
\end{proposition}

\begin{remark}\label{rem:strongerOt}
	Under the assumption that $\liminf_{i\to\infty}d(i)\geq R$ we can adapt the proof of~\eqref{eq:ROt} (and by using~\eqref{eq:LlimR} rather than~\eqref{eq:Lprecise}) to show that for any $\eps>0$ there exists $p>0$ such that 
	\be \label{eq:ROtweak}
	\lim_{t\to\infty}\Ps{\forall s\in[(1-p)t,(1+p)t]: \Big|\frac{O_s^\cont}{t}- \frac{ R}{\lambda^*+R}\Big|<\eps}=1.
	\ee 
	It yields a weaker bound on $O_s^\cont$, but with a weaker condition on $d$ and a larger range of $s$.
	\ensymboldremark
\end{remark} 
\begin{remark} \label{rem:improvetoas}
	When setting $s=t$ in~\eqref{eq:ROt} and~\eqref{eq:Otconv}, the quantity in the absolute values converges to zero in probability. Though in some cases where $\cK$ grows sufficiently fast it may be possible to strengthen the convergence in probability to almost sure convergence, we are unable to derive the stronger bounds due to the use of~\eqref{eq:lbB} and~\eqref{eq:ubB} in Corollary~\ref{cor:growth}. Stronger versions of these results would be required to improve to almost sure convergence.	\ensymboldremark
\end{remark} 

\begin{remark}
	Proposition~\ref{prop:oldage} provides the main step in proving the asymptotic behaviour of $O_n$, as presented in Theorems~\ref{thrm:asPA} through~\ref{thrm:biggerR}. It remains to translate the asymptotic behaviour of $O_t^\cont$ into the asymptotic behaviour of $O_n$, which is carried out in Section~\ref{sec:mainproofs}.
	\ensymboldremark
\end{remark}

Intuitively, an individual born at time $T<t$ survives up to time $t$ with probability $\P{L>t-T}$, independently of all other individuals. At the same time, the number of individuals born around time $T$ is roughly $\e^{\lambda^* T}$ by Proposition~\ref{prop:growth}. The proof uses first and second moment bounds to establish the optimal choice of $T$ such that no individuals born before $T$ are alive at time $t$, whereas many individuals born after $T$ are alive at time $t$. 

\begin{proof}
	We first prove~\eqref{eq:ROt}. Fix $K,p>0$ and define 
	\be\label{eq:ellu}
	\ell_t:=\frac{R}{\lambda^*+R}t-K\log t, \qquad u_t:=\frac{R}{\lambda^*+R}t+K\log t,\qquad \wt u_t:=\frac{R}{\lambda^*+R}t, 
	\ee 
	and
	\be 
	\underline r(t):=t-p\log t, \qquad  \overline r(t):=t+p\log t.
	\ee
	We then observe that $O_t^\cont$ is increasing in $t$, so that 
	\be \label{eq:Otsplit}
	\mathbb P_{\cS}\bigg(\Big|O_s^\cont - \frac{ R}{\lambda^*+R}t\Big|<K\log t \text{ for all }s\in[\underline r(t),\overline r(t)]\bigg)\geq 1-\Ps{O_{\underline r(t)}^\cont\leq \ell_t}-\Ps{O_{\overline r(t)}^\cont\geq u_t}.
	\ee 
	We thus aim to bound the probabilities on the right-hand side from  above, starting with the leftmost one. For any $\delta>0$, recalling that $\cA^\cont_t$ denotes the set of alive individuals at time $t$,
	\be \ba \label{eq:elltbound}
	\Ps{O_{\underline r(t)}^\cont\leq \ell_t }={}&\Ps{\cB(0,\ell_t)\cap \cA_{\underline r(t)}\neq\emptyset }\\ 
	\leq{}& \frac{1}{\P{\cS}}\Big[\P{\cB(0,\ell_t)\cap\cA_{\underline r(t)}\neq\emptyset, |\cB(0,\ell_t)|\leq \e^{\lambda^*\ell_t+\delta \log t}}\\ 
	&\hphantom{\frac{1}{\P{\cS}}\ }+\P{|\cB(0,\ell_t)|\geq \e^{\lambda^*\ell_t+\delta \log t}}\Big].
	\ea \ee 
	The second probability in the brackets converges to zero by Corollary~\ref{cor:growth} (with $r=0, s=\ell_t$, and $u=\delta \log t$). Furthermore, lifetimes among individuals are i.i.d.\ and each individual born before time $\ell_t$ needs to live for at least $\underline r(t)-\ell_t$ time to be alive at time $\underline r(t)$. Hence, by a union bound and using the upper bound in~\eqref{eq:Lprecise} of Lemma~\ref{lemma:lifetime}, we arrive for some large constant $C>0$ at the upper bound
	\be \label{eq:LasympR}
	\P{\cS}^{-1}\e^{\lambda^*\ell_t+\delta \log t}\P{L\geq \underline r(t)-\ell_t}+o(1)\leq \e^{\lambda^*\ell_t+\delta \log t-R(\underline r(t)-\ell_t)+C\log(\underline r(t)-\ell_t)}+o(1).
	\ee 
	By the definition of $\ell_t$ and $\underline r(t)$, this equals
	\be 
	\exp\Big(\big[-(\lambda^*+R)K+\delta+Rp +C\big]\log t+\cO(1)\Big)+o(1). 
	\ee 
	Hence, for any $K$ larger than $K_0:=C(\lambda^*+R)^{-1}$, we can choose $\delta$ and $p$ sufficiently small so that the terms in the square brackets are negative. As a result, the upper bound in~\eqref{eq:LasympR} tends to zero with $t$ for these choices of $K, p$, and $\delta$.   
	
	In a similar way, we bound the second probability on the right-hand side of~\eqref{eq:Otsplit}. We recall $\wt u_t$ from~\eqref{eq:ellu} and write 
	\be \label{eq:utbound}
	\mathbb P_{\cS}(O_{\overline r(t)}^\cont\geq  u_t)\leq \mathbb P_{\cS}(\cB(\wt u_t,u_t)\cap \cA_{\overline r(t)}= \emptyset)\leq  \mathbb P_{\cS}\big(\forall v\in \cB(\wt u_t,u_t): L^{(v)}\leq \overline r(t)-\wt u_t\big).
	\ee 
	For ease of writing, we define the events 
	\be 
	\cE_t:=\{\forall v\in \cB(\wt u_t,u_t): L^{(v)}\leq \overline r(t)-\wt u_t\}, \qquad \cD_t:=\{|\cB(\wt u_t,u_t)|>m\} \quad \text{for }m\in\N,t \geq0.
	\ee 
	We then write the right-hand side of~\eqref{eq:utbound} as 
	\be 
	\Ps{\cE_t}=\P{\cS}^{-1}\E{\ind_\cS\ind_{\cE_t}}\leq \P{\cS}^{-1}\E{\ind_\cS\ind_{\cE_t}\ind_{\cD_t}}+\P{\cS}^{-1}\E{\ind_\cS\ind_{\cD_t^c}}. 
	\ee 
	The second term on the right-hand side equals the probability of the event $\cD_t^c$ under the conditional probability measure $\mathbb P_\cS$. For the first term, we condition on two things. First, we condition $\cF_{\wt u_t}$, the $\sigma$-algebra generated by the branching process up to time $\wt u_t$. Second, we condition on $|\cB(\wt u_t, u_t)|$, the number of individuals born in the branching process in the interval $(\wt u_t,u_t)$. We can thus write
	\be \ba 
	\P{\cS}^{-1}\E{\ind_\cS\ind_{\cE_t}\ind_{\cD_t}}&=\P{\cS}^{-1}\E{\E{\ind_\cS\ind_{\cE_t}\ind_{\cD_t}\,|\, \cF_{\wt u_t}, |\cB(\wt u_t,u_t)|}}\\ 
	&\leq \P{\cS}^{-1}\E{\ind_{\cD_t}\E{\ind_{\cE_t}\,|\, \cF_{\wt u_t}, |\cB(\wt u_t,u_t)|}}\\ 
	&=\P{\cS}^{-1}\E{\ind_{\cD_t}\P{\cE_t\,|\, \cF_{\wt u_t}, |\cB(\wt u_t,u_t)|}}.
	\ea \ee 
	Conditionally on $|\cB(\wt u_t, u_t)|$ and upon the event $\cD_t$, we can use the independence of the lifetimes of distinct individuals to bound the terms in the expected value from above by 
	\be 
	\ind_{\cD_t}\P{\cE_t\,|\, \cF_{\wt u_t}, |\cB(\wt u_t,u_t)|}\leq \ind_{\cD_t}\big(1-\P{L\geq \overline r(t)-\wt u_t}\big)^m\leq \exp\big(-m\P{L\geq \overline r(t)-\wt u_t}\big),
	\ee 
	where we have omitted the indicator random variable and used that $1-x\leq \e^{-x}$ for $x\in\R$ in the last step. Combining all of the above in~\eqref{eq:utbound}, we thus arrive at 
	\be 
	\Ps{O_{\overline r(t)}^\cont\geq u_t}\leq \P{\cS}^{-1}\exp\big(-m\P{L\geq \overline r(t)-\wt u_t}\big)+\Ps{|\cB(\wt u_t,u_t)|\leq m}.
	\ee 
	We choose $m=\exp(\lambda^* u_t-\delta \log t)$ and apply Corollary~\ref{cor:growth} (with $r=\wt u_t$, $s=u_t$, and $u=\delta \log t$) to the second term on the right-hand side. This yields that the second term tends to zero with $t$. For the first term we arrive at 
	\be \label{eq:utbound2}
	\P{\cS}^{-1}\exp\big(-\P{L\geq \overline r(t)-\wt u_t}\e^{\lambda^*u_t-\delta \log t}\big).
	\ee 
	By the lower bound in~\eqref{eq:Lprecise} of Lemma~\ref{lemma:lifetime} and the definition of $u_t$ and $\wt u_t$, we arrive at
	\be 
	\exp\Big(-(\exp\Big(\big[\lambda^*K-\delta -Rp\big]\log t+\cO(1)\Big)\Big).
	\ee 
	For any $K>0$ (and thus for any $K>K_0$, in particular) we can choose $\delta$ and $p$  sufficiently small so that the upper bound tends to zero. This yields that the left-hand side of~\eqref{eq:Otsplit} converges to one as $t$ tends to infinity, so that the desired result follows.
	
	We then prove~\eqref{eq:Otconv}. The proof follows a similar argument as the above, but with a refined result. To this end, we use  Lemma~\ref{lemma:lifetime2nd} and Assumption~\hyperref[ass:Kalpha]{$\cK_\alpha$}. From the latter we obtain $r=r(t)$ such that 
	\be \label{eq:sass}
	\cK_\alpha(r(t))-\cK_\alpha\big(\tfrac{\lambda^*}{\lambda^*+d^*}t-\tfrac{1}{\lambda^*+d^*}\cK_\alpha(r(t))\big)=o(\cK(r(t))).
	\ee 
	We fix $\eps>0$ and set 
	\be \ba 
	\ell_t&:=\frac{d^*}{\lambda^*+d^*}t+\tfrac{1}{\lambda^*+d^*}\cK_\alpha(r(t))-\eps\cK(r(t)),& u_t&:= \frac{d^*}{\lambda^*+d^*}t+\tfrac{1}{\lambda^*+d^*}\cK_\alpha(r(t))+\eps\cK(r(t)),\\ 
	\underline r(t){}&:=t-p\cK\big(\tfrac{\lambda^*}{\lambda^*+d^*}t\big), &\overline r(t)&:=t+p\cK\big(\tfrac{\lambda^*}{\lambda^*+d^*}t\big).
	\ea \ee 
	We then use a similar bound as in~\eqref{eq:Otsplit}, to obtain 
	\be \ba \label{eq:Ot2ndsplit}
	\mathbb P_{\cS}\Big({}&\Big|\cK\big(\tfrac{\lambda^*}{\lambda^*+d^*}t\big)^{-1}\big(O_s^\cont-\tfrac{d^*}{\lambda^*+d^*}t-\tfrac{1}{\lambda^*+d^*}\cK_\alpha(r(t))\big)\Big|<\eps, \text{ for all }s\in [\underline r(t), \overline r(t)]\Big)\\ 
	&\geq 1-\Ps{O_{\underline r(t)}^\cont\leq \ell_t}-\Ps{O_{\overline r(t)}^\cont\geq u_t}.
	\ea \ee 
	We again bound the probabilities on the right-hand side from above, and start with the leftmost one. We use a similar bound as in~\eqref{eq:elltbound}, but with the event 
	\be 
	\cE_t:=\big\{|\cB(0,\ell_t)|\leq \exp\big(\lambda^*\ell_t+\delta\cK(r(t))\big)\big\},
	\ee 
	where $\delta>0$ is small. We then have the upper bound
	\be 
	\Ps{O_{\underline r(t)}^\cont\leq \ell_t}\leq \Ps{\cB(0,\ell_t)\cap \cA_{\underline r(t)}\neq\varnothing}\leq  \frac{1}{\P{\cS}}\big[\P{\{\cB(0,\ell_t)\cap\cA_{\underline r(t)}\neq\emptyset\}\cap \cE_t}+\P{\cE_t^c}\big].
	\ee 
	Again, by Corollary~\ref{cor:growth} (with $r=0$, $s=\ell_t$, and $u=\delta \cK(r(t))$) and as $\cK$ and $s$ diverge, the term $\P{\cE_t^c}$ tends to zero with $t$. We can thus focus on the first term in the square brackets. With a similar approach as in~\eqref{eq:LasympR} and with $\xi>0$ fixed, but using Lemma~\ref{lemma:lifetime2nd} instead, we bound this term from above by
	\be \ba 
	\exp\big(\lambda^*\ell_t{}&+\delta \cK(r(t))-d^*(\underline r(t)-\ell_t)-\cK_\alpha(\underline r(t)-\ell_t)+\xi\cK(\underline r(t)-\ell_t)\big)\\ 
	=\exp\Big({}&\cK_\alpha(r(t))-\cK_\alpha\big(\tfrac{\lambda^*}{\lambda^*+d^*}t-\tfrac{1}{\lambda^*+d^*}\cK_\alpha(r(t))+\eps\cK(r(t))-p\cK\big(\tfrac{\lambda^*}{\lambda^*+d^*}t\big)\big)\\ 
	&+(\delta -(\lambda^*+d^*)\eps)\cK(r(t))+\xi \cK\big(\tfrac{\lambda^*}{\lambda^*+d^*}t-\tfrac{1}{\lambda^*+d^*}\cK_\alpha(r(t))+\eps\cK(r(t))-p\cK\big(\tfrac{\lambda^*}{\lambda^*+d^*}t\big)\big)\\ 
	&+d^*p\cK\big(\tfrac{\lambda^*}{\lambda^*+d^*}t\big)\Big).
	\ea \ee 
	We combine~\eqref{eq:sass} with $(d)$ in Lemma~\ref{lemma:func} to obtain that $r(t)=\frac{\lambda^*}{\lambda^*+d^*}t+o(t)$. Further, as $b$ diverges, it follows that $\cK(t)=o(t)$ by $(a)$ in Lemma~\ref{lemma:func}. We then apply Assumption~\hyperref[ass:K]{$\cK$} to write 
	\be \label{eq:Kapprox}
	\cK\big(\tfrac{\lambda^*}{\lambda^*+d^*}t-\tfrac{1}{\lambda^*+d^*}\cK_\alpha(r(t))+\eps\cK(r(t))-p\cK\big(\tfrac{\lambda^*}{\lambda^*+d^*}t\big)\big)=(1+o(1))\cK\big(\tfrac{\lambda^*}{\lambda^*+d^*}t\big), 
	\ee 
	and
	\be \label{eq:Kalphaapprox}
	\cK_\alpha(r(t))-\cK_\alpha\big(\tfrac{\lambda^*}{\lambda^*+d^*}t-\tfrac{1}{\lambda^*+d^*}\cK_\alpha(r(t))\big)=o\big(\cK(r(t))\big)=o\big(\cK\big(\tfrac{\lambda^*}{\lambda^*+d^*}t\big)\big).
	\ee 
	We thus arrive at
	\be 
	\exp\Big((\delta+\xi+d^*p-(\lambda^*+d^*)\eps+o(1))\cK\big(\tfrac{\lambda^*}{\lambda^*+d^*}t\big)\Big).
	\ee 
	By choosing $\delta$, $p$, and $\xi$ sufficiently small with respect to $\eps$, the upper bound tends to zero with $t$.
	
	We then bound the second probability on the right-hand side of~\eqref{eq:Ot2ndsplit}.  For ease of writing, we set 
	\be 
	\wt u_t:=\frac{d^*}{\lambda^*+d^*}t+\tfrac{1}{\lambda^*+d^*}\cK_\alpha(r(t)).
	\ee 
	Then, using the event 
	\be
	\wt\cE_t:=\big\{|\cB(\wt u_t,u_t)|\geq \exp\big(\lambda^* u_t-\delta \cK(r(t))\big)\big\},
	\ee 
	with $\delta>0$ small, we follow the same steps as in~\eqref{eq:utbound} through~\eqref{eq:utbound2} to arrive at 
	\be \ba \label{eq:Butbound}
	\Ps{O_{\overline r(t)}^\cont\geq u_t}&\leq \Ps{\cB\big(\wt u_t,u_t\big)\cap \cA_{\overline r(t)}=\varnothing}\\ 
	&\leq \P{\cS}^{-1}\exp\big(-\P{L\geq \overline r(t)-\wt u_t}\exp\big(\lambda^*u_t-\delta \cK(r(t))\big)\big)+\Ps{\wt\cE_t^c}.
	\ea\ee 
	Again, by Corollary~\ref{cor:growth} (with $r=\wt u_t$, $s=u_t$, and $u=\delta \cK(r(t))$), the final term tends to zero. By using Lemma~\ref{lemma:lifetime2nd}, we bound the argument of the exponential term from above, with $\xi>0$, by 
	\be\ba
	-{}&\exp\Big(\lambda^*u_t-\delta \cK(r(t))-d^*(\overline r(t)-\wt u_t)-\cK_\alpha(\overline r(t)-\wt u_t)-\xi \cK(\overline r(t)-\wt u_t)\Big)\\ 
	&=-\exp\Big(\cK_\alpha(r(t))-\cK_\alpha\big(\tfrac{\lambda^*}{\lambda^*+d^*}t-\tfrac{1}{\lambda^*+d^*}\cK_\alpha(r(t))+p\cK\big(\tfrac{\lambda^*}{\lambda^*+d^*}t\big)\big)+(\eps(\lambda^*+d^*)-\delta)\cK(r(t))\\ 
	&\hphantom{=-\exp\Big(\,}-\xi\cK\big(\tfrac{\lambda^*}{\lambda^*+d^*}t-\tfrac{1}{\lambda^*+d^*}\cK_\alpha(r(t))+p\cK\big(\tfrac{\lambda^*}{\lambda^*+d^*}t\big)\big)-d^*p\cK\big(\tfrac{\lambda^*}{\lambda^*+d^*}t\big)\Big).
	\ea \ee 
	We then combine this with~\eqref{eq:Kapprox} and~\eqref{eq:Kalphaapprox} in~\eqref{eq:Butbound} to arrive at
	\be 
	\exp\Big(-\exp\big([\eps(\lambda^*+d^*)-\delta-\xi-d^*p+o(1)]\cK\big(\tfrac{\lambda^*}{\lambda^*+d^*}t\big)\big)\Big)+o(1), 
	\ee 
	which tends to zero with $t$ when we choose $\delta$, $p$, and $\xi$ sufficiently small with respect to $\eps$, since $\cK$ tends to infinity. We thus obtain that the left-hand side of~\eqref{eq:Ot2ndsplit} tends to one as $t$ tends to infinity, which yields the desired result. 
	
	Finally, we prove~\eqref{eq:Otas}. It follows from Lemma~\ref{lemma:Dlifedistr} that $\P{D=\infty}>0$ when Assumption~\eqref{ass:A2} is not satisfied. Furthermore, Corollary~\ref{cor:inflifetime} implies that $L$ is finite almost surely, conditionally on $\{D<\infty\}$, and that $L$ is infinite almost surely, conditionally on $\{D=\infty\}$. As a result, let $N$ denote the number of individuals born in the branching process, when the first individual $v$, such that $D^{(v)}=\infty$, is born. From the independence of $(D^{(v)})_{v\in\cU_\infty}$, it follows that $N$ is geometrically distributed with parameter $\P{D=\infty}>0$. Hence, $N$ is finite almost surely. Then, let $v_i$ be the $i^{\mathrm{th}}$ individual born in the branching process with birth-time $\sigma_i:=\sigma_{v_i}$, for $i\in[N]$. Note that, since $\sigma_i\neq\sigma_j$ almost surely and there is a finite number of births in any neighbourhood of $0$ almost surely, this enumeration is well-defined. Moreover, $L_{v_i}<\infty$ almost surely for all $i\in[N-1]$ and $L_{v_N}=\infty$, by definition. It thus follows that, conditionally on $\cS$,  
	\be 
	O_t^\cont=\sigma_N\quad \text{for all }t\geq \max_{i\in[N-1]}(\sigma_i+L_{v_i}). 
	\ee 
	Indeed, the event $\cS$ ensures that the individual $v_N$ is born before all other individuals $v_i$, with $i<N$, die. Note that the birth of an individual with an infinite lifetime then guarantees survival of the branching process. Further, once all individuals $v_1,\ldots, v_{N-1}$ have died, $v_N$ is the oldest alive individual and, as it lives forever, remains to be so for all time after. This yields~\eqref{eq:Otas} with $O'=\tau_N$ and thus completes the proof. 
\end{proof}

\section{The individual with the largest offspring in the `rich are old' regime}\label{sec:max}

In this section we analyse the growth-rate of $I_t^\cont$ in the `rich are old' regime. Recall $R$ from~\eqref{eq:R}, and $\varphi_2$ from~\eqref{eq:seqs}. We assume that $\varphi_2$ tends to infinity and that either Assumption~\eqref{ass:A2} is not met, or otherwise that $\lim_{i\to\infty}d(i)=d^*$ for some $d^*\in[0,\infty)$. Finally, we recall the set $F_{p,t}$ from~\eqref{eq:Ft}. We then have the following result.

\begin{proposition}\label{prop:It}
	Suppose that $b$ and $d$ are such that $b$ tends to infinity and such that Assumptions~\eqref{ass:A1}  and~\eqref{ass:C1} are satisfied. Further, suppose that one of the following holds: 
	\begin{enumerate}[label=(\arabic*)]
		\item\label{item:1} Assumption~\eqref{ass:A2} is not satisfied. We set $d^*:=0$.
		\item\label{item:2} Assumption~\eqref{ass:A2} is satisfied, $\lim_{i\to\infty}d(i)=d^*\in[0,\infty)$, and  Assumption~\hyperref[ass:Kalpha]{$\cK_\alpha$} is satisfied.
	\end{enumerate} 
	Finally, suppose that Assumptions~\hyperref[ass:K]{$\cK$} and~\eqref{ass:varphi2} are satisfied. Then, conditionally on $\cS$, there exists $\eps_0$ such that for any $\eps\in(0,\eps_0)$ there exists $p>0$ such that the following event holds with high probability as $t$ tends to infinity: For all $s\in F_{p,t}$, 
	\begin{align}
		\bigg|\frac{1}{\cK(\frac{\lambda^*}{\lambda^*+d^*}t)}\Big(I_s^\cont-\frac{d^*}{\lambda^*+d^*}t-\frac{1}{\lambda^*+d^*}\cK_\alpha (r(t))\Big)- \frac{\lambda^*+d^*}{2}\bigg|&<\eps,
		\intertext{and} 
		\bigg|\frac{1}{\cK(\frac{\lambda^*}{\lambda^*+d^*}t)}\Big(\varphi_1(\max_{v\in \cA_s^\cont}\deg^{(v)}(s))-\frac{\lambda^*}{\lambda^*+d^*}t+\frac{1}{\lambda^*+d^*}\cK_\alpha(r(t))\Big)- \frac{\lambda^*-d^*}{2}\bigg|&<\eps.
	\end{align} 
\end{proposition} 

\begin{remark}
	When setting $s=t$ in the above result, we obtain convergence to zero of the quantities in the absolute values. Similar to Remark~\ref{rem:improvetoas}, it may be possible to strengthen this to almost sure convergence if $\cK$ grows sufficiently fast. However, this again requires stronger results compared to those in Corollary~\ref{cor:growth}. \ensymboldremark
\end{remark} 

\begin{remark}
	Proposition~\ref{prop:It} also provides a precise asymptotic expansion for the birth-time of the individual with the largest offspring when $d^*\geq R$, i.e.\ in the \emph{rich die young regime}. However, we prove a (weaker) lower bound for $I^\cont_t$ in this regime that holds in greater generality in Section~\ref{sec:maxrdy}. \ensymboldremark
\end{remark}

\begin{remark}
	In Case~\ref{item:1} Assumption~\eqref{ass:A2} is not satisfied, i.e.\ when $\rho_1$ converges. By definition $\alpha\equiv \rho_1$ in this case. As a result, $\cK_\alpha$ converges in this case. Since $\cK$ diverges by Assumption~\eqref{ass:varphi2}, the results in Proposition~\ref{prop:It} clearly hold when one omits the term $\cK_\alpha(r(t))$. Similarly and related to Remark~\ref{rem:d0ass}, when Assumption~\eqref{ass:A2} is not satisfied but $\lim_{i\to\infty}d(i)=0$ such that $\alpha=\rho_1=o(\varphi_2)$, the condition that Assumption~\hyperref[ass:Kalpha]{$\cK_\alpha$} is satisfied can be omitted.\ensymboldremark
\end{remark}

We split the proof of Proposition~\ref{prop:It} into several parts. Let us first outline the proof strategy. 

We define the functions $H,\wt H\colon [0, \infty)\to\R$ as
\be \label{eq:H}
H(x):=-(x+d^*)+\sqrt{2(\lambda^*+d^*)x}\qquad\text{and}\qquad \wt H(x):=\begin{cases} 
	H(x) &\mbox{if } x\geq \frac{(d^*)^2}{2(\lambda^*+d^*)},\\ 
	-\frac{(d^*)^2}{2(\lambda^*+d^*)}&\mbox{if } x<\frac{(d^*)^2}{2(\lambda^*+d^*)}.
\end{cases}
\ee 
It is readily checked that $H$ and $\wt H$ have a unique maximum at $u^*>0$, with 
\be 
u^*:=\frac{\lambda^*+d^*}{2}, \quad\text{and}\quad \ H(u^*)=\wt H(u^*)=\frac{\lambda^*-d^*}{2}. 
\ee 
See also Figure~\ref{fig:Hplot}.
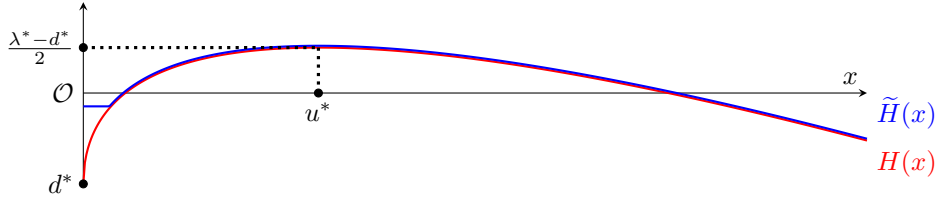
\begin{figure}[h]
	\centering 
	\begin{tikzpicture}
		\begin{axis}[
			axis lines=middle,
			width=12cm,
			height=4cm,
			xmin=0, xmax=5,
			ymin=-1, ymax=1,
			samples=500,
			domain=0:5,
			clip=false,
			legend pos=north east,
			xlabel={$x$},
			xtick=\empty,
			ytick=\empty
			]
			\addplot[red, thick] {-(x+1) + sqrt(6*x)} node[below right] {$H(x)$};
			\addplot[blue, thick] {(x<=3/2)*max(-(x+1) + sqrt(6*x), -1/6)+(x>3/2)*(-(x+1)+sqrt(6*x)) + 0.02} node[above right] {$\widetilde{H}(x)$};

			\draw[black, fill=black] (axis cs:0,0.5) circle[radius=1.5pt] node[left] {$\frac{\lambda^*-d^*}{2}$};
			\draw[black, fill=black] (axis cs:1.5,0) circle[radius=1.5pt] node[below] {$u^*$};
			\draw[black, fill=black] (axis cs:0,0) node[left] {$\mathcal{O}$}; 
			\draw[black, fill=black] (axis cs:0,-1) circle[radius=1.5pt] node[left] {$d^*$}; 
			
			\draw[black, dotted,very thick] (axis cs:1.5,0) -- (axis cs:1.5,0.5);
			\draw[black, dotted,very thick] (axis cs:0,0.5) -- (axis cs:1.5,0.5);
		\end{axis}
		
	\end{tikzpicture}
	\caption{The functions $H$ and $\wt H$, with their unique maximum $u^*$.}\label{fig:Hplot}
\end{figure}

We then partition the time interval $[0,t]$ into a variety of intervals. Recall the functions $\cK_\alpha$ and $\cK$, defined in~\eqref{eq:Ks}. First, for $\eps>0$ small, we omit the interval
\be 
\Big[0,\frac{d^*}{\lambda^*+d^*}t+\frac{1}{\lambda^*+d^*}\cK_\alpha(r(t))-\eps\cK\Big(\frac{\lambda^*}{\lambda^*+d^*}t\Big)\Big], 
\ee 
since we know by Proposition~\ref{prop:oldage} that no individual born in this interval is alive at time $s$, for any $s\in F_{p,t}$ with high probability. Then, we consider
\begin{align} 
	U^{\e}_t&:=\Big[\frac{d^*}{\lambda^*+d^*}t+\frac{1}{\lambda^*+d^*}\cK_\alpha(r(t))-\eps\cK\Big(\frac{\lambda^*}{\lambda^*+d^*}t\Big),\frac{d^*}{\lambda^*+d^*}t+\frac{1}{\lambda^*+d^*}\cK_\alpha(r(t))\Big], 
	\intertext{and, for some large integer $A\in\N$ to be determined,}
	U^{\ell}_t(A)&:=\Big [\frac{d^*}{\lambda^*+d^*}t+\frac{1}{\lambda^*+d^*}\cK_\alpha(r(t))+A\cK\Big(\frac{\lambda^*}{\lambda^*+d^*}t\Big),t\Big]
\end{align}
as two intervals where individuals with `too few'  offspring are born. We further partition these intervals into smaller parts later. What remains is the interval 
\be \label{eq:mainint}
\Big[\frac{d^*}{\lambda^*+d^*}t+\frac{1}{\lambda^*+d^*}\cK_\alpha(r(t)),\frac{d^*}{\lambda^*+d^*}t+\frac{1}{\lambda^*+d^*}\cK_\alpha(r(t))+A\cK\Big(\frac{\lambda^*}{\lambda^*+d^*}t\Big)\Big]. 
\ee 
We  partition $[0,u^*-\eps)\cup[u^*+\eps,A)$ into intervals $[x_i,x_{i+1})$ of length $\zeta<u^*-\eps$ for $i\in \mathbb I$ (with $\mathbb I$ an index set depending on $A$ and $\eps$), and let $t^*:=\frac{\lambda^*}{\lambda^*+d^*}t$ for ease of writing. Then, we partition the interval in~\eqref{eq:mainint} into 
\be
\wt U_{t,i}:=\Big[\frac{d^*}{\lambda^*+d^*}t+\frac{1}{\lambda^*+d^*}\cK_\alpha(r(t))+x_i\cK(t^*),\frac{d^*}{\lambda^*+d^*}t+\frac{1}{\lambda^*+d^*}\cK_\alpha(r(t))+x_{i+1}\cK(t^*)\Big)\quad \text{for }i\in \I,
\ee 
and 
\be   U_t:=\Big[\frac{d^*}{\lambda^*+d^*}t+\frac{1}{\lambda^*+d^*}\cK_\alpha(r(t))+(u^*-\eps)\cK(t^*),\frac{d^*}{\lambda^*+d^*}t+\frac{1}{\lambda^*+d^*}\cK_\alpha(r(t))+(u^*+\eps)\cK(t^*)\Big),
\ee
see Figure~\ref{fig:part}.

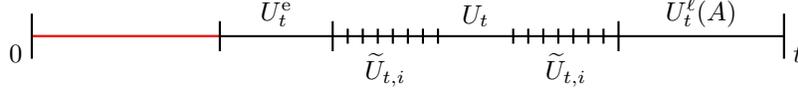
\begin{figure}[h]
	\centering
	\begin{tikzpicture}
		\draw[thick] (0,0) -- (10,0);
		\draw[thick,red] (0,0) -- (2.5,0);
		
		\draw[thick] (0,-0.3) -- (0,0.3);
		\draw[thick] (10,-0.3) -- (10,0.3);
		
		\foreach \x in {2.5,4,7.8} {
			\draw[thick] (\x,-0.2) -- (\x,0.2);
		}
		
		\foreach \x in {4.2,4.4,4.6,4.8,5,5.2,5.4,6.4,6.6,6.8,7,7.2,7.4,7.6} {
			\draw[thick] (\x,-0.1) -- (\x,0.1);
		}
		
		\node[below left] at (0,0) {$0$};
		\node[below right] at (10,0) {$t$};
		\node[above] at (3.25,0) {$U^{\mathrm{e}}_t$};
		\node[above] at (8.9,0) {$U^{\mathrm{\ell}}_t(A)$};
		\node[above] at (5.9,0) {$U_t$};
		\node[below] at (4.7,-0.1) {$\wt U_{t,i}$};
		\node[below] at (7.1,-0.1) {$\wt U_{t,i}$};			
	\end{tikzpicture}
	\caption{Partition of the interval $[0,t]$. The red interval is the part we omit. The smallest intervals are the $\wt U_{t,i}$ with $i\in\I$.}\label{fig:part}
\end{figure} 

The aim of the proof is to show that the probability that there exists an individual who, at time $s$, is alive, has the largest offspring, and has a birth-time in $U_t$, for all $s\in F_{p,t}$, converges to one, for some sufficiently small $p>0$ (with $F_{p,t}$ as in~\eqref{eq:Ft}). To make this precise, we let 
\begin{itemize}
	\item $D^{\max}_{u^*-\eps,u^*+\eps}(s)$ denote the largest offspring of individuals born in $U_t$ and alive at time $s$,
	\item $D^{\max}_{x_i,x_{i+1}}(s)$ denote the largest offspring among individuals born in $\wt U_{t,i}$ and alive at time $s$,
	\item $D^{\max}_{\eps,t}(s)$ denote the largest offspring among individuals born in $U^\e_t$ and alive at time $s$,
	\item $D^{\max}_{A,t}(s)$ denote the largest offspring among individuals born in $U^\ell_t(A)$ and alive at time $s$.
\end{itemize} 
Then, we take 
\be \label{eq:xi}
\xi=\xi(\eps):=\frac12\min\{H(u^*)-H(u^*-\eps),H(u^*)-H(u^*+\eps)\}=\frac14 \frac{\eps^2}{\lambda^*+d^*}-\cO(\eps^3).
\ee 
As a result, we have $H(u^*)-\xi>\wt H(x_{i+1})+\xi$ for all $i\in \mathbb I$ when $\eps$ is sufficiently small, since $-\frac{(d^*)^2}{2(\lambda^*+d^*)}<\frac{\lambda^*-d^*}{2}=H(u^*)$. As a result,
\be \ba 
\mathbb P_\cS(I_s^\cont{}&\in U_t\text{ for all }s\in  F_{p,t})\\ 
\geq \mathbb P_\cS\big({}& D^{\max}_{u^*-\eps,u^*+\eps}(s)\geq \lceil \varphi_1^{-1}\big(t^*-\tfrac{1}{\lambda^*+d^*}\cK_\alpha(r(t))+(H(u^*)-\xi )\cK(t^*)\big)\rceil, \\
{}& D^{\max}_{x_i,x_{i+1}}(s)\leq \lceil\varphi_1^{-1}\big(t^*-\tfrac{1}{\lambda^*+d^*}\cK_\alpha(r(t))+(\wt H(x_{i+1})+\xi)\cK(t^*)\big)\rceil \text{ for all }i\in\mathbb I,\\ 
{}& D^{\max}_{\eps,t}(s)\leq \lceil\varphi_1^{-1}\big( t^*-\tfrac{1}{\lambda^*+d^*}\cK_\alpha(r(t))-(d^*-\xi)\cK(t^*)\big)\rceil,\\
{}& D^{\max}_{A,t}(s)\leq\lceil\varphi_1^{-1}\big(t^*-\tfrac{1}{\lambda^*+d^*}\cK_\alpha(r(t))  \big)\rceil,\text{ for all }s\in F_{p,t}\big).
\ea\ee 
Indeed, the event in the lower bound guarantees that $I_s^\cont\in U_t$ for all $s\in F_{p,t}$, since $H$ has a unique maximum at $u^*$ and $\wt H(x_{i+1})+\xi<H(u^*)-\xi$. A union bound thus yields that 
\be \ba \label{eq:probsplit4}
\mathbb P_\cS{}&(\{I_s^\cont\in U_t\text{ for all }s\in  F_{p,t}\}^c)\\ 
\leq{}& \mathbb P_\cS\big(\{D^{\max}_{u^*-\eps,u^*+\eps}(s)\geq \lceil\varphi_1^{-1}\big(t^*-\tfrac{1}{\lambda^*+d^*}\cK_\alpha(r(t))+(H(u^*)-\xi)\cK(t^*)\big)\rceil \text{ for all }s\in F_{p,t}\}^c) \\
&+\sum_{i\in \mathbb I}\mathbb P_\cS\big(\{ D^{\max}_{x_i,x_{i+1}}(s)\leq\lceil \varphi_1^{-1}\big(t^*-\tfrac{1}{\lambda^*+d^*}\cK_\alpha(r(t))+(\wt H(x_{i+1})+\xi)\cK(t^*)\big)\rceil\text{ for all }s\in F_{p,t}\}^c\big)\\ 
&+ \mathbb P_\cS\big(\{D^{\max}_{\eps,t}(s)\leq \lceil\varphi_1^{-1}\big( t^*-\tfrac{1}{\lambda^*+d^*}\cK_\alpha(r(t))-(d^*-\xi)\cK(t^*)\big)\rceil \text{ for all }s\in F_{p,t}\}^c\big)\\
&+\mathbb P_\cS\big(\{ D^{\max}_{A,t}(s)\leq \lceil\varphi_1^{-1}\big(t^*-\tfrac{1}{\lambda^*+d^*}\cK_\alpha(r(t))\big)\rceil \text{ for all }s\in F_{p,t}\}^c\big).
\ea \ee 
We show in the next three lemmas that each term on the right-hand side  converges to zero as $t$ tends to infinity. As $\mathbb I$ is a finite set, this proves the first result in Proposition~\ref{prop:It}.

\begin{lemma}[Lower bound on offspring of individuals in the the optimal window]\label{lemma:optwin}
	Suppose that $b$ and $d$ satisfy the conditions in Proposition~\ref{prop:It} and recall $\xi$ from~\eqref{eq:xi}. There exists $\eps_0>0$ such that for all $\eps\in(0,\eps_0)$ there exist  $p=p(\eps)>0$ such that
	\be 
	\lim_{t\to\infty}\mathbb P_\cS \big(D^{\max}_{u^*-\eps,u^*+\eps}(s)\geq \lceil \varphi_1^{-1}\big(t^*-\tfrac{1}{\lambda^*+d^*}\cK_\alpha(r(t))+(H(u^*)-\xi)\cK(t^*)\big)\rceil \text{ for all }s\in F_{p,t})=1.
	\ee
\end{lemma}

The lemma shows that there exist individuals born in the `optimal window' $U_t$ that are alive and have a large offspring around time $t$. The next lemma provides upper bound for the size of the offspring of individuals born close to the optimal window (i.e.\ in a sub-optimal window), as well as for individuals in the optimal window. 

\begin{lemma}[Upper bound on offspring of individuals in the (sub-)optimal window]\label{lemma:suboptwin}
	Suppose $b$ $\,$ and $d$ satisfy the conditions in Proposition~\ref{prop:It} and recall $\xi$ from~\eqref{eq:xi}. There exist $\zeta=\zeta(\eps),p=p(\eps)>0$ such that for all $i\in\I$,
	\be 
	\lim_{t\to\infty}\mathbb P_\cS\big( D^{\max}_{x_i,x_{i+1}}(s)\leq \lceil\varphi_1^{-1}\big(t^*-\tfrac{1}{\lambda^*+d^*}\cK_\alpha(r(t))+(\wt H(x_{i+1})+\xi)\cK(t^*)\big)\rceil\text{ for all }s\in F_{p,t}\big)=1.
	\ee 
	Similarly, there exists $\eps_0>0$ such that for all $\eps\in(0,\eps_0)$ there exists $p=p(\eps)>0$ such that 
	\be 
	\lim_{t\to\infty}\mathbb P_\cS \big(D^{\max}_{u^*-\eps,u^*+\eps}(s)\leq\lceil \varphi_1^{-1}\big(t^*-\tfrac{1}{\lambda^*+d^*}\cK_\alpha(r(t))+(H(u^*)+\xi)\cK(t^*)\big)\rceil \text{ for all }s\in F_{p,t})=1,
	\ee 
	Finally, exists $p(\eps)>0$ such that 
	\be 
	\lim_{t\to\infty}\mathbb P_\cS\big(D^{\max}_{\eps,t}(s)\leq \lceil\varphi_1^{-1}\big( t^*-\tfrac{1}{\lambda^*+d^*}\cK_\alpha(r(t))-(d^*-\xi)\cK(t^*)\big)\rceil \text{ for all }s\in F_{p,t}\big)=1.
	\ee 
\end{lemma}

\begin{remark} 
	The first result in Lemma~\ref{lemma:suboptwin} combined with~\eqref{eq:probsplit4} implies the second result of Proposition~\ref{prop:It}. \ensymboldremark
\end{remark}

Finally, the third lemma shows that individuals that are born far away from the optimal window have small offspring compared to those in the optimal window. 

\begin{lemma}[Individuals far outside the optimal window have few offspring]\label{lemma:outoptwin}
	Suppose that $b$ and $d$ satisfy the conditions in Proposition~\ref{prop:It}.  There exist $A>0$ large and $p>0$ small such that
	\be 
	\lim_{t\to\infty}\mathbb P_\cS\big( D^{\max}_{A,t}(s)\leq\lceil \varphi_1^{-1}\big(t^*-\tfrac{1}{\lambda^*+d^*}\cK_\alpha(r(t))\big)\rceil \text{ for all }s\in F_{p,t}\big)=1.
	\ee 
\end{lemma}

It is clear that Lemmas~\ref{lemma:optwin}, \ref{lemma:suboptwin}, and~\ref{lemma:outoptwin}, combined with Proposition~\ref{prop:oldage}, imply Proposition~\ref{prop:It}. Furthermore, Proposition~\ref{prop:It} is the main part for the proof of~\eqref{eq:Inplusmax} in Theorem~\ref{thrm:asPA} and~\eqref{eq:Inconv} and~\eqref{eq:maxdegconv} in Theorem~\ref{thrm:conv}.

We proceed by proving the three lemmas.

\begin{proof}[Proof of Lemma~\ref{lemma:optwin}]		
	For ease of writing, we set
	\be 
	k:=\ceil{\varphi_1^{-1}\big(t^*-\tfrac{1}{\lambda^*+d^*}\cK_\alpha(r(t))+(H(u^*)-\xi)\cK(t^*)\big)},
	\ee 
	and for some small constant $c\in(0,1)$ to be determined, we let $\cB(u^*-c\eps^2,u^*)$ denote  the set of individuals born in the interval 
	\be 
	\Big[\frac{d^*}{\lambda^*+d^*}t+\frac{1}{\lambda^*+d^*}\cK_\alpha(r(t))+(u^*-c\eps^2)\cK(t^*),\frac{d^*}{\lambda^*+d^*}t+\frac{1}{\lambda^*+d^*}\cK_\alpha(r(t))+u^*\cK(t^*) \Big],
	\ee 
	which we note is a subset of the individuals whose offspring we take the maximum of in the random variable $D^{\max}_{u^*-\eps,u^*+\eps}(s)$. We can thus bound
	\be\ba 
	\mathbb P_\cS{}&\big(D^{\max}_{u^*-\eps,u^*+\eps}(s)\geq k, \text{ for all }s\in F_{p,t})\\ 
	&\geq \Ps{\exists v\in \cB(u^*-c\eps^2,u^*): v\text{ alive at time }t+p\cK(t^*), \deg^{(v)}(t-p\cK(t^*))\geq k}.
	\ea\ee 
	Indeed, restring ourselves to individuals born in $\cB(u^*-c\eps^2,u^*)$ and requiring that individuals live until the end of the interval $F_{p,t}$ and produce at least $k$ offspring by the start of the interval yields a lower bound. We first write $\{v\text{ is alive at time }t+p\cK(t^*)\text{ and }\deg^{(v)}(t-p\cK(t^*))\geq k\}$ as
	\be 
	\{D^{(v)}\geq k, S^{(v)}_k\leq t-p\cK(t^*)-\sigma_v, S^{(v)}_{D^{(v)}+1}>t+p\cK(t^*)-\sigma_v\}.
	\ee 
	By using that $v\in \cB(u^*-c\eps^2, u^*)$, we can bound its birth-time $\sigma_v$ from below and above in the second and third part of the event, respectively. Hence, the desired event is a subset of
	\be \ba 
	\{{}&D^{(v)}\geq k, S^{(v)}_k\leq  t^*-\tfrac{1}{\lambda^*+d^*}\cK_\alpha(r(t))-(u^*+p)\cK(t^*),\\ 
	&S^{(v)}_{D^{(v)}+1}> t^*-\tfrac{1}{\lambda^*+d^*}\cK_\alpha(r(t))-(u^*-c\eps^2-p)\cK(t^*)\}.
	\ea\ee  
	With a similar approach as in the steps between~\eqref{eq:utbound} and~\eqref{eq:utbound2} and by leveraging the independence of the reproduction process of distinct individuals (when not conditioning on survival), we obtain for any $m\in \N$ the lower bound
	\be\ba \label{eq:mbound}
	1-\mathbb P_{\cS}(|\cB(u^*-c\eps^2, u^*{}&)|\leq m)\\
	-\P{\cS}^{-1}\Big[1-\mathbb P\Big({}&D\geq k, S_k\leq  t^*-\tfrac{1}{\lambda^*+d^*}\cK_\alpha(r(t))-(u^*+p)\cK(t^*),\\ 
	&S_{D+1}> t^*-\tfrac{1}{\lambda^*+d^*}\cK_\alpha(r(t))-(u^*-c\eps^2-p)\cK(t^*)\Big)\Big]^m.
	\ea\ee 
	We then set 
	\be 
	m:=\exp\Big(\frac{\lambda^*d^*}{\lambda^*+d^*}t+\frac{\lambda^*}{\lambda^*+d^*}\cK_\alpha(r(t))+(\lambda^*-\delta)u^*\cK(t^*)\Big),
	\ee
	and conclude that the probability on the first line tends to zero with $t$ for any $\delta>0$ by Corollary~\ref{cor:growth} (with $r=\frac{d^*}{\lambda^*+d^*}t+\frac{1}{\lambda^*+d^*}\cK_\alpha(r(t))+(u^*-\eps)\cK(t^*)$, $s=\frac{d^*}{\lambda^*+d^*}t+\frac{1}{\lambda^*+d^*}\cK_\alpha(r(t))+u^*\cK(t^*)$, and $u=\delta u^*\cK(t^*)$) and since $\cK$ diverges. What remains, is to bound the term in the straight brackets in~\eqref{eq:mbound} from above. By using that $1-x\leq \e^{-x}$ for any $x\in\R$, we obtain the upper bound
	\be\ba\label{eq:expBlb}  
	\exp\bigg({}&-\mathbb P\Big(D\geq k, S_k\leq  t^*-\frac{1}{\lambda^*+d^*}\cK_\alpha(r(t))-(u^*+p)\cK(t^*),\\
	&\hphantom{-\mathbb P\Big(}\ S_{D+1}> t^*-\frac{1}{\lambda^*+d^*}\cK_\alpha(r(t))-(u^*-c\eps^2-p)\cK(t^*)\Big)\\ 
	&\times\exp\Big(\frac{\lambda^*d^*}{\lambda^*+d^*}t+\frac{\lambda^*}{\lambda^*+d^*}\cK_\alpha(r(t))+(\lambda^*-\delta)u^*\cK(t^*)\Big)\bigg).
	\ea \ee 
	We bound the probability from below to obtain an upper bound. From this point onwards, we consider Case~\ref{item:2} in the statement of Proposition~\ref{prop:It}. We comment on the required adaptations in Case~\ref{item:1} at the end of the proof. 
	
	In Case~\ref{item:2} and for any $\eta>0$, there exists $t'>0$ such that, for all $t\geq t'$, we have $d(i)\leq d^*+\eta $ for all $i \geq k$ (as $k$ grows with $t$). We can thus use $(ii)$ in Lemma~\ref{lemma:survdeg} to obtain the lower bound
	\be \ba 
	\P{D\geq k}\e^{-(d^*+\eta)(2p+c\eps^2)\cK(t^*)} \mathbb E\bigg[{}&\ind\{S_k\leq t^*-\tfrac{1}{\lambda^*+d^*}\cK_\alpha(r(t))-(u^*+p)\cK(t^*) \}\\ 
	&\exp\Big((d^*+\eta)\big(S_k-\big(t^*-\tfrac{1}{\lambda^*+d^*}\cK_\alpha(r(t))-(u^*+p)\cK(t^*)\big)\big)\Big)\bigg].
	\ea \ee 
	By the choice of $k$, it follows that 
	\be 
	\varphi_1(k-1)\leq t^*-\frac{1}{\lambda^*+d^*}\cK_\alpha(r(t)) +(H(u^*)-\xi)\cK(t^*) \leq \varphi_1(k).
	\ee 
	Hence, as $\cK(t^*)=(1+o(1))\varphi_2(k)$ by Assumption~\hyperref[ass:K]{$\cK$} and $(a)$ and $(d)$ in Lemma~\ref{lemma:func}, and $\varphi(k-1)=\varphi_1(k)-o(1)$ as $b$ diverges,
	\be \ba 
	\varphi_1(k)-(H(u^*)+u^* +\tfrac32p-\xi)\varphi_2(k)&\leq t^*-\frac{1}{\lambda^*+d^*}\cK_\alpha(r(t))-(u^*+p)\cK(t^*) \\
	&\leq \varphi_1(k)-(H(u^*)+u^*+\tfrac12p- \xi)\varphi_2(k),
	\ea \ee 
	for all $t$ large. We thus obtain the lower bound
	\be \ba 
	\P{D\geq k}\e^{-(d^*+\eta)(2p+c\eps^2)\cK(t^*)} \mathbb E\bigg[{}&\ind\{S_k\leq \varphi_1(k)-(H(u^*)+u^* +\tfrac32p-\xi)\varphi_2(k)  \}\\ 
	&\exp\Big((d^*+\eta)\big(S_k-\big(\varphi_1(k)-(H(u^*)+u^*+\tfrac12p- \xi)\varphi_2(k)\big)\big)\Big)\bigg].
	\ea \ee 
	Note that $H(u^*)+u^*+\tfrac32p- \xi=\lambda^*+\tfrac32p- \xi>0$ for $\xi$ sufficiently small. As a result, we can use Proposition~\ref{prop:mdptilt}, with $z=H(u^*)+u^*+\tfrac32p-\xi>0$, $y=H(u^*)+u^*+\tfrac12p- \xi$, and $\theta=d^*+\eta$, to bound the expected value from below by 
	\be \label{eq:sklb}
	\exp\big(-\big[(1+(d^*+\eta))p +\tfrac12 (H(u^*)+u^*+\tfrac32p- \xi)^2\big]\varphi_2(k)\big).
	\ee 
	We combine this with Lemma~\ref{lemma:Dtail}. By Assumption~\eqref{ass:A2}, $\rho_1$ diverges. If, furthermore, $\rho_2$ also diverges, then 
	\be 
	\P{D\geq k}=\e^{-\rho_1(k)-(\frac12+o(1))\rho_2(k)}=\e^{-\rho_1(k)-(\frac12(d^*)^2+o(1))\cK(t^*)}, 
	\ee 
	where we combine that $d$ converges to $d^*\in[0,\infty)$ and thus $\rho_2(k)=((d^*)^2+o(1))\varphi_2(k)$ by $(c)$ in Lemma~\ref{lemma:func} together with Assumption~\hyperref[ass:K]{$\cK$} in the last step. If, on the other hand, $\rho_2$ converges (which can only occur for $d^*=0$), then again by Lemma~\ref{lemma:Dtail},
	\be 
	\P{D\geq k}=\e^{-\rho_1(k)-\cO(1)}=\e^{-\rho_1(k)-o(\cK(t^*))},
	\ee 
	where the final step follows since $\varphi_2$ diverges by Assumption~\eqref{ass:varphi2}. We can unify both cases by writing $o(\cK(t^*))=(\frac12 (d^*)^2+o(1))\cK(t^*)$ in the second case, for which $d^*=0$. Together with~\eqref{eq:sklb}, this yields the lower bound
	\be \ba 
	\mathbb P\Big({}&D\geq k,S_k\leq  t^*-\tfrac{1}{\lambda^*+d^*}\cK_\alpha(r(t))-(u^*+p)\cK(t^*),\\ 
	&S_{D+1}> t^*-\tfrac{1}{\lambda^*+d^*}\cK_\alpha(r(t))-(u^*-c\eps^2-p)\cK(t^*)\Big)\\
	\geq{}& \exp\Big(-\rho_1(k)-(\tfrac12(d^*)^2+o(1))\cK(t^*)-(d^*+\eta)(2p+c\eps^2)\cK(t^*)\\ 
	&\hphantom{\exp\Big(\ }-\big[(1+(d^*+\eta))p +\tfrac12 (H(u^*)+u^*+\tfrac32p-\xi)^2\big]\varphi_2(k)\Big).
	\ea\ee 
	Writing $\rho_1(k)=d^*\varphi_1(k)+\alpha(k)$ and applying Assumption~\hyperref[ass:K]{$\cK$} to the term $\varphi_2(k)$ together with $(a)$ and $(d)$ in Lemma~\ref{lemma:func}, we arrive at 
	\be 
	\exp\Big(-d^*\varphi_1(k)-\alpha(k)-\big[\tfrac12 (d^*)^2+\tfrac12 (H(u^*)+u^*+\tfrac32p-\xi)^2+p+(3p+c\eps^2)(d^*+\eta) +o(1)\big]\cK(t^*)\Big).
	\ee 
	Finally, by the choice of $k$ and by using Assumptions~\hyperref[ass:K]{$\cK$} and~\hyperref[ass:Kalpha]{$\cK_\alpha$}, we obtain 
	\be \ba 
	\exp\Big(-\frac{\lambda^*d^*}{\lambda^*+d^*}t-\cK_\alpha(r(t))-\big[{}&d^*(H(u^*)-\xi)+\tfrac12 (d^*)^2+\tfrac12 (H(u^*)+u^*+\tfrac32p-\xi)^2+p \\ 
	&+(3p+c\eps^2)(d^*+\eta) +o(1)\big]\cK(t^*)\Big).
	\ea \ee 
	Multiplying this with the exponential term on the second line of~\eqref{eq:expBlb} and using that $H(u^*)+u^*=\lambda^*$, we bound the inner exponential in~\eqref{eq:expBlb} from below
	\be \ba\label{eq:boundchange}
	\exp\Big({}&-\big[-(\lambda^*+d^*)u^*+\tfrac12 (H(u^*)+u^*+d^*)^2\big]\cK(t^*)+\xi\big[H(u^*)+u^*+d^*\big]\cK(t^*)\\ 
	&-\big[\tfrac32 p\lambda^*+\tfrac12(\tfrac32p-\xi)^2+p+(3p+c\eps^2)(d^*+\eta)  +\delta u^*+o(1)\big]\cK(t^*)\Big).
	\ea \ee
	The terms in the first brackets equal zero by the choice of $H$ and $u^*$, and the terms in the second brackets equal $\lambda^*+d^*>0$. Then, we recall that $\xi$ is of the order $\eps^2$, as follows from~\eqref{eq:xi}. Hence, by choosing $c$ sufficiently small, that is, $c<\frac14(d^*+\eta)^{-1}$, it follows that $(\lambda^*+d^*)\xi-c(d^*+\eta)\eps^2$ is strictly positive when $\eps$ and $\eta$ are sufficiently small.  The remaining terms on the second line can be made arbitrarily small with respect to $\eps^2$ by choosing $p,\eta$, and $\delta$ sufficiently small. As a result, the entire term diverges with $t$, so that we arrive at the desired result. 
	
	To deal with Case~\ref{item:1}, we apply $(iii)$ in Lemma~\ref{lemma:survdeg} to the probability in~\eqref{eq:expBlb} to obtain the upper bound
	\be \ba
	\exp\bigg({}&-\mathbb P\Big(S_k\leq  t^*-\frac{1}{\lambda^*+d^*}\cK_\alpha(r(t))-(u^*+p)\cK(t^*)\Big)\\ 
	&\times\exp\Big(\frac{\lambda^*d^*}{\lambda^*+d^*}t+\frac{\lambda^*}{\lambda^*+d^*}\cK_\alpha(r(t))+(\lambda^*-\delta)u^*\cK(t^*)\Big)\bigg).
	\ea\ee 
	The remaining steps in the proof above can then be followed, using Lemma~\ref{lemma:mdp} rather than Proposition~\ref{prop:mdptilt}, to arrive at~\eqref{eq:boundchange} with $d^*$ set to equal $0$, which  concludes the proof.
\end{proof} 

\begin{proof}[Proof of Lemma~\ref{lemma:suboptwin}]
	We set
	\be\label{eq:k2}
	k:=\ceil{\varphi_1^{-1}\big(t^*-\tfrac{1}{\lambda^*+d^*}\cK_\alpha(r(t))+(\wt H(x_{i+1})+\xi)\cK(t^*)\big)},
	\ee 
	and 
	\be 
	B(x_i,x_{i+1}):=\cB\Big(\frac{d^*}{\lambda^*+d^*}t+\frac{1}{\lambda^*+d^*}\cK_\alpha(r(t))+x_i\cK(t^*),\frac{d^*}{\lambda^*+d^*}t+\frac{1}{\lambda^*+d^*}\cK_\alpha(r(t))+x_{i+1}\cK(t^*)\Big).
	\ee 
	Then, we bound
	\be \ba \label{eq:suboptinwindow}
	\mathbb P_\cS{}&\big( D^{\max}_{x_i,x_{i+1}}(s)\leq k,\text{ for all }s\in F_{p,t}\big)\\ 
	&\geq 1-\Ps{\exists v\in B(x_i,x_{i+1})\ \exists s\in F_{p,t}: v\text{ alive at time }s\text{ and }\deg^{(v)}(s)\geq k\big)}\\ 
	&= 1-\Ps{\exists v\in B(x_i,x_{i+1})\ \exists s\in F_{p,t}: D\geq k, S_k\leq s-\sigma_v, S_{D+1}>s-\sigma_v}.
	\ea\ee 
	We bound the probability on the last line from above to show that it tends to zero with $t$. First, as $s\in F_{p,t}$ and omitting the requirement that $S_{D+1}>s-\sigma_v$, we have the inclusion 
	\be 
	\{D\geq k, S_k\leq s-\sigma_v, S_{D+1}>s-\sigma_v\}\subseteq \{D\geq k, S_k\leq t+p\cK(t^*)-\sigma_v\}.
	\ee 
	We then use that $v\in B(x_i,x_{i+1})$ to bound $\sigma_v$ from  below. We also introduce the event 
	\be 
	\cE_t:=\Big\{|B(x_i,x_{i+1})|\leq \exp\Big(\frac{\lambda^*d^*}{\lambda^*+d^*}t+\frac{\lambda^*}{\lambda^*+d^*}\cK_\alpha(r(t))+(\lambda^*+\delta)x_{i+1} \cK(t^*)\Big)\Big)\Big\}. 
	\ee 
	By using a union bound, we then arrive at the upper bound
	\be\ba \label{eq:unionupper}
	\mathbb P_{\cS}(\exists v\in B(x_i,x_{i+1})\ {}&\exists s\in F_{p,t}: D\geq k, S_k\leq s-\sigma_v, S_{D+1}>s-\sigma_v)\\ 
	\leq \P{\cS}^{-1}\exp{}&\Big(\frac{\lambda^*d^*}{\lambda^*+d^*}t+\frac{\lambda^*}{\lambda^*+d^*}\cK_\alpha(r(t))+(\lambda^*+\delta)x_{i+1} \cK(t^*)\Big)\\ 
	\times \mathbb P\Big({}&D\geq k,S_k\leq t^*-\tfrac{1}{\lambda^*+d^*}\cK_\alpha(r(t))-(x_i-p)\cK(t^*)\Big)+\Ps{\cE_t^c}.
	\ea \ee 
	The last term tends to zero by Corollary~\ref{cor:growth} (with $r=\frac{d^*}{\lambda^*+d^*}t+\frac{1}{\lambda^*+d^*}\cK_\alpha(r(t))+x_i\cK(t^*)$, $s=\frac{d^*}{\lambda^*+d^*}t+\frac{1}{\lambda^*+d^*}\cK_\alpha(r(t))+x_{i+1}\cK(t^*)$, and $u=\delta x_{i+1}\cK(t^*)$), since the event $\cS$ has positive probability and because $\cK$ diverges by Assumption~\eqref{ass:varphi2}. By the choice of $k$, it follows that 
	\be 
	\varphi_1(k)\geq t^*-\tfrac{1}{\lambda^*+d^*}\cK_\alpha(r(t))+(\wt H(x_{i+1})+\xi)\cK(t^*).
	\ee 
	Hence, as $\cK(t^*)=(1+o(1))\varphi_2(k)$ by Assumption~\hyperref[ass:K]{$\cK$} and $(a)$ in Lemma~\ref{lemma:func},
	\be
	\varphi_1(k)-(\wt H(x_{i+1})+x_i-\tfrac32p+\xi)\varphi_2(k)\geq t^*-\tfrac{1}{\lambda^*+d^*}\cK_\alpha(r(t))-(x_i-p)\cK(t^*),
	\ee 
	for all $t$ large. Using this in the first probability on the third line of~\eqref{eq:unionupper} combined with the independence of $D$ and $S_k$, we thus obtain the upper bound
	\be \ba \label{eq:largedegbound}
	\mathbb P\Big({}&D\geq k,S_k\leq t^*-\tfrac{1}{\lambda^*+d^*}\cK_\alpha(r(t))-(x_i-p)\cK(t^*)\Big)\\ 
	&\leq \P{D\geq k}\P{S_k-\varphi_1(k)\leq -(\wt H(x_{i+1})+x_i-\tfrac32p+\xi)\varphi_2(k) }.
	\ea\ee 
	We now distinguish between two cases: $(a)$ $x_{i+1}\geq \frac{(d^*)^2}{2(\lambda^*+d^*)}$ and $(b)$ $x_{i+1}<\frac{(d^*)^2}{2(\lambda^*+d^*)}$. In Case $(a)$, it is readily checked that $\wt H(x_{i+1})+x_{i+1}\geq 0$. Hence, by choosing $p$ and $\zeta$ small such that $2p+\zeta<  \xi$, noting that $x_i=x_{i+1}-\zeta$, it follows that $\wt H(x_{i+1})+x_i-\frac32 p+\xi>\frac12 p>0$ for all $i\in\mathbb I$. We can thus use Lemma~\ref{lemma:mdp} to bound
	\be 
	\P{S_k-\varphi_1(k)\leq -(\wt H(x_{i+1})+x_i-\tfrac32p+\xi)\varphi_2(k) }\leq \exp\big(-\tfrac12(\wt H(x_{i+1})+x_{i+1}-\zeta-2p+\xi)^2\varphi_2(k)\big).
	\ee  
	In Case $(b)$, we instead bound the probability on the right-hand side from above by  one. We then apply Lemma~\ref{lemma:Dtail} to also bound $\P{D\geq k}\leq \exp\big(-\rho_1(k)-\tfrac12\rho_2(k)\big)$. As $\rho_1$ is increasing,
	\be \ba 
	\rho_1(k)\geq{}& \rho_1\big(\varphi_1^{-1}\big(t^*-\tfrac{1}{\lambda^*+d^*}\cK_\alpha(r(t))+(\wt H(x_{i+1})+\xi)\cK(t^*)\big)\big)\\ 
	={}& d^*t^*-\tfrac{d^*}{\lambda^*+d^*}\cK_\alpha(r(t))+d^*(\wt H(x_{i+1})+\xi)\cK(t^*)\\ 
	&+\cK_\alpha\big(t^*-\tfrac{1}{\lambda^*+d^*}\cK_\alpha(r(t))+(\wt H(x_{i+1})+\xi)\cK(t^*)\big),
	\ea\ee 
	where we use the definitions of $\alpha$ and $\cK_\alpha$ as in~\eqref{eq:alpha} and~\eqref{eq:Ks}, respectively, in the last step. Finally, we also note that $\rho_2(k)=((d^*)^2+o(1))\varphi_2(k)=((d^*)^2+o(1))\cK(t^*)$ since $d$ converges to $d^*$ and by using $(c)$ in Lemma~\ref{lemma:func}and Assumption~\hyperref[ass:K]{$\cK$}. Using both bounds in~\eqref{eq:largedegbound} and Assumption~\hyperref[ass:K]{$\cK$}, we arrive at 
	\be \ba\label{eq:expandub}
	\mathbb P\Big(D\geq k{}&,S_k\leq t^*-\tfrac{1}{\lambda^*+d^*}\cK_\alpha(r(t))-(x_i-p)\cK(t^*)\Big)\\ 
	\leq\exp\big({}&-d^*t^*+\tfrac{d^*}{\lambda^*+d^*}\cK_\alpha(r(t))-\cK_\alpha\big(t^*-\tfrac{1}{\lambda^*+d^*}\cK_\alpha(r(t))+(\wt H(x_{i+1})+\xi)\cK(t^*)\big)\\
	&-\big[d^*(\wt H(x_{i+1})+\xi)+\tfrac12(\wt H(x_{i+1})+x_{i+1}-\zeta-2p+\xi)^2+\tfrac12(d^*)^2+o(1)\big]\cK(t^*)\big)
	\ea\ee
	in Case $(a)$ and 
	\be \ba\label{eq:expandub2}
	\mathbb P\Big(D\geq k{}&,S_k\leq t^*-\tfrac{1}{\lambda^*+d^*}\cK_\alpha(r(t))-(x_i-p)\cK(t^*)\Big)\\ 
	\leq\exp\big({}&-d^*t^*+\tfrac{d^*}{\lambda^*+d^*}\cK_\alpha(r(t))-\cK_\alpha\big(t^*-\tfrac{1}{\lambda^*+d^*}\cK_\alpha(r(t))+(\wt H(x_{i+1})+\xi)\cK(t^*)\big)\\
	&-\big[d^*(\wt H(x_{i+1})+\xi)+\tfrac12(d^*)^2+o(1)\big]\cK(t^*)\big)
	\ea\ee 
	in Case $(b)$. We now make two further case distinctions according to the conditions in the statement of Proposition~\ref{prop:It}. In Case~\ref{item:2} we apply $(d)$ in Lemma~\ref{lemma:func} and Assumption~\hyperref[ass:Kalpha]{$\cK_\alpha$} to obtain 
	\be \label{eq:Kalphadif}
	\cK_\alpha\big(t^*-\tfrac{1}{\lambda^*+d^*}\cK_\alpha(r(t))+(\wt H(x_{i+1})+\xi)\cK(t^*)\big)=\cK_\alpha(r(t))+o(\cK(t)).
	\ee 
	As a result, we can simplify the upper bound in~\eqref{eq:expandub} to 
	\be 
	\exp\big(-d^*t^*-\tfrac{\lambda^*}{\lambda^*+d^*}\cK_\alpha(r(t))-\big[d^*(\wt H(x_{i+1})+\xi)+\tfrac12(\wt H(x_{i+1})+x_{i+1}-\zeta-2p+\xi)^2+\tfrac12(d^*)^2+o(1)\big]\cK(t^*)\big)
	\ee 
	and the one in~\eqref{eq:expandub2} to
	\be 
	\exp\big(-d^*t^*-\tfrac{\lambda^*}{\lambda^*+d^*}\cK_\alpha(r(t))-\big[d^*(\wt H(x_{i+1})+\xi)+\tfrac12(d^*)^2+o(1)\big]\cK(t^*)\big).
	\ee 
	These upper bounds remain valid in Case~\ref{item:1}, since we can change the $o(\cK(t))$ in~\eqref{eq:Kalphadif} to $o(1)$ (without the need for Assumption~\hyperref[ass:Kalpha]{$\cK_\alpha$}), as $\cK_\alpha(t)$ converges with $t$ in this case. This completes the upper bound for the first term on the second line of~\eqref{eq:unionupper}. Multiplying this with the exponential term on the first line of~\eqref{eq:unionupper} yields
	\be \ba 
	\exp\Big({}&\big[(\lambda^*+d^*)x_{i+1}-\tfrac12 (\wt H(x_{i+1})+x_{i+1}+d^*)^2\big]\cK(t^*)-(\xi-(\zeta+2p))\big[\wt H(x_{i+1})+x_{i+1}+d^*\big]\cK(t^*)\\ 
	&-\big[\tfrac12(\xi-(\zeta+2p))^2+(\zeta+2p)d^*-\delta x_{i+1} +o(1)\big]\cK(t^*)\Big).
	\ea \ee 
	in Case $(a)$ and 
	\be \ba 
	\exp{}&\big([\lambda^*x_{i+1}-d^*\wt H(x_{i+1})-\tfrac12(d^*)^2-d^*\xi+\delta x_{i+1}+o(1)]\cK(t^*)\big) \\ 
	&=\exp\big(\big[\lambda^*x_{i+1}-\tfrac{\lambda^*(d^*)^2}{2(\lambda^*+d^*)}-d^*\xi+\delta x_{i+1}\big]\cK(t^*)\big)
	\ea\ee 
	in Case $(b)$, where we use the definition of $\wt H$, as in~\eqref{eq:H}, in the last step. In Case $(a)$ we have $\wt H(x_{i+1})=H(x_{i+1})$, so that the terms in the first square brackets equal zero by the choice of $H$, and  the terms in the second square brackets are positive, since 
	\be 
	H(x_{i+1})+x_{i+1}+d^*=\sqrt{2(\lambda^*+d^*)x_{i+1}}\geq d^*, 
	\ee 
	as $x_{i+1}\geq (d^*)^2/(2(\lambda^*+d^*)$ in this case. Choosing $\delta, p$, and $\zeta$ sufficiently small with respect to $\xi=\xi(\eps)$ thus yields the desired result. In Case $(b)$ we have directly have that the exponential tends to zero by the choice of $x_{i+1}$ in this case and by choosing $\delta$ sufficiently small. This concludes the proof of the first result in Lemma~\ref{lemma:suboptwin}. 	
	
	We then consider the probability
	\be 
	\mathbb P_\cS\big(D^{\max}_{u^*-\eps,u^*+\eps}(s)\leq \ceil{\varphi_1^{-1}\big( t^*-\tfrac{1}{\lambda^*+d^*}\cK_\alpha(r(t))-(H(u^*)+\xi)\cK(t^*)\big)}, \text{ for all }s\in F_{p,t}\big).
	\ee 
	In the same manner as we bound the probability in~\eqref{eq:suboptinwindow}, but replacing $x_i$ and $x_{i+1}$ with $u^*-\eps$ and $u^*+\eps$, respectively, and setting $k:=\ceil{\varphi_1^{-1}( t^*-\tfrac{1}{\lambda^*+d^*}\cK_\alpha(r(t))-(H(u^*)+\xi)\cK(t^*)}$ instead, we can follow the same proof to conclude that this probability tends to one as $t$ tends to infinity when $p$ is sufficiently small with respect to $\xi$ (and thus with respect to $\eps$, by~\eqref{eq:xi}).
	
	Finally, we consider the probability
	\be 
	\mathbb P_\cS\big(D^{\max}_{\eps,t}(s)\leq \ceil{\varphi_1^{-1}\big( t^*-\tfrac{1}{\lambda^*+d^*}\cK_\alpha(r(t))-(d^*-\xi)\cK(t^*)\big)}, \text{ for all }s\in F_{p,t}\big).
	\ee 
	Again, in the same manner as we bound the probability in~\eqref{eq:suboptinwindow}, but replacing $x_i$ and $x_{i+1}$ with $-\eps$ and $0$, respectively, and setting $k:=\ceil{\varphi_1^{-1}( t^*-\tfrac{1}{\lambda^*+d^*}\cK_\alpha(r(t))-(d^*-\xi)\cK(t^*)}$ instead, we can follow the same proof to conclude that this probability tends to one as $t$ tends to infinity when $p$ is sufficiently small with respect to $\xi$ (and thus with respect to $\eps$, by~\eqref{eq:xi}).
\end{proof}

To conclude the section, we prove Lemma~\ref{lemma:outoptwin}. 

\begin{proof}[Proof of Lemma~\ref{lemma:outoptwin}]
	
	We consider the probability 
	\be \label{eq:prob3}
	\mathbb P_\cS\big( D^{\max}_{A,t}(s)\leq \ceil{\varphi_1^{-1}(t^*-\tfrac{1}{\lambda^*+d^*}\cK_\alpha(r(t)))} \text{ for all }s\in F_{p,t}\big),
	\ee 
	that all individuals born after time $\frac{d^*}{\lambda^*+d^*}t+A\cK(t^*)$ that are alive at time $s$ have degree at most $\ceil{\varphi_1^{-1}(t^*-\tfrac{1}{\lambda^*+d^*}\cK_\alpha(r(t)))}$ for all $s\in F_{p,t}$. We fix $\eps>0$ and set $k=\floor{\varphi_1^{-1}(t^*-\tfrac{1}{\lambda^*+d^*}\cK_\alpha(r(t)))}$ and
	\be
	B_{A,t}:=\cB\Big(\frac{d^*}{\lambda^*+d^*}t+\tfrac{1}{\lambda^*+d^*}\cK_\alpha(r(t))+A\cK(t^*),t\Big).
	\ee 
	We write the probability of interest as 
	\be\ba 
	1{}&-\Ps{\exists s\in F_{p,t}\ \exists v\in B_{A,t}: D^{(v)}\geq k, S_k^{(v)}\leq s-\sigma_v, S_{D^{(v)}+1}^{(v)}>s-\sigma_v}\\ 
	&\geq 1-\Ps{\exists s\in F_{p,t}\ \exists v\in B_{A,t}: D^{(v)}\geq k, S_k^{(v)}\leq s-\sigma_v}.
	\ea \ee  
	We bound $s$ from above in the second part of the event in the probability. Further,  we cover the interval $[\frac{d^*}{\lambda^*+d^*}t+\tfrac{1}{\lambda^*+d^*}\cK_\alpha(r(t))+A\cK(t^*), t]$  by smaller intervals of the form
	\be 
	\Big[\frac{d^*}{\lambda^*+d^*}t+\tfrac{1}{\lambda^*+d^*}\cK_\alpha(r(t))+(j-1)\cK(t^*),\frac{d^*}{\lambda^*+d^*}t+\tfrac{1}{\lambda^*+d^*}\cK_\alpha(r(t))+j\cK(t^*)\Big], \ee 
	where 
	\be 
	j\in \Big\{ A+1,\ldots ,\Big\lceil\frac{1}{\lambda^*+d^*}\frac{\lambda^*t-\cK_\alpha(r(t))}{\cK(t^*)}\Big\rceil\Big\}=:I_{A,\eps}. 
	\ee 
	We also define, for $j\in I_{A,\eps}$,  
	\be 
	B_{j,t}:=\cB\Big(\frac{d^*}{\lambda^*+d^*}t+\tfrac{1}{\lambda^*+d^*}\cK_\alpha(r(t))+(j-1)\cK(t^*),\frac{d^*}{\lambda^*+d^*}t+\tfrac{1}{\lambda^*+d^*}\cK_\alpha(r(t))+j\cK(t^*)\Big).
	\ee 
	Finally, we introduce the event 
	\be 
	\cE_M:=\big\{\sup_{t\geq 0}\e^{-\lambda^*\!t}|\cB(0,t)|\leq M\big\},
	\ee 
	which holds with probability at least $1-\eps $ when choosing $M=M(\eps)$ sufficiently large by Corollary~\ref{cor:growth}. Combining all this, we can bound the probability from below by 
	\be \label{eq:sumsplit}
	1-\mathbb P_\cS(\cE_M^c)-\mathbb P_\cS\bigg(\cE_M\cap\Big\{\exists v\in B_{A,t}: D^{(v)}\geq k, S_k^{(v)}\leq t+p\cK(t^*)-\sigma_v\}\bigg).
	\ee 
	As stated, we can bound the first probability from above by $\eps$, for any $\eps>0$, by choosing $M$ sufficiently large. It thus remains to bound the second probability from above. Using a union bound we obtain the upper bound 
	\be \label{eq:sumjbound}
	\sum_{j\in I_{A,\eps}}\!\!\Ps{\cE_M\cap\Big\{\exists v\in B_{j,t}: D^{(v)}\geq k, S_k^{(v)}\leq t+p\cK(t^*)-\sigma_v\Big\}}.
	\ee
	We bound each term in the sum separately. We  use the fact that $v\in B_{j,t}$ to bound $\sigma_v$ from below to arrive at the upper bound
	\be \ba 
	\mathbb P_\cS\bigg(\cE_M\cap\bigcup_{ v\in B_{j,t}}\Big\{{}& D^{(v)}\geq k, S_k^{(v)}\leq t^*-\tfrac{1}{\lambda^*+d^*}\cK_\alpha(r(t))-(j-1-p)\cK(t^*),\Big\}\bigg).
	\ea\ee 
	On the event $\cE_M$, we can bound the size of $B_{j,t}$ from above. Combined with a union bound and using the independence of $D$ and $S_k$, this yields the upper bound
	\be\ba \label{eq:probsplitnonopt} 
	M{}&\P{\cS}^{-1}\exp\Big(\frac{\lambda^*d^*}{\lambda^*+d^*}t+\tfrac{\lambda^*}{\lambda^*+d^*}\cK_\alpha(r(t))+j\lambda^*\cK(t^*)\Big)\\ 
	&\times \P{D\geq k}\P{S_k\leq t^*-\tfrac{1}{\lambda^*+d^*}\cK_\alpha(r(t))-(j-1-p)\cK(t^*)}.
	\ea \ee 
	We apply a Chernoff bound to the second probability on the second line, for some $\theta>0$ to be determined, to bound the second line from above by 
	\be \ba \label{eq:prodcomb}
	\P{D\geq k}\bigg(\prod_{i=0}^{k-1}\frac{b(i)+d(i)}{b(i)+d(i)+\theta}\bigg) \exp\Big(\theta \big[t^*-\tfrac{1}{\lambda^*+d^*}\cK_\alpha(r(t))-(j-1-p)\cK(t^*)\big]\Big).
	\ea\ee 
	By applying the bound $\log(1-x)\leq -x-x^2/2$ for $x\in[0,1)$ and similar to~\eqref{eq:sumineq}, we have
	\be\label{eq:prodbound}
	\prod_{i=0}^{k-1}\frac{b(i)+d(i)}{b(i)+d(i)+\theta} 
	\leq \exp\Big(-\theta\varphi_1(k)+\frac{\theta^2}{2}\varphi_2(k)\Big).
	\ee 
	Combined with Lemma~\ref{lemma:Dtail}, we can thus bound~\eqref{eq:prodcomb} from above by 
	\be \ba 
	\exp\Big(\theta\big(t^*-\tfrac{1}{\lambda^*+d^*}\cK_\alpha(r(t))-(j-1-p)\cK(t^*)\big)-\rho_1(k)-\frac12\rho_2(k)-\theta\varphi_1(k)+\frac{\theta^2}{2}\varphi_2(k)\Big).
	\ea\ee 
	Since $k=\ceil{\varphi_1^{-1}(t^*-\tfrac{1}{\lambda^*+d^*}\cK_\alpha(r(t)))}$ and $\varphi_1$ is increasing, we can write this as
	\be \ba 
	\exp{}&\Big(  -\rho_1(k)-\frac12 \rho_2(k)+ \frac{\theta^2}{2}\varphi_2(k)-\theta(j-1-p)\cK(t^*) \Big)\\ 
	&= \exp\Big(  -\rho_1(k)-\frac12 \rho_2(k)- \frac{(j-1-p)^2\cK(t^*)^2}{2\varphi_2(k)} \Big),
	\ea \ee 
	where the second step optimises over $\theta$ by taking $\theta=(j-1-p)\cK(t^*)/\varphi_2(k)$. Further, we write $\rho_1(k)=d^*\varphi_1(k)+\alpha(k)$ and use $(c)$ in Lemma~\ref{lemma:func} and Assumption~\hyperref[ass:K]{$\cK$} to conclude that $\rho_2(k)=((d^*)^2+o(1))\cK(t^*)$. For Case~\ref{item:2} in the statement of Proposition~\ref{prop:It}, we apply $(a)$ and $(d)$ in Lemma~\ref{lemma:func} and Assumption~\hyperref[ass:Kalpha]{$\cK_\alpha$} to the term $\alpha(k)$ to arrive at
	\be 
	\exp\Big(-\frac{\lambda^*d^*}{\lambda^*+d^*}t-\frac{\lambda^*}{\lambda^*+d^*}\cK_\alpha(r(t))-\Big[\tfrac12(j-1-p)^2+ \tfrac12(d^*)^2+o(1)\Big]\cK(t^*)\Big).
	\ee 
	In Case~\ref{item:1}, we have $\rho_1(k)=\alpha(k)=\cO(1)$ without the need for Assumption~\hyperref[ass:Kalpha]{$\cK_\alpha$}, so that the upper bound remains valid in Case~\ref{item:1}. Multiplying this with the remaining terms on the first line of~\eqref{eq:probsplitnonopt} yields
	\be\label{eq:bound1stline} 
	M\P{\cS}^{-1}\exp\Big(-\Big[\tfrac12(j-1-p)^2-j\lambda^*+\tfrac12(d^*)^2+o(1)\Big]\cK(t^*)\Big).
	\ee  
	Observe that this upper bound is summable in $j$ and tends to zero with $t$. We now choose $A\in\N$ large enough such that $\tfrac12(j-p-1)^2-j\lambda^*+\tfrac12(d^*)^2$ is strictly positive and increasing in $j$ on $[A,\infty)$. Using this in~\eqref{eq:sumjbound}, we thus obtain for some $c>0$, 
	\be\ba 
	\sum_{j\in I_{A,\eps}}\!\!\Ps{\cE_M\cap\Big\{\exists v\in B_{j,t}: D^{(v)}\geq k, S_k^{(v)}<t^*-\tfrac{1}{\lambda^*+d^*}\cK_\alpha(r(t))-(j-1-p)\cK(t^*)}\leq \e^{-c \cK(t^*)}. 
	\ea\ee 
	This completes the proof, as we then have for this choice of $A$ that
	\be 
	\mathbb P\big( D^{\max}_{A,t}(s)\leq \varphi_1^{-1}\big(t^*-\tfrac{1}{\lambda^*+d^*}\cK_\alpha(r(t))\big), \text{ for all }s\in F_{p,t}\big)\geq 1-\eps -\e^{-c \cK(t^*)},
	\ee 
	and $\eps$ is arbitrary.		 
\end{proof}

\section{Lack of persistence  in the `rich die young' regime}\label{sec:maxrdy}

In this section we analyse the random variable $I^\cont_t$, the birth-time of the oldest individual that has the largest offspring at time $t$ in the `rich die old' regime. Unlike the previous section, where we proved a precise scaling limit for $I_t^\cont$ in the `rich are old' regime when the death rates converge to a constant, we only provide a lower bound for $I_t^\cont$ in this section. However, this lower bound holds in a general setting where few assumptions on the birth and death rates are required. 

Recall $R$ from~\eqref{eq:R} and $\lambda^*$ from Assumption~\eqref{ass:C1}. Set 
\be\label{eq:dinfsup} 
\underline d:=\liminf_{i\to\infty}d(i) \qquad \text{and}\qquad \overline d:=\limsup_{i\to\infty}d(i) ,
\ee 
and, for a large constant $A>0$, define the function 
\be \label{eq:w}
w(k):=\begin{cases} 
	\varphi_2(k)&\mbox{if } \lim_{j\to\infty}\varphi_2(j)=\infty, \\
	A&\mbox{otherwise,}
\end{cases} \qquad\text{for } k\in\N_0.
\ee  
Furthermore, fix $K,K'>0$, set  $k:=\ceil{ \rho_1^{-1}(K'\log t)}$, and define 
\be \label{eq:Fstar} 
F^*_{K',t}:=[t-w(k), t+w(k)]  \qquad\text{and}\qquad \ B_{K,t}:=\cB\Big(\frac{R}{\lambda^*+R}t-K\log t,\frac{R}{\lambda^*+R}t+K\log t\Big).
\ee 
When $\underline d>R$, it follows from Proposition~\ref{prop:oldage} that the oldest alive individual at time $t$ is an element of $B_{K,t}$. We have the following result, which tells us that no individual in $B_{K,t}$ is alive and additionally has a `large' offspring at time $s$, for any $s\in F_{K',t}^*$. 

\begin{proposition}[Oldest alive individuals do not have large offspring]\label{prop:nooldhighdeg}
	Suppose that $b$ and $d$ are such that Assumptions~\eqref{ass:A1} and~\eqref{ass:C1} are satisfied, that $b$ tends to infinity, and that $\underline d>R$. For any $K>0$ there exists $K'>0$ such that 
	\be 
	\lim_{t\to\infty}\Ps{\forall s\in  F_{K',t}^*\  \exists v\in B_{K,t} : v\ \mathrm{alive} \ \mathrm{at}\ \mathrm{time}\ s,\ \deg^{(v)}(s)\geq \lceil \rho_1^{-1}(K'\log t)\rceil}=0.
	\ee 
\end{proposition}

\begin{remark}
	Whilst $\rho_1$ need not be invertible in general (when $d(i)=0$ for some $i\in\N_0$), since $\underline d>0$ in Proposition~\ref{prop:divdeathhighdeg}, it follows that $\rho_1$ is invertible on $[C,\infty)$, where $C$ is a sufficiently large constant, so that the quantity $k$ is well-defined for all large enough $t$. 	\ensymboldremark
\end{remark} 

\begin{proof}		
	As we can bound the birth-time of individuals $v\in B_{K,t}$ from above and below, we obtain the upper bound
	\be\ba \label{eq:nolargedeg}
	\mathbb P_\cS({}&\forall s\in F_{K',t}^*\ \exists v\in B_{K,t}: v\text{ alive at time }s, \deg^{(v)}(s)\geq k)\\ 
	={}& \Ps{\bigcap_{s\in F_{K',t}^*}\bigcup_{v\in B_{K,t}}\!\!\!\{D^{(v)}\geq k, S_k^{(v)}\leq s-\sigma_v, S_{D^{(v)}+1}^{(v)}>s-\sigma_v\}}\\ 
	\leq{}& \mathbb P_{\cS}\Bigg(\!\bigcup_{v\in B_{K,t}}\!\!\!\!\!\Big\{D^{(v)}\geq k, S_k^{(v)}\leq \frac{\lambda^*}{\lambda^*+R}t+K\log t+w(k), S_{D^{(v)}+1}^{(v)}>\frac{\lambda^*}{\lambda^*+R}t-K\log t-w(k)\Big\}\!\Bigg).
	\ea\ee 
	By conditioning on $|B_{K,t}|$, a union bound  yields the upper bound
	\be \ba \label{eq:unionubR}
	\mathbb P\bigg({}&|B_{K,t}|\geq \exp\Big(\frac{\lambda^*R}{\lambda^*+R}t+\lambda^*(K+\delta)\log t\Big)\bigg)+	\exp\Big(\frac{\lambda^*R}{\lambda^*+R}t+\lambda^*(K+\delta)\log t\Big)\\ 
	{}&\times \P{D\geq k, S_k\leq \frac{\lambda^*}{\lambda^*+R}t+K \log t+w(k), S_{D+1}>\frac{\lambda^*}{\lambda^*+R}t-K \log t-w(k)}.
	\ea\ee 
	The probability on the first line tends to zero with $t$ by~\eqref{eq:ubB} in Corollary~\ref{cor:growth} (using $r=\frac{R}{\lambda^*+R}t-K\log t$, $s=\frac{R}{\lambda^*+R}t+K\log t$, and $u=\delta \log t$). We further split the probability on the second line in two terms, namely 
	\be \ba\label{eq:splitub}
	\mathbb P\bigg({}&D\geq k, S_k<\frac{\lambda^*}{\lambda^*+R}t-K\log t-w(k), S_{D+1}>\frac{\lambda^*}{\lambda^*+R}t-K\log t-w(k)\bigg)\\ 
	&+\P{D\geq k}\P{ S_k\in\Big(\frac{\lambda^*}{\lambda^*+R}t-K\log t-w(k),\frac{\lambda^*}{\lambda^*+R}t+K\log t+w(k)\Big]}.
	\ea\ee 
	We bound both terms from above, starting with the one on the first line. As $\underline d>R$, there exists $I\in\N$ such that $d(i)>R-1/t$ for all $i\geq I$ and all $t> 0$. As $k\geq I$ for all $t$ sufficiently large, we can use Lemmas~\ref{lemma:Dtail} and~\ref{lemma:survdeg} to bound the first term on the right-hand side from above by 
	\be 
	\e^{-\rho_1(k)}\E{\ind_{\{S_k<\frac{\lambda^*}{\lambda^*+R}t-K\log t-w(k)\}}\exp\Big(\big(R-\tfrac1t\big)\Big[S_k-\Big(\frac{\lambda^*}{\lambda^*+R}t-K\log t-w(k)\Big)\Big]\Big)}.
	\ee 
	We omit the indicator to arrive at the upper bound
	\be \label{eq:exprho}
	\exp\bigg(-\rho_1(k)-\big(R-\tfrac1t\big)\Big(\frac{\lambda^*}{\lambda^*+R}t-K\log t-w(k)\Big)\bigg)(Rt)^{|[k-1]_R|}\!\!\!\!\prod_{\substack{i=0\\ i\not\in[k-1]_R}}^{k-1}\!\!\!\!\frac{b(i)+d(i)}{b(i)+d(i)-R+1/t},
	\ee 
	where we recall $[k-1]_R$ from~\eqref{eq:I-1R} as the set of indices $i\in\{0,\ldots, k-1\}$ such that $b(i)+d(i)=R$. 
	We bound, using that $1+x\leq \e^x$ for all $x\in \R$,
	\be \label{eq:prodwithoutR}
	\prod_{\substack{i=0\\ i\not\in[k-1]_R}}^{k-1}\!\!\!\!\frac{b(i)+d(i)}{b(i)+d(i)-R+1/t}\leq \exp\bigg(R\!\!\!\sum_{\substack{i=0\\ i\not\in[k-1]_R}}^{k-1}\!\!\!\!\frac{1}{b(i)+d(i)-R}\bigg)= \exp\big(R\varphi_1(k)+\cO(\varphi_2(k))\big), 
	\ee 
	where the final step uses that $b$ tends to infinity. Furthermore, we have $\varphi_2=o(\varphi_1)$ and $\varphi_1(k)\leq \rho_1(k)/(\underline d-\eps)$ for any $\eps>0$ and $k$ large by $(a)$ and $(b)$ in Lemma~\ref{lemma:func}, since $b$ tends to infinity. Additionally, by the choice of $w$, it follows that $w(k)=o(\rho_1(k))$.  Combining this with~\eqref{eq:prodwithoutR} in~\eqref{eq:exprho}, we arrive at the upper bound 
	\be 
	\exp\bigg(-\Big(1-\frac{R}{\underline d-\eps}+o(1)\Big)\rho_1(k)-\frac{\lambda^*R}{\lambda^*+R}t+[RK+|[k-1]_R|+o(1)]\log t\Big).
	\ee 
	As $\underline d >R$, it follows that $[k-1]_R\leq L$ for some constant $L\in \N$, uniformly in $k$. We can further use the choice of $k$ and the fact that $\rho_1$ is increasing to finally obtain 
	\be \ba \label{eq:firsttermbound}
	\mathbb P{}&\Big(D\geq k, S_k<\frac{\lambda^*}{\lambda^*+R}t-K\log t-w(k), S_{D+1}>\frac{\lambda^*}{\lambda^*+R}t-K\log t-w(k)\Big)\\ 
	&\leq \exp\bigg(-\frac{\lambda^*R}{\lambda^*+R}t+\big[RK+L-\big(1-\tfrac{R}{\underline d-\eps}\big)K'+o(1)\big]\log t\Big).
	\ea\ee 
	We then bound the term on the second line of~\eqref{eq:splitub} from above. We relax the upper bound on $S_k$, again use Lemma~\ref{lemma:Dtail}, and apply a Chernoff bound with parameter $\theta:=R-1/t$, to obtain 
	\be 
	\e^{-\rho_1(k)}\P{S_k>\frac{\lambda^*}{\lambda^*+R}t-K\log t-w(k)}\leq \exp\bigg(-\rho_1(k)-\theta\Big[\frac{\lambda^*}{\lambda^*+R}t-K\log t-w(k)\Big]\bigg)\E{\e^{\theta S_k}}.
	\ee  
	By the choice of $\theta$, we thus arrive at 
	\be \label{eq:closechernoff}
	\exp\bigg(-\rho_1(k)-\Big(R-\frac1t\Big)\Big[\frac{\lambda^*}{\lambda^*+R}t-K\log t-w(k)\Big]\bigg)(Rt)^{|[k-1]_R|}\!\!\!\!\prod_{\substack{i=0\\ i\not\in [k-1]_R}}^{k-1}\!\!\!\!\frac{b(i)+d(i)}{b(i)+d(i)-R+1/t},
	\ee
	which is the same upper bound as in~\eqref{eq:exprho}. The second line of~\eqref{eq:unionubR} is thus at most
	\be 
	\exp\bigg(-\frac{\lambda^*R}{\lambda^*+R}t+\big[RK+L-\big(1-\tfrac{R}{\underline d-\eps}\big)K'+o(1)\big]\log t\bigg).
	\ee 
	Multiplied with the exponential term on the first line of~\eqref{eq:unionubR}, we finally arrive at the upper bound
	\be \ba 
	\mathbb P(\forall {}&s\in F^*_{K',t}\ \exists v\in B_{K,t}: v\text{ alive at time }s, \deg^{(v)}(s)\geq k)\\
	\leq{}& \exp\bigg(\big[\lambda^*(K+\delta)+RK+L-\big(1-\tfrac{R}{\underline d-\eps}\big)K'+o(1)\big]\log t\bigg)+o(1),
	\ea \ee 
	which tends to zero for $K'$ sufficiently large.		
\end{proof}

Now that we know that the oldest  individuals alive at time $t$ do not have large offspring, it remains to show that there are individuals alive at time $t$, born much  later than the oldest individual, that do indeed attain such large degrees.  We abuse notation to write
\be \label{eq:overlined}
\overline d(i):=\sup_{j\leq i}d(j)\qquad \text{for }i\in\N_0. 
\ee 
It  follows that $\overline d:=\limsup_{i\to\infty}d(i)=\lim_{i\to\infty}\overline d(i)$. We then have the following result. 

\begin{proposition}\label{prop:divdeathhighdeg}
	Suppose that $b$ and and $d$ are such that Assumptions~\eqref{ass:A1} and~\eqref{ass:C1} are satisfied. Recall $R$ from~\eqref{eq:R}, $\underline d,\overline d$ from~\eqref{eq:dinfsup} and $F^*_{K',t}$ from~\eqref{eq:Fstar}. Suppose that $\underline d>0$ and that $b$ tends to infinity, such that $\overline d(k)=o(b(k))$ and $b(k)=\cO(k)$. Finally, assume that $b(k)$ and $\overline d(k)$ are regularly varying in $k$ with a non-negative exponent. For any $K,K'>0$, with $k:=\lceil \rho_1^{-1}(K'\log t)\rceil$,		
	\be 
	\lim_{t\to\infty}\Ps{\forall s\in F^*_{K',t}\ \exists v\in \cB\Big(\frac{R}{\lambda^*+R}t+K\log t, t\Big): v\ \mathrm{alive}\ \mathrm{at}\ \mathrm{time}\ s, \deg^{(v)}(s)\geq k}=1.
	\ee 
\end{proposition} 

The combination of Propositions~\ref{prop:nooldhighdeg} and~\ref{prop:divdeathhighdeg} $I^\cont_t\gg O^\cont_t$ when $\underline d>R$. It remains to translate this result to the quantities $I_n$ and $O_n$ to prove Theorem~\ref{thrm:biggerR}, which is carried out in Section~\ref{sec:mainproofs}. 

\begin{proof}
	To bound the probability in the proposition statement from below, we only consider individuals in
	\be 
	B_{w,t}:=\cB\big( t-\varphi_1(k), t-\varphi_1(k)+w(k)\big). 
	\ee 
	Note that, since $\underline d>0$, it follows that $\varphi_1(k)=\cO(\rho_1(k))=\cO(\log t)$ by $(b)$ in Lemma~\ref{lemma:func}. Similarly, regardless of whether $\varphi_2$ tends to infinity or not, we have $w(k)=o(\varphi_1(k))$ since $b$ diverges and by $(a)$ in Lemma~\ref{lemma:func}. As a result, 
	\be 
	B_{w,t}\subset \cB\Big(\frac{R}{\lambda^*+R}t+K\log t,t\Big).  
	\ee 
	We show that there exist individuals in the set $B_{w,t}$  who are alive at time $s$ and have produced at least $k$ children by time $s$, for all $s\in F^*_{K',t}$. By bounding the birth-times of individuals in $B_{w,t}$ and $s$ from above and below, we obtain the lower bound
	\be \ba 
	\mathbb P_\cS\bigg(\forall{}& s\in F^*_{K',t}\ \exists v\in \cB\Big(\frac{R}{\lambda^*+R}t+K\log t,t\Big): v\text{ alive at time }s, \deg^{(v)}(s)\geq k\bigg)\\ 
	\geq{}&\mathbb P_\cS\Bigg(\bigcup_{v\in B_{w,t}}\Big\{D^{(v)}\geq k, S_k^{(v)}\leq \varphi_1(k)-2w(k), S_{D^{(v)}+1}^{(v)}> \varphi_1(k)+w(k)\Big\} \Bigg).
	\ea\ee 
	With a similar approach as in the steps between~\eqref{eq:utbound} and~\eqref{eq:utbound2} and by leveraging the independence of the reproduction processes of distinct individuals (when not conditioning on survival), we obtain for any $m\in\N$ the lower bound
	\be  
	1-\Ps{|B_{w,t}|\leq m} -\P{\cS}^{-1}\big(1-\P{D\geq k, S_k\leq \varphi_1(k)-2w(k), S_{D+1}> \varphi_1(k)+w(k)}\big)^m.
	\ee 
	We now set 
	\be 
	m:=\exp\big(\lambda^*(t-\varphi_1(k)+\tfrac12w(k))\big), 
	\ee 
	and use that $1-x\leq \e^{-x}$ for all $x\in\R$ to obtain the lower bound 
	\be \ba \label{eq:dinftyprodlb}
	1{}&-\frac{1}{\P{\cS}}\exp\Big(-\P{D\geq k, S_k\leq \varphi_1(k)-2w(k), S_{D+1}> \varphi_1(k)+ w(k)}\e^{\lambda^*(t-\varphi_1(k)+\frac12w(k)) }\Big)\\ 
	&-\Ps{|B_{w,t}|\leq \exp\big(\lambda^*(t-\varphi_1(k)+\tfrac12w(k))\big)}.
	\ea \ee
	We first bound the probability on the first line from below and later show that the term on the second line can be made arbitrarily small. To this end, we condition on the event $\{D\geq k\}$ and the random variable $S_k$ to arrive at  
	\be \ba \label{eq:lbbeforedom}
	\mathbb P({}&D\geq k, S_k\leq \varphi_1(k)-2w(k), S_{D+1}> \varphi_1(k)+ w(k))\\
	&=\P{D\geq k}\E{\ind_{\{S_k< \varphi_1(k)-2w(k)\}}\P{\sum_{i=k}^{k+D_k} E_i> \varphi_1(k)+ w(k)-S_k\,\Bigg|\, S_k}}.
	\ea \ee 
	We bound the expected value from below by further restricting the value of $S_k$ to the interval  $\cI_k:=( \varphi_1(k)-3w(k), \varphi_1(k)-2w(k))$ in the indicator and then bounding $S_k$ from below by $\varphi_1(k)-3w(k)$ in the probability. This yields the lower bound
	\be \label{eq:genlb}
	\P{D\geq k}\P{S_k\in \cI_k}\P{\sum_{i=k}^{k+D_k} E_i> 4w(k)}.
	\ee 
	To bound the third probability from below, we distinguish between the cases $\overline d<\infty$ and $\overline d=\infty$. 
	
	$\boldsymbol{\overline d<\infty. }$\ There exists $C_1\geq \overline d$ such that $d(i)\leq C_1$ for all $i\geq k$, so that Lemma~\ref{lemma:stochdomexp} yields
	\be 
	\P{\sum_{i=k}^{k+D_k} E_i> 4w(k)}\geq \e^{-4C_1w(k)}\geq \e^{-C_2 \overline d(k)w(k)},
	\ee 
	where $C_2>0$ is large enough so that $C_2\overline d(k)\geq C_1$ for $k$ large, which is possible since $\overline d(k)$ converges to $\overline d$ as $k$ tends to infinity. 
	
	$\boldsymbol{\overline d=\infty. }$\ We recall $\overline d(i)$ from~\eqref{eq:overlined} and introduce the random variables $(\overline E_i)_{i\in\N_0}$ and $\overline D_k$, where
	\be 
	\overline E_i\sim\text{Exp}(b(i)+\overline d(i))\qquad\text{and}\qquad \P{\overline D_k\geq\ell}=\prod_{i=k}^{k+\ell-1}\frac{b(i)}{b(i)+\overline d(i)}. 
	\ee 
	Since $d(i)\leq \overline d(i)$ for all $i\in\N_0$, it follows that $E_i\succeq \overline E_i$ for all $i\in\N_0$ and $D_k\succeq \overline D_k$. Hence, 
	\be 
	\P{\sum_{i=k}^{k+D_k} E_i> 4w(k)}\geq \P{\overline d(k)\sum_{i=k}^{k+\overline D_k} \overline E_i> 4\overline d(k)w(k)}.
	\ee 
	We apply Lemma~\ref{lemma:remlifediv} to the probability on the right-hand side, i.e.\ when using the sequences $(b(k))_{k\in\N_0}$ and $(\overline d(k))_{k\in\N_0}$, for which we need to check several conditions. First, we note that $\overline d(k)$ is increasing by definition and tends to infinity with $k$ since $\overline d=\infty$. Then, by assumption $\overline d(k)$ is regularly varying with non-negative exponent (and increasing), and $\overline d(k)=o(b(k))$. Since $b(k)=\cO(k)$ and $\overline d(k)$ tends to infinity, it follows that $b(k)=o(k\overline d(k))$.  Finally, we show that $t_k:=4\overline d(k)w(k)$ satisfies the required conditions. As $w(k)\geq A$ for all $k$ sufficiently large and $\overline d(k)$ tends to infinity with $k$, it follows that $\liminf_{k\to\infty}t_k>0$. Then, recall that $b$ is regularly varying with non-negative exponent and that $\overline d(k)=o(b(k))$. If $b$ and $d$ are such that $\varphi_2$ tends to infinity, 
	\be 
	w(k)=\varphi_2(k)=\sum_{i=0}^{k-1}\frac{1}{(b(i)+d(i))^2}\leq \sum_{i=0}^{k-1}\frac{1}{b(i)^2}=\cO\Big(\frac{k}{b(k)^2}\Big).
	\ee 
	As a result, $t_k=\cO(k\overline d(k)/b(k)^2)=o(k\overline d(k)/b(k))$, where the final step holds since $b$ tends to infinity. When, instead, $\varphi_2$ converges, then $t_k=4A\overline d(k)=o(kd(k)/b(k))$ when $b(k)=o(k)$. As a result, in both cases we obtain the lower bound
	\be 
	\P{\overline d(k)\sum_{i=k}^{k+\overline D_k} \overline E_i> 4\overline d(k)w(k)}\geq \e^{-(4+o(1))\overline d(k)w(k)}.
	\ee 
	When $\varphi_2$ converges but $b(k)=\cO(k)$ rather than $b(k)=o(k)$, we use Remark~\ref{rem:linb} to obtain a weaker lower bound which holds under the weaker assumption $t_k=\cO(\overline d(k))$, which is the case when $t_k=4A\overline d(k)$ regardless of the value of $A$. This weaker lower bound yields
	\be 
	\P{\overline d(k)\sum_{i=k}^{k+\overline D_k} \overline E_i> 4\overline d(k)w(k)}\geq \e^{-C_3 \overline d(k)w(k)},
	\ee 
	for some large constant $C_3>4$.  It thus follows from Lemma~\ref{lemma:remlifediv} (and Remark~\ref{rem:linb}) that, with $C_4=\max\{C_2,C_3\}$,
	in both  cases $\overline d<\infty$ and $\overline d=\infty$ we arrive the lower bound 
	\be 
	\P{\sum_{i=k}^{k+D_k} E_i> 4w(k)}\geq \e^{-C_4\overline d(k)w(k)},
	\ee 
	for  $k$ sufficiently large. Using this in~\eqref{eq:genlb}, we arrive at
	\be \label{eq:genlb2}
	\P{D\geq k}\P{S_k\in \cI_k}\P{\sum_{i=k}^{k+D_k} E_i> 4w(k)}\geq \P{D\geq k}\P{S_k\in \cI_k}\e^{-C_4\overline d(k)w(k)}.
	\ee
	To conclude the proof, we distinguish between whether $\varphi_2$ tends to infinity or converges. 
	
	\textbf{$\boldsymbol{\varphi_2}$ tends to infinity. } It follows from~\eqref{eq:w} that $w(k)=\varphi_2(k)$. By Lemma~\ref{lemma:mdp},
	\be  \label{eq:Skmdp}
	\P{S_k\in \cI_k}=\P{S_k\in(\varphi_1(k)-3\varphi_2(k),\varphi_1(k)-2\varphi_2(k))} =\exp\big(- (2+o(1))\varphi_2(k)\big).
	\ee
	Combined with Lemma~\ref{lemma:Dtail} and using that $\rho_2=o(\rho_1)$ since $d=o(b)$ by $(a)$ in Lemma~\ref{lemma:func}, we can thus write the lower bound in~\eqref{eq:genlb2} as 
	\be 
	\exp\big(-(1+o(1))\rho_1(k)-(2+o(1))\varphi_2(k)-C_4\overline d(k)\varphi_2(k)\big)=\exp\big(-(1+o(1))\rho_1(k)), 
	\ee 
	where we use that $b$ tends to infinity and is regularly varying with non-negative exponent  and that $\overline d(k)=o(b(k))$, so that $\varphi_2(k)=\cO(\overline d(k)\varphi_2(k))=o(\varphi_1(k))$, and $\varphi_1(k)=\cO(\rho_1(k))$ since $\underline d>0$ by $(a)$ and $(b)$ in Lemma~\ref{lemma:func}. By the choice of $k$, we thus arrive at the lower bound
	\be 
	\P{D\geq k, S_k\leq \varphi_1(k)-2w(k), S_{D+1}> \varphi_1(k)+ w(k)} \geq \e^{-(K'+o(1))\log t}.
	\ee 
	Using this in~\eqref{eq:dinftyprodlb} yields the lower bound
	\be \ba 
	1{}&-\exp\Big(-\exp\big(-(K'+o(1))\log t+\lambda^*(t-\varphi_1(k)+\tfrac12\varphi_2(k))\big)\Big)\\
	&-\Ps{|B_{w,t}|\leq \exp\big(\lambda^*(t-\varphi_1(k)+\tfrac12\varphi_2(k))\big)}
	\ea\ee 
	Again using that $\varphi_2=o(\varphi_1)$ and $\varphi_1=\cO(\rho_1)$, so that $\varphi_1(k)=o(t)$ and  $\varphi_2(k)=o(t)$, the exponential term equals $\exp(-\exp((\lambda^*+o(1))t))$, which tends to zero with $t$. It remains to show that the probability on the second line tends to zero. Here, we use~\eqref{eq:lbB} in Corollary~\ref{cor:growth}, with $r=t-\varphi_1(k)$, $s=t-\varphi_1(k)+\varphi_2(k)$, and $u=\frac12\varphi_2(k)$, to arrive at the desired result. 
	
	\textbf{$\boldsymbol{\varphi_2}$ converges. } We  have
	\be \ba 
	\P{S_k\in \cI_k}=\P{S_k\in( \varphi_1(k)-3A, \varphi_1(k)-2A)}=p_A+o(1), 
	\ea \ee
	for some constant $p_A\in(0,1)$. Indeed,  $S_k-\varphi_1(k)=S_k-\E{S_k}$ converges almost surely in this case, as it is a martingale whose quadratic variation equals $\varphi_2$, which converges by assumption. Again, combined with Lemma~\ref{lemma:Dtail}, this yields 
	\be \ba 
	\mathbb P{}&(D\geq k, S_k<t-w(k)-(t- \varphi_1(k)+w(k)), S_{D+1}>t+w(k)-(t- \varphi_1(k)))\\ 
	&\geq\exp\big(-(1+o(1))\rho_1(k)-AC_4\overline d(k)\big).
	\ea\ee 
	Using this in~\eqref{eq:dinftyprodlb}, we arrive at the lower bound
	\be \ba 
	1{}&-\exp\Big(-\exp\big(-(1+o(1))\rho_1(k)-AC_4\overline d(k)+\lambda^*(t-\varphi_1(k)+2A)\big)\Big)\\
	&-\Ps{|B_{w,t}|\leq \exp\big(\lambda^*(t-\varphi_1(k)+\tfrac12A)\big)}.
	\ea\ee 
	As in the first case, $\varphi_1(k)=o(t)$. Furthermore, by the assumption that $\overline d(k)=o(b(k))=\cO(k)$ it follows that $\overline d(k)=o(t)$ by the choice of $k$, irrespective of the value of $K'$. The exponential term on the first line thus tends to zero with $t$ independently of the choice of $A$ and $K'$. To bound the probability on the second line, we write 
	\be \ba 
	\mathbb P_\cS{}&\big(|B_{w,t}|\leq \exp\big(\lambda^*(t-\varphi_1(k)+\tfrac12A)\big)\big)\\ 
	={}&\Ps{|\cB(0,t- \varphi_1(k)+A)|-|\cB(t-\varphi_1(k),t- \varphi_1(k)+A)|\leq \e^{\lambda^*(t- \varphi_1(k)+\frac12 A)}}\\ 
	\leq{}& \Ps{|\cB(0,t- \varphi_1(k)+A)|\e^{-\lambda^*(t- \varphi_1(k)+A)}\leq 2\e^{-\frac12\lambda^*A}}\\
	&+\Ps{|\cB(0,t-\varphi_1(k))|\e^{-\lambda^*(t-\varphi_1(k))}\geq \e^{\frac12\lambda^*A} },
	\ea\ee  
	where the final step uses a union bound. We then bound the right-hand side from above by
	\be 
	\mathbb P\Big(\inf_{t\geq0}|\cB(0,t)|\e^{-\lambda^*t}\leq 2\e^{-\frac12\lambda^*A}\Big)+\mathbb P\Big(\sup_{t\geq0}|\cB(0,t)|\e^{-\lambda^*t}\geq \e^{\frac12\lambda^*A}\Big).  		
	\ee 
	We can bound either term from above by $\eps/2$ when choosing $A=A(\eps)$ sufficiently large by using~\eqref{eq:supB} and~\eqref{eq:infB} in Proposition~\ref{prop:growth}. Combining both cases, we thus arrive at 
	\be 
	\limsup_{t\to\infty}\Ps{\forall s\in F^*_{K',t}\ \exists v\in \cB\Big(\frac{R}{\lambda^*+R}t+K\log t, t\Big): v\ \mathrm{alive}\ \mathrm{at}\ \mathrm{time}\ s, \deg^{(v)}(s)\geq k}\geq 1-\eps. 
	\ee 
	As $\eps$ is arbitrary, we arrive the desired result and conclude the proof.
\end{proof}

\section{Proofs of the main results}\label{sec:mainproofs}

In this section we combine the results proved in Sections~\ref{sec:old}, ~\ref{sec:max}, and~\ref{sec:maxrdy} (Propositions~\ref{prop:oldage}, ~\ref{prop:It}, \ref{prop:nooldhighdeg}, and~\ref{prop:divdeathhighdeg} in particular) to prove the main results presented in Section~\ref{sec:results}. 

\begin{proof}[Proof of Theorem~\ref{thrm:conv}]	
	Recall that $\tau_n$, as in~\eqref{eq:taun}, is the stopping time at which exactly $n$ many births or deaths have occurred in the branching process $\bp$. Also, recall that $\cB(0,t)$ and $\cA^\cont_t$  denote the number of births that occur in $\bp$ up to time $t$ and the number of alive individuals in $\bp(t)$, respectively. By Proposition~\ref{prop:growth} and Equation~\eqref{eq:charlim}, there exist random variables $W_1,W_2$ such that 
	\be 
	|\cB(0,t)|\e^{-\lambda^*\!t}\toas W_1, \qquad\text{and}\qquad \ |\cA^\cont_t|\e^{-\lambda^*\!t}\toas W_2. 
	\ee 
	Moreover, $W_1$ and $W_2$ are strictly positive almost surely conditionally on the event $\cS$ that the branching process survives. If we let $N(t)$ denote the the number of births and deaths that occur up to time $t$, then we observe that
	\be 
	N(t)=|\cB(0,t)|+(|\cB(0,t)|-\cA^\cont_t),\qquad t\geq0.
	\ee 
	Note that $N(t)\geq 0$ for all $t\geq 0$, and that $N(t)$ is increasing in $t$. Since $(|\cB(0,t)|)_{t\geq 0}$ and $|(\cA^\cont_t|)_{t\geq 0}$ are c\'adl\'ag, it follows that $N(\tau_n)=n$. Hence,
	\be\label{eq:taunconv} 
	\tau_n-\frac{1}{\lambda^*}\log n\toas -\frac{1}{\lambda^*}\log (2W_1-W_2). 
	\ee 
	As $W_2\leq W_1$ since $|\cA^\cont_t|\leq |\cB(0,t)|$ and $W_2>0$ almost surely conditionally on $\cS$, the limit is finite $\mathbb P_{\cS}$-almost surely. Since we assume that $\varphi_1$ and $\varphi_2$ diverge, and thus that $\cK$ diverges, conditionally on $\cS$, the event 
	\be \label{eq:tauconv}
	\tau_n\in \Big[\frac{1}{\lambda^*}\log n-p \cK\big(\tfrac{1}{\lambda^*}\log n\big), \frac{1}{\lambda^*}\log n+p \cK\big(\tfrac{1}{\lambda^*}\log n\big)\Big]=:F_{p,n}',
	\ee 
	holds with high probability as $n$ tends to infinity, for any fixed $p>0$. Then, 
	\be \ba \label{eq:Otaunbound}
	\mathbb P_\cS{}&\bigg(\bigg|\frac{1}{\cK\big(\tfrac{1}{\lambda^*+d^*}\log n\big)}\Big(O_{\tau_n}^\cont-\frac{d^*}{\lambda^*(\lambda^*+d^*)}\log n -\frac{1}{\lambda^*+d^*}\cK_{\alpha}\big(r\big(\tfrac{1}{\lambda^*}\log n\big)\big)\Big)\bigg|<\eps\bigg)\\ 
	\geq{}& \mathbb P_\cS\bigg(\bigg|\frac{1}{\cK\big(\tfrac{1}{\lambda^*+d^*}\log n\big)}\Big(O_s^\cont-\frac{d^*}{\lambda^*(\lambda^*+d^*)}\log n -\frac{1}{\lambda^*+d^*}\cK_{\alpha}\big(r\big(\tfrac{1}{\lambda^*}\log n\big)\big)\Big)\bigg|<\eps, \forall s\in F_{p,n}'\!\bigg)\\
	&-\Ps{\tau_n\not \in F_{p,n}'}.
	\ea \ee  	
	Applying Proposition~\ref{prop:oldage} (\eqref{eq:Otconv} in particular) with $t=\frac{1}{\lambda^*}\log n$ to the first probability on the right-hand side,  with $p$ sufficiently small, and~\eqref{eq:tauconv} to the final probability yields that the left-hand side converges to one as $n$ tends to infinity. We now define, for ease of writing, 
	\be 
	\ell_n(c):=\frac{d^*}{\lambda^*+d^*}\log n +\frac{\lambda^*}{\lambda^*+d^*}\cK_\alpha\big(r\big(\tfrac{1}{\lambda^*}\log n \big)\big)+c\lambda^* \cK\big(\tfrac{1}{\lambda^* +d^*}\log n\big),
	\ee 
	where $c\in \R$ is a constant. We then write the event 
	\be  \label{eq:bigeventOn}
	\bigg\{\bigg|\frac{1}{\cK(\frac{1}{\lambda^*+d^*}\log n)}\Big(\log O_n -\frac{d^*}{\lambda^*+d^*}\log n-\frac{\lambda^*}{\lambda^*+d^*}\cK_\alpha\big(r\big(\tfrac{1}{\lambda^*}\log n\big)\big)\Big)\bigg|\leq 2\lambda^*\eps \bigg\}
	\ee 
	as the event 
	\be \ba \label{eq:smalleventOn}
	\big\{O_n \in \big[\exp\big(\ell_n(-2\eps)\big),\exp\big(\ell_n(2\eps)\big)\big]\big\}.
	\ea \ee 
	The random variable $O_n$ equals the label of the oldest alive vertex in $T_n$, whereas $O_{\tau_n}^\cont$ equals the birth-time of the oldest alive vertex in $\bp(\tau_n)$. Since $O_n\overset d=N(O_{\tau_n}^\cont)$ (see~\eqref{eq:OtItequiv}), we thus need to know the number of births and deaths that have occurred in $\bp(O_{\tau_n}^\cont)$. Hence, we bound
	\be\ba\label{eq:Ondiscrconv}
	\mathbb P_\cS\bigg({}& \bigg|\frac{1}{\cK(\frac{1}{\lambda^*+d^*}\log n)}\Big(\log O_n -\frac{d^*}{\lambda^*+d^*}\log n-\frac{\lambda^*}{\lambda^*+d^*}\cK_\alpha\big(r\big(\tfrac{1}{\lambda^*}\log n\big)\big)\Big)\bigg|\leq 2\lambda^*\eps \bigg)\\
	={}&\mathbb P_\cS\Big(O_n \in \big[\exp\big(\ell_n(-2\eps)\big),\exp\big(\ell_n(2\eps)\big)\big]\Big)\\
	\geq {}& \Ps{O_{\tau_n}^\cont\in \Big[\frac{1}{\lambda^*}\ell_n(-\eps),\frac{1}{\lambda^*}\ell_n(\eps)\Big]}-\Ps{ N\Big(\frac{1}{\lambda^*}\ell_n(-\eps)\Big) \leq \exp\big(\ell_n(-2\eps)\big)}\\ 
	&-\Ps{N\Big(\frac{1}{\lambda^*}\ell_n(\eps)\Big) \geq \exp\big(\ell_n(2\eps)\big)}.
	\ea\ee 
	We then note that the first probability on the right-hand side tends to one by~\eqref{eq:Otaunbound}, and the other two probabilities tend to zero by Corollary~\ref{cor:growth}. Indeed, as we can bound
	\be \label{eq:Nbounds}
	|\cB(0,t)|\leq N(t)\leq 2|\cB(0,t)|, 
	\ee 
	we take $r=0$, $u=\eps \lambda^*\cK(\frac{1}{\lambda^*+d^*}\log n)$, and  $s=\frac{1}{\lambda^*}\ell_n(-\eps) $ for the middle probability and $s=\frac{1}{\lambda^*}\ell_n(\eps)$ for the last probability on the right-hand side of~\eqref{eq:Ondiscrconv}. This yields~\eqref{eq:Onconv}.
	
	In a similar manner, 
	\be \ba \label{eq:Itaunbound}
	\mathbb P_\cS\bigg({}&\bigg|\frac{1}{\cK(\frac{ 1}{\lambda^*+d^*}\log n)}\Big(I_{\tau_n}^\cont-\frac{d^*}{\lambda^*(\lambda^*+d^*)}\log n-\frac{1}{\lambda^*+d^*}\cK_\alpha \big(r\big(\tfrac{1}{\lambda^*}\log n\big)\big)\Big)- \frac{\lambda^*+d^*}{2}\bigg|<\eps\bigg)\\ 
	\geq{}& \Ps{\Bigg|\frac{I_s^\cont-\frac{d^*}{\lambda^*(\lambda^*+d^*)}\log n-\frac{1}{\lambda^*+d^*}\cK_\alpha \big(r\big(\tfrac{1}{\lambda^*}\log n\big)\big)}{\cK(\frac{ 1}{\lambda^*+d^*}\log n)}- \frac{\lambda^*+d^*}{2}\Bigg|<\eps\text{ for all }s\in F_{p,n}'}\\ 
	&-\Ps{\tau_n\not \in F_{p,n}'},
	\ea\ee 
	which converges to one as $n$ tends to infinity by applying Proposition~\ref{prop:It} with $t=\frac{1}{\lambda^*}\log n$ and $p$ sufficiently small, and using~\eqref{eq:tauconv}. Then,  as $I_n\overset d= N(I_{\tau_n}^\cont)$ and using a similar rewriting as in~\eqref{eq:bigeventOn} and~\eqref{eq:smalleventOn},
	\be \ba 
	\mathbb P_\cS{}&\Big(I_n \in \big[\exp\big(\ell_n\big(\tfrac{\lambda^*+d^*}{2}-2\eps\big)\big),\exp\big(\ell_n\big(\tfrac{\lambda^*+d^*}{2}+2\eps\big)\big)\big]\Big)\\
	\geq {}& \Ps{I_{\tau_n}^\cont\not\in \Big[\frac{1}{\lambda^*}\ell_n\big(\tfrac{\lambda^*+d^*}{2}-\eps\big),\frac{1}{\lambda^*}\ell_n\big(\tfrac{\lambda^*+d^*}{2}+\eps\big)\Big]}\\ 
	&-\Ps{N\Big(\frac{1}{\lambda^*}\ell_n\big(\tfrac{\lambda^*+d^*}{2}-\eps\big)\Big) \leq \exp\big(\ell_n\big(\tfrac{\lambda^*+d^*}{2}-2\eps\big)\big)}\\ 
	&-\Ps{N\Big(\frac{1}{\lambda^*}\ell_n\big(\tfrac{\lambda^*+d^*}{2}+\eps\big)\Big) \geq \exp\big(\ell_n\big(\tfrac{\lambda^*+d^*}{2}+2\eps\big)\big)}.
	\ea\ee 
	The right-hand side tends to zero by~\eqref{eq:Itaunbound} and combining~\eqref{eq:Nbounds} with Corollary~\ref{cor:growth} (using $r=0$, $u=\eps\cK(\frac{1}{\lambda^*+d^*}\log n)$ and $s=\frac{1}{\lambda^*}\ell_n(\tfrac{\lambda^*+d^*}{2}-\eps)$ for the second and $s=\frac{1}{\lambda^*}\ell_n(\tfrac{\lambda^*+d^*}{2}+\eps) $ for the third probability). This yields~\eqref{eq:Inconv}.
	
	Finally, we let $\deg^{\max}(s):=\max_{v\in \cA^\cont_s}\deg^{(v)}(s)$ for brevity. With the same approach, we again have 
	\be\ba 
	\mathbb P_\cS\Bigg({}&\Bigg|\frac{\varphi_1(\deg^{\max}(\tau_n))-\frac{1}{\lambda^*+d^*}\log n+\frac{1}{\lambda^*+d^*}\cK_\alpha\big(r\big(\tfrac{1}{\lambda^*}\log n\big)\big)}{\cK\big(\tfrac{1}{\lambda^*+d^*}\log n\big)}-\frac{\lambda^*-d^*}{2}\Bigg|<\eps\Bigg)\\ 
	\geq{}&\Ps{\Bigg|\frac{\varphi_1(\deg^{\max}(s))-\frac{1}{\lambda^*+d^*}\log n+\frac{1}{\lambda^*+d^*}\cK_\alpha\big(r\big(\tfrac{1}{\lambda^*}\log n\big)\big)}{\cK\big(\tfrac{1}{\lambda^*+d^*}\log n\big)}-\frac{\lambda^*-d^*}{2}\Bigg|<\eps\text{ for all }s\in F_{p,n}'}\\ 
	&-\Ps{\tau_n\not \in F_{p,n}'}
	\ea\ee 
	which converges to one as $n$ tends to infinity by applying Proposition~\ref{prop:It} with $t=\frac{1}{\lambda^*}\log n$ and  $p$ sufficiently small, and using~\eqref{eq:tauconv}. As 
	\be 
	\{\varphi_1(\max_{v\in \cA_{\tau_n}^\cont}\deg^{(v)}(\tau_n)):n\in\N\}\overset d=\{ \varphi_1(\max_{v\in \cA_n}\deg^{(v)}(n)):n\in\N\},
	\ee 
	by Proposition~\ref{prop:embed}, we thus arrive at~\eqref{eq:maxdegconv}, which concludes the proof.
\end{proof}

We then provide a short proof of Theorem~\ref{thrm:asPA}, as most results follow directly from Theorem~\ref{thrm:conv}. 

\begin{proof}[Proof of Theorem~\ref{thrm:asPA}]
	The results in~\eqref{eq:Inplusmax} directly follow from Theorem~\ref{thrm:conv}. It remains to prove~\eqref{eq:Otasconv}. This follows directly from Proposition~\ref{prop:oldage}, in particular~\eqref{eq:Otas} (and its proof), which states that there exists some (random) time $T<\infty$ such that $O_t^\cont=O'$ for all $t\geq T$ almost surely conditionally on $\cS$, and for some almost surely finite random variable $O'$. After $O'$ amount of time, the branching process has observed $N(O')$ many births or deaths, so that $O_n\to O:=N(O')$ $\mathbb P_\cS$-a.s. 
\end{proof}
We then prove Theorem~\ref{thrm:biggerR}, which uses a similar approach as the proof of Theorem~\ref{thrm:conv}. 

\begin{proof}[Proof of Theorem~\ref{thrm:biggerR}]
	First, by~\eqref{eq:taunconv}, for any $p>0$ the event 
	\be \label{eq:tauconvR}
	\tau_n\in \Big[\frac{1}{\lambda^*}\log n-p\log\big(\tfrac{1}{\lambda^*}\log n\big),\frac{1}{\lambda^*}\log n+p\log\big(\tfrac{1}{\lambda^*}\log n\big)\Big]=:\wt F'_{p,n}
	\ee 
	holds with probability tending to one as $n$ tends to infinity. To prove the convergence of $O_n$, we apply a similar argument to that in~\eqref{eq:Otaunbound} and~\eqref{eq:Ondiscrconv}, to arrive at 
	\be \ba \label{eq:OnRconv}
	\mathbb P_\cS{}&\bigg(\Big|\frac{\log O_n}{\log n}-\frac{R}{\lambda^*+R}\Big|<2K\lambda^*\frac{\log\big(\tfrac{1}{\lambda^*}\log n\big)}{\log n}\bigg)	\\ 
	\geq{}& \mathbb P_{\cS}\bigg(O_{\tau_n}^\cont\in \Big[\frac{R}{\lambda^*(\lambda^*+R)}\log n -K\log\big(\tfrac{1}{\lambda^*}\log n\big),\frac{R}{\lambda^*(\lambda^*+R)}\log n+K\log\big(\tfrac{1}{\lambda^*}\log n\big)\Big]\bigg)\\
	&-\mathbb P_{\cS}\bigg(N\Big(\frac{R}{\lambda^*(\lambda^*+R)}\log n -K\log\big(\tfrac{1}{\lambda^*}\log n\big)\Big)\leq \exp\Big(\frac{R}{\lambda^*+R}\log n -2K\lambda^*\log\big(\tfrac{1}{\lambda^*}\log n\big)\Big)\bigg)\\
	&-\mathbb P_{\cS}\bigg(N\Big(\frac{R}{\lambda^*(\lambda^*+R)}\log n +K\log\big(\tfrac{1}{\lambda^*}\log n\big)\Big)\geq \exp\Big(\frac{R}{\lambda^*+R}\log n +2K\lambda^*\log\big(\tfrac{1}{\lambda^*}\log n\big)\Big)\bigg).
	\ea\ee 
	The last two probabilities on the right-hand side tend to zero by using~\eqref{eq:Nbounds} in combination with Corollary~\ref{cor:growth} (with $r=0$, $u=K\lambda^*\log\big(\tfrac{1}{\lambda^*} \log n)$, and $s=\frac{R}{\lambda^*(\lambda^*+R)}\log n-K\log\big(\tfrac{1}{\lambda^*}\log n\big)$ in the second and $s=\frac{R}{\lambda^*(\lambda^*+R)}\log n+K\log\big(\tfrac{1}{\lambda^*}\log n\big)$ in the third probability). We bound the first probability in the same way as in~\eqref{eq:Otaunbound}, to arrive at
	\be\ba
	\mathbb P_\cS{}&\bigg(O_{\tau_n}^\cont\in \Big[\frac{R}{\lambda^*(\lambda^*+R)}\log n -K\log\big(\tfrac{1}{\lambda^*}\log n\big),\frac{R}{\lambda^*(\lambda^*+R)}\log n+K\log\big(\tfrac{1}{\lambda^*}\log n\big)\Big]\bigg)\\ 
	\geq{}& \mathbb P_\cS\bigg(O_s^\cont\in \Big[\frac{R}{\lambda^*(\lambda^*+R)}\log n-K\log\big(\tfrac{1}{\lambda^*}\log n\big),\frac{R}{\lambda^*(\lambda^*+R)}\log n+K\log\big(\tfrac{1}{\lambda^*}\log n\big)\Big], \forall s\in\wt F'_{p,n}\bigg)\\ 
	&-\Ps{\tau_n\not \in \wt F_{p,n}'}.
	\ea\ee 
	The first probability on the right-hand side tends to one for $K>K_0$ and $p$ sufficiently small by using~\eqref{eq:ROt} in Proposition~\ref{prop:oldage} (with $t=\frac{1}{\lambda^*}\log n$). The second probability tends to zero by~\eqref{eq:tauconvR}. As a result, the left-hand side of~\eqref{eq:OnRconv} tends to one for any $K>K_0$, which yields~\eqref{eq:OnR}.
	
	We then prove a lower bound for $I_n$, by using results for $I^\cont_{\tau_n}$. We fix $K'>0$, recall $w$ from~\eqref{eq:w}, and set  
	\be 
	k_n:=\ceil{\rho_1^{-1}\big(K'\log \big(\tfrac{1}{\lambda^*}\log n\big)\big)}\qquad\text{and}\qquad \wt F^*_{K',n}:=\Big[\frac{1}{\lambda^*}\log n-w(k_n), \frac{1}{\lambda^*}\log n +w(k_n)\Big].
	\ee  
	We recall the constant $K_0$ from  Proposition~\ref{prop:oldage}, fix $K_1>K_0$, and bound 
	\be \ba \label{eq:ItaunR}
	\mathbb P_\cS{}&\Big( I_{\tau_n}^\cont\geq \frac{R}{\lambda^*(\lambda^*+R)}\log n+K_1\log\big(\tfrac{1}{\lambda^*}\log n\big)\Big)\\
	\geq{}&\Ps{I_s\geq \frac{R}{\lambda^*(\lambda^*+R)}\log n+K_1\log\big(\tfrac{1}{\lambda^*}\log n\big)\text{ for all }s\in \wt F^*_{K',n}}-\Ps{\tau_n\not \in \wt F^*_{K',n}}.
	\ea \ee 
	The final probability is at most $\eps$ by choosing the constant $A$ in the definition of $w$ sufficiently large with respect to $\eps$ and using~\eqref{eq:taunconv}. With $t=\frac{1}{\lambda^*}\log n$, we claim that Propositions~\ref{prop:oldage}, \ref{prop:nooldhighdeg}, and~\ref{prop:divdeathhighdeg} yield that the first probability on the right-hand side converges to one. First, Proposition~\ref{prop:oldage} implies that there exists $p>0$ such that oldest individual alive at time $s$ has a birth-time at least $\frac{R}{\lambda^*(\lambda^*+R)}\log n-K_1\log\big(\tfrac{1}{\lambda^*}\log n\big)$, for all $s$ in the interval
	\be 
	\Big[\frac{1}{\lambda^*}\log n-p\log\big(\tfrac{1}{\lambda^*}\log n\big),\frac{1}{\lambda^*}\log n+p\log\big(\tfrac{1}{\lambda^*}\log n\big)\Big]. 
	\ee
	We then note that $w(k_n)$ is either equal to a constant $A>0$ (when $\varphi_2$ converges) or is equal to $\varphi_2(k_n)$ (when $\varphi_2$ tends to infinity), which is negligible compared to $\rho_1(k_n)$, since $\varphi_2=o(\varphi_1)$, as $b$ diverges, and $\varphi_1=\cO(\rho_1)$ since $\underline d>0$. As a result, in either case and for any $A,K',p>0$ it follows that 
	\be 
	\wt F^*_{K',n}\subset \Big[\frac{1}{\lambda^*}\log n-p\log\big(\tfrac{1}{\lambda^*}\log n\big),\frac{1}{\lambda^*}\log n+p\log\big(\tfrac{1}{\lambda^*}\log n\big)\Big]
	\ee 
	for all $n$ large (i.e.\ larger than some fixed $N=N(A,K',p)\in\N$). Hence, Proposition~\ref{prop:oldage} implies that the oldest individual alive at time $s$ has a birth-time at least $\frac{R}{\lambda^*(\lambda^*+R)}\log n-K_1\log\big(\tfrac{1}{\lambda^*}\log n\big)$, for all $s\in \wt F^*_{K',n}$ with high probability. 
	
	Then, Proposition~\ref{prop:nooldhighdeg} tells us that, by choosing $K'$ sufficiently large with respect to $K_1$, there does not exist $s\in\wt  F^*_{K',n}$ and an individual born before time $\frac{R}{\lambda^*(\lambda^*+R)}\log n+K_1\log\big(\tfrac{1}{\lambda^*}\log n\big)$ that is alive at time $s$ and has at least $k_n$ many children. Finally, Proposition~\ref{prop:divdeathhighdeg} yields that at any time $s\in\wt F^*_{K',n}$ there exists an individual born after time $\frac{R}{\lambda^*(\lambda^*+R)}\log n+K_1\log\big(\tfrac{1}{\lambda^*}\log n\big)$ that is alive at time $s$ and has at least $k_n$ many children. Combined, we conclude that the first term on the right-hand side of~\eqref{eq:ItaunR} indeed converges to one with $n$. For any $K_1>K_0$, we thus arrive at
	\be \label{eq:liminfInbound}
	\liminf_{n\to\infty}\mathbb P_\cS\Big( I_{\tau_n}^\cont\geq \frac{R}{\lambda^*(\lambda^*+R)}\log n+K_1\log\big(\tfrac{1}{\lambda^*}\log n\big)\Big)\geq 1-\eps. 
	\ee 
	As $\eps$ is arbitrary, the left-hand side is in fact equal to one. Finally, we bound 
	\be \ba \label{eq:InRconv}
	\mathbb P_\cS{}&\bigg(\log I_n\geq \frac{R}{\lambda^*+R}\log n+2K_1\lambda^*\log\big(\tfrac{1}{\lambda^*}\log n\big)\bigg)	\\ 
	\geq{}& \mathbb P_{\cS}\bigg(I_{\tau_n}^\cont\geq \frac{R}{\lambda^*(\lambda^*+R)}\log n+K_1\log\big(\tfrac{1}{\lambda^*}\log n\big)\bigg)\\
	&-\mathbb P_{\cS}\bigg(N\Big(\frac{R}{\lambda^*(\lambda^*+R)}\log n +K_1\log\big(\tfrac{1}{\lambda^*}\log n\big)\Big)\geq \exp\Big(\frac{R}{\lambda^*+R}\log n +2K_1\lambda^*\log\big(\tfrac{1}{\lambda^*}\log n\big)\Big)\bigg).
	\ea\ee 
	The first probability on the right-hand side converges to one for any $K_1>K_0$ by~\eqref{eq:liminfInbound}. The second probability tends to zero by using~\eqref{eq:Nbounds} and~\eqref{eq:ubB} in Corollary~\ref{cor:growth} (with $r=0$, $s=\frac{R}{\lambda^*(\lambda^*+R)}\log n +K_1\log\big(\tfrac{1}{\lambda^*}\log n\big)$, and $u=K_1\lambda^*\log\big(\tfrac{1}{\lambda^*}\log n\big)$). We thus arrive at 
	\be 
	\lim_{n\to\infty}\mathbb P_\cS\bigg(\log I_n\geq \frac{R}{\lambda^*+R}\log n+2K_1\lambda^*\log\big(\tfrac{1}{\lambda^*}\log n\big)\bigg)=1,
	\ee 
	for any $K_1>K_0$. Combined with the fact that the left-hand side of~\eqref{eq:OnRconv} converges to one for any $K>K_0$, we can thus choose $K_1>K>K_0$ to arrive at the desired result, since then
	\be 
	\lim_{n\to\infty}\Ps{\log\Big(\frac{I_n}{O_n}\Big)\geq 2(K_1-K)\lambda^*\log\big(\tfrac{1}{\lambda^*}\log n\big)}=1, 
	\ee 
	and thus $I_n/O_n\toinps \infty$, as desired, which concludes the proof.
\end{proof}  

\section{Converging birth rates and constant death rates: Theorem~\ref{thrm:const}}\label{sec:const}

In this section we discuss some of the  adaptations to the analysis in Sections~\ref{sec:old} through~\ref{sec:mainproofs} required to prove Theorem~\ref{thrm:const}.

\textbf{Assumptions. } We observe that Assumptions~\eqref{ass:A1} and~\eqref{ass:A2} are  both satisfied when $b$ and $d$ converge. Assumptions~\hyperref[ass:K]{$\cK$} and~\hyperref[ass:Kalpha]{$\cK_\alpha$} and~\eqref{ass:varphi2} are not required.  

\textbf{Preliminary results. } As $d$ converges to one, it follows from Lemma~\ref{lemma:stochdomexp} that the  remaining lifetime $\sum_{i=k}^{k+D_k}E_i$ is asymptotically distributed as a rate-one exponential random variable (for large $k$). By the upper bound in Lemma~\ref{lemma:survdeg}, with $x=1- \eps$,  $t'\geq t\geq 0$ and $k$ large, we thus have
\be\label{eq:probconstub}
\P{D\geq k, S_k\leq t, S_{D+1}>t'}\leq \e^{-(1-\eps)t'}\P{D\geq k}\E{\ind_{\{S_k\leq t\}}\e^{(1-\eps)S_k}}.
\ee 
The lower bound in Lemma~\ref{lemma:survdeg} yields a lower bound for the probability when changing $1-\eps$ to $1+\eps$, for $k$ large.  Then, as $b$ converges to $c$ and $d$ converges to $1$, 
\be 
\frac{\E{S_k}}{k}=\frac{\varphi_1(k)}{k}=\frac{1+o(1)}{c+1}. 
\ee 
As $S_k$ is a sum of independent exponential random variables whose rates converge to $c+1$, it is clear (by comparing with the case of constant rates equal to $c+1$) that $S_k$ satisfies a large deviation principle   rate function $I$  defined as 
\be 
I(\alpha):=\alpha(c+1)-1-\log(\alpha(c+1))\qquad\text{for }\alpha>0. 
\ee 
That is, for $\alpha<\frac{1}{c+1}$ so that $\alpha k<\E{S_k}$ for all $k$ large (resp.\ $\alpha>\frac{1}{c+1}$), we have
\be 
\lim_{k\to\infty}\frac1k \log\P{S_k\leq \alpha k}=-I(\alpha)\qquad\bigg(\text{resp.\ } 
\lim_{k\to\infty}\frac1k \log\P{S_k\geq \alpha k}=-I(\alpha)\bigg).
\ee 
Similar to the moderate deviation principle with exponential tilting in Proposition~\ref{prop:mdptilt}, one can prove a large deviation principle with exponential tilting in this setting, in the sense that
\begin{align}\label{eq:ldptilt}
	\lim_{k\to\infty}\frac1k\log\E{\ind_{\{S_k\leq \alpha k\}}\e^{\theta S_k}}=-I(\alpha)+\alpha \theta \qquad\text{for }\alpha<\frac{1}{c+1},\theta\in\R,
	\intertext{and}
	\lim_{k\to\infty}\frac1k\log\E{\ind_{\{S_k\geq \alpha k\}}\e^{\theta S_k}}=-I(\alpha)+\alpha \theta\qquad \text{for }\alpha>\frac{1}{c+1},\theta\in\R. 
\end{align}  

\textbf{The oldest alive individual. } As in the `rich are old' and `rich die young' regimes,  $L$ is approximately distributed as an  exponential random variable with rate $\min\{1,R\}$, where we recall that $R:=\inf_{i\in\N_0}(b(i)+d(i))$. More precisely, 
\be 
\lim_{t\to\infty}\frac1t\log\P{L>t}=-\min\{1,R\}=:M. 
\ee 
One can then show, following a proof similar to that of~\eqref{eq:ROt}   and~\eqref{eq:Otconv} (with $d^*=1$) in Proposition~\ref{prop:oldage} that for any $\eps>0$ there exists $p>0$ such that
\be 
\lim_{t\to\infty}\Ps{\forall s\in [(1-p)t,(1+p)t]: \Big| \frac{O_s^\cont}{t}-\frac{M}{\lambda^*+M}\Big|<\eps }=1.
\ee 
\textbf{The individual with the largest offspring. } Define the constant
\be 
u^*=u^*(\lambda^*):=1-\frac{\lambda^*}{2(\lambda^*+1)\log2}.
\ee 
One can then combine the upper bound in~\eqref{eq:probconstub} (and the equivalent lower bound where $1-\eps$ is replaced with $1+\eps$) with the tilted large deviation principle in~\eqref{eq:ldptilt} in a proof similar to that of Proposition~\ref{prop:It} to show that for any $\eps>0$ there exists $p>0$ such that
\be 
\lim_{t\to\infty}\Ps{\forall s\in [(1-p)t,(1+p)t]: \Big|\frac{I_s^\cont}{t}-u^*\Big|<\eps\text{ and }\Big|\frac{\max_{v\in \cA_s^\cont}d^{(v)}(s)}{t}-\frac{\lambda^*}{\log 2}\Big|<\eps }=1.
\ee 
Indeed, we partition $[\frac{M}{\lambda^*+M},u^*-\eps)\cup(u^*+\eps, 1]$ in small intervals $[x_i,x_{i+1})$ of length $\zeta>0$, for $i\in\mathbb I$. Similar to the proof of Proposition~\ref{prop:It}, one can show that for any $\eps>0$ there exists $p>0$ such that 
\be 
\lim_{t\to\infty}\Ps{I_s\in [u^*-\eps,u^*+\eps]\text{ for all }s\in  [(1-p)t,(1+p)t]}=1.
\ee 
To this end, with $\xi>0$ small, one is required to show that
\be 
\lim_{t\to\infty}\mathbb P\big( D^{\max}_{u^*-\eps,u^*+\eps}(s)\geq (H(u^*)-\xi)t,\text{ for all }s\in[(1-p)t,(1+p)t]\big)=1,
\ee 	
and
\be 
\lim_{t\to\infty}\mathbb P\big( D^{\max}_{x_i,x_{i+1}}(s)\leq (H(x_{i+1})+\xi)t, \text{ for all }i\in\mathbb I\text{ and for all }s\in[(1-p)t,(1+p)t]\big)=1,
\ee 
where $H\colon [(\lambda^*+1)^{-1},1)\to [-1-(\lambda^*+1)^{-1},\infty)$ is defined as 
\begin{align} 
	H(x)&:=(\lambda^*+1)(1-x)\Big(1+\frac{1}{\lambda^*+1}+w^{-1}\Big(-\frac{\lambda^*x}{(\lambda^*+1)(1-x)}\Big)\Big),
	\intertext{and $w\colon [-(\lambda^*+1)^{-1},\infty)\to (-\infty, -(\lambda^*+1)^{-1}]$ is defined as} w(x)&:=x-\Big(1+\frac{1}{\lambda^*+1}+x\Big)\log\Big(1+\frac{1}{\lambda^*+1}+x\Big).
\end{align}
It is readily verified that $H$ has a unique maximum at $u^*$ of value $H(u^*)=\lambda^*/\log2$. 
Finally, using a similar proof as that of Theorem~\ref{thrm:conv} yields the asymptotic behaviour of $O_n$ and $I_n$ and proves Theorem~\ref{thrm:const}.

\textbf{Constant birth and death rates (Remark~\ref{rem:const}). } In the case that $b\equiv c>1$ and $d\equiv 1$, by~\cite[Proposition $1.1$]{Dei10} (see also Proposition~\ref{prop:rhoexpl}), 
\be 
\widehat \mu(\lambda)=\sum_{k=0}^\infty \prod_{i=0}^k \frac{b(i)}{\lambda +b(i)+d(i)}=\sum_{k=0}^\infty \Big(\frac{c}{\lambda +c+1}\Big)^{k+1}=\frac{c}{\lambda+1}. 
\ee 
As a result, Assumption~\eqref{ass:C1} is satisfied with $\lambda^*=c-1>0$ and $\underline \lambda =0$, so that the general result in Theorem~\ref{thrm:const} follows for this case with $\lambda^*=c-1$.

\section{Discussion and open problems}\label{sec:disc}

In this article we investigated an evolving tree process where, at each step, either a new alive vertex is introduced that connects to an already present alive vertex, or an alive vertex is selected and killed. Via an embedding in a CMJ branching process, we analysed this process and studied lack of persistence of the maximum degree: under what conditions is the label of the vertex with the largest in-degree much larger than the label of the oldest alive vertex (i.e.\ the alive vertex with the smallest label)?

We identified the `rich are old' and `rich die young' regimes', in which lack of persistence occurs for different reasons. In the former regime, where behaviour is somewhat similar to that of preferential attachment without death, lack of persistence occurs due to the fact that one can find individuals (in the branching process) that are slightly older than the oldest alive individuals and both manage to survive and to produce a large number of children, faster than typical individuals. In the latter regime, the oldest individuals manage to survive by only producing a bounded number of children, whereas individuals that produce a large number of children can only survive for a much shorter amount of time. This regime is entirely novel and cannot be observed in preferential attachment without death, and has not been studied nor identified before.

\emph{Limitations, conjectures, and open problems.}

Though we present a variety of results that provide a good understanding of the phase transition between the `rich are old' and the `rich die young' regimes, there are still cases we are currently  not able to deal with, and limitations to the set-up we considered, which we address here.

\textbf{Behaviour of $I_n$ in the `rich die young' regime. } The results in the `rich die young' regime, in particular that there is always lack of persistence of the maximum degree, hold for quite a general choice of birth and death rates. Unlike in the `rich are old' regime, however, we do not provide a scaling limit for $I_n$. When the death rates converge to a constant $d^*\geq R$, a similar result as in Theorem~\ref{thrm:conv} holds (see Proposition~\ref{prop:It}). However, when the death rates diverge, making it much harder for individuals to survive for a long time and also obtain a large offspring, it is not entirely clear to us what the right scaling of $I_n$ should be. We leave this as an open problem.

\begin{problem}
	Suppose that the sequence $d$ tends to infinity.  Theorem~\ref{thrm:biggerR} shows that $O_n=n^{R/(\lambda^*+R)+o(1)}$ conditionally on survival. What is the order of $I_n$? We expect that $I_n=o(n)$ but $I_n\gg n^{1-\eps}$ for any $\eps>0$. 
\end{problem}

\textbf{Bounded non-converging death rates. } The results presented in the `rich are old' regime do not consider the case when the death rates are bounded but do not converge. This is mainly due to the difficulty to determine the lifetime distribution of individuals in complete generality. As an example, let us consider the case of converging birth rates and alternating convergent death rates. That is, take three sequences $(c_i)_{i\in\N}$, $(d_{1,i})_{i\in\N}$ and $(d_{2,i})_{i\in\N}$ such that $\lim_{i\to\infty}c_i=c>0$ and $\lim_{i\to\infty}d_{j,i}=d_j$ for $j\in\{1,2\}$, with $d_1\neq d_2\geq 0$. We set
\be 
d=(d_{1,1},d_{2,1},d_{1,2},d_{2,2},\ldots), \qquad b(i)=c_i.
\ee 
We then claim (without proof) that the lifetime distribution satisfies 
\be 
\lim_{t\to\infty} \frac1t\log \P{L>t}=-\min\bigg\{c+\frac{d_1+d_2}{2}-\sqrt{c^2+\Big(\frac{d_1-d_2}{2}\Big)^2},R\bigg\}=:-\theta. 
\ee 
This already non-trivial result can be obtained due to the somewhat simple choice of the birth and death rates. More general birth or death rates (e.g.\ $b$ that diverge to infinity) make it much harder to determine the limiting value of $\theta$ that characterises the exponential decay of the lifetime distribution. 

We expect that in this more general setting, a distinction between the `rich are old' and the `rich die young' regimes can be made depending on whether $\theta<R$ or $\theta=R$.  We leave this as an open problem. 

\begin{problem}
	Suppose the sequences $b$ and $d$ are such that Assumptions~\eqref{ass:A1} and~\eqref{ass:C1} are satisfied. Assume that $\theta:=\lim_{t\to\infty}-\frac1t\log\P{L>t}$ exists and recall $R$ from~\eqref{eq:R}. Can we define the `rich are old' and the `rich die young' regimes via $\theta<R$ and $\theta=R$, respectively?
\end{problem}

We recall that  $\theta=\min\{d^*,R\}$ in the case that $d$ converges to a constant $d^*\geq 0$, which indeed agrees with the open problem. 

\textbf{Omitting assumptions. } There are various assumptions used throughout the analysis to prove the results presented here. A number of these assumptions are not necessary to determine whether persistence of lack thereof occurs, we believe. Assumptions~\hyperref[ass:K]{$\cK$} and~\hyperref[ass:Kalpha]{$\cK_\alpha$} are necessary to obtain higher-order behaviour of the quantities $O_n$ and $I_n$ (and their continuous counterparts $O_t^\cont$ and $I_t^\cont$), which we use to determine whether lack of persistence occurs. However, we believe that, like the recent developments presented by Iyer in~\cite{Iyer24}, it should be possible to prove (lack of) persistence occurs even without knowing the order of $O_n$ and $I_n$. This would deem Assumptions~\hyperref[ass:K]{$\cK$} and~\hyperref[ass:Kalpha]{$\cK_\alpha$} unnecessary, as presented in the following conjecture.

\begin{conjecture}\label{conj:omitass}
	Consider the PAVD model, as in Definition~\ref{def:pavd}.  Suppose that $b$ and $d$ are such that Assumptions~\eqref{ass:A1} and~\eqref{ass:C1}  are satisfied. Recall $R$ from~\eqref{eq:R} and suppose that $d$ satisfies that $\limsup_{i\to\infty}d(i)<R$ and that $b$ diverges. Then, persistence holds as in the sense of~\eqref{eq:pers} when Assumption~\eqref{ass:varphi2} is not satisfied, whereas lack of persistence holds as in the sense of~\eqref{eq:nopers} when Assumption~\eqref{ass:varphi2} is satisfied.
\end{conjecture}

One could possibly go even further and omit Assumption~\eqref{ass:C1}, which guarantees the existence of a \emph{Malthusian parameter}, as well.

\begin{problem}
	Does the statement in Conjecture~\ref{conj:omitass} still hold without Assumption~\eqref{ass:C1}?
\end{problem} 

\textbf{Non-tree graphs. } The analysis carried out here considers trees only. It would be interesting to also study evolving graphs with vertex death that are not trees, e.g.\ when every new vertex connects to $m>1$ many alive vertices. The restriction to trees is due to the continuous-time embedding into a CMJ branching process that is central in the analysis, and which naturally restricts us to the setting of trees. 

In the case of \emph{affine} preferential attachment graphs without vertex death, a procedure known as \emph{collapsing branching processes} is used to extend the branching process analysis to non-tree graphs as well~\cite{GarHof18}. However, this is limited to the affine case only. When including vertex death, this could only have a potential when the death rates are affine as well. Even then, this introduces difficulties due to the fact that now the lifetime of distinct vertices could have dependencies, which further complicates the analysis.

\textbf{Other modelling choices. } In the (embedding of the) PAVD model the size of the offspring of an individual equals the number of children said individual has produced. Similarly, the rate at which an individual produces its next child or dies depends on the total number of children produced. Instead, one could count only \emph{alive} children towards the offspring, or let the birth and death rates depend on the number of \emph{alive} children. Either way of defining the model can be motivated from applications. However, the analysis for these alternative choices present more challenges, since now either an individual needs to produce $n$ children out of which $n-k$ should die to have an offspring of size $k$, or the rates depend not only on the parent individual but also on the lifetimes of its children. Deijfen concludes from simulations presented in~\cite{Dei10} that the asymptotic behaviour of the degree distribution should not change significantly, and we expect this to be the case for other properties such as persistence as well. Still, it would be interesting to further investigate these models as well.

\section*{Acknowledgements} 

The authors wish to thank Mia Deijfen for several interesting and stimulating discussions. Bas Lodewijks has received funding from the European Union’s Horizon 2022 research and innovation programme under the Marie Sk\l{}odowska-Curie grant agreement no.~$101108569$, ``DynaNet".
Part of this work was done while Markus Heydenreich was in residence at the Mathematical Sciences Research Institute in Berkeley supported by NSF Grant No. DMS-1928930. 

\bibliographystyle{abbrv}
\bibliography{paverbib}

\appendix 

\section{Proofs of preliminary results}\label{app:proofs}

In this appendix we present the proofs of the preliminary results of Section~\ref{sec:prelim}. 

\begin{proof}[Proof of Lemma~\ref{lemma:Dlifedistr}]
	The distribution of $D$ follows directly from its definition in~\eqref{eq:D} and Definition~\ref{def:inhomgeom} of an inhomogeneous geometric random variable. What remains is to show that $D$ is almost surely finite if and only if Assumption~\eqref{ass:A2} is satisfied. To this end, we write 
	\be \label{eq:telescopic}
	\sum_{k=0}^\infty \P{D= k}=\sum_{k=0}^\infty  \frac{d(k)}{b(k)+d(k)}\prod_{i=0}^{k-1}\frac{b(i)}{b(i)+d(i)}= \sum_{k=0}^\infty \Big( \prod_{i=0}^{k-1}\frac{b(i)}{b(i)+d(i)}-\prod_{i=0}^k\frac{b(i)}{b(i)+d(i)}\Big). 
	\ee 
	We thus have a telescopic sum, where the first term in the brackets equals one when $k=0$. As a result, we obtain that $D$ is finite almost surely if and only if 
	\be
	\prod_{i=0}^\infty \frac{b(i)}{b(i)+d(i)}=0 \quad \Leftrightarrow \quad \sum_{i=0}^\infty \frac{d(i)}{b(i)+d(i)}=\infty. 
	\ee 
	Here, the final claim follows directly from the equivalence of the convergence of infinite sums and products. That is, for $(x_i)_{i\in\N_0}\subset[0,1)$, 
	\be 
	\prod_{i=0}^\infty (1-x_i)>0 \quad\Leftrightarrow\quad\sum_{i=0}^\infty x_i<\infty , 
	\ee  
	which concludes the proof.
\end{proof}

\begin{proof}[Proof of Lemma~\ref{lemma:Dtail}]
	First, we prove the upper bound. Using the inequality $\log(1-x)\leq -x-x^2/2$ for $x\in [0,1)$ and the tail distribution of $D$, as in~\eqref{eq:tailD}, we obtain 
	\be 
	\P{D\geq k}=\prod_{i=0}^{k-1}\frac{b(i)}{b(i)+d(i)}\leq \exp\bigg(-\sum_{i=0}^{k-1}\frac{d(i)}{b(i)+d(i)}-\frac12 \sum_{i=0}^{k-1}\Big(\frac{d(i)}{b(i)+d(i)}\Big)^2\bigg)=\e^{-\rho_1(k)-\frac12 \rho_2(k)}. 
	\ee 
	For the lower bounds, we use that $\log(1-x)\geq -x-x^2/2-x^3/(1-x)$ for $x\in[0,1)$ to obtain 
	\be 
	\P{D\geq k}\geq \exp\bigg(-\sum_{i=0}^{k-1}\frac{d(i)}{b(i)+d(i)}-\frac12 \sum_{i=0}^{k-1}\Big(\frac{d(i)}{b(i)+d(i)}\Big)^2-\sum_{i=0}^{k-1}\Big(\frac{d(i)}{b(i)+d(i)}\Big)^2\frac{d(i)}{b(i)}\bigg).
	\ee 
	In the case that $\rho_2$ diverges and $d=o(b)$, it follows that the third sum is negligible compared to the second sum (which equals $\rho_2(k)$), so that we obtain the lower bound
	\be 
	\P{D\geq k}\geq \e^{-\rho_1(k)-(\frac12+o(1))\rho_2(k)}, 
	\ee 
	as desired. In the case that $\rho_2$ converges, it follows that the fraction $d(i)/(b(i)+d(i))$ tends to zero with $i$. As a result, we have $d=o(b)$, so that the third sum converges as well. We again obtain the desired lower bound, which concludes the proof. 
\end{proof}

\begin{proof}[Proof of Lemma~\ref{lemma:stochdomexp}]
	We start by proving~\eqref{eq:Lexp}. First, note that $D_0\overset \dd= D$, so that the result $L\overset \dd=\mathrm{Exp}(d^*)$ when $d\equiv d^*$ directly follows from the case $I=k=0$ in~\eqref{eq:Lexp}. We can interpret the sum in~\eqref{eq:Lexp} as follows. We have exponential `birth clocks' $(E^b_i)_{i\geq k}$ that ring at rates $(b(i))_{i\geq k}$, and exponential `death clocks' $(E^d_i)_{i\geq k}$ that ring at rate $d^*$. All clocks are independent. We first let $E^b_k$ and $E^d_k$ run. If $E^b_k$ rings before $E^d_k$, we let $E^b_{k+1}$ and $E^d_{k+1}$ run. This process continues until some $j\geq k$, where $j$ is the first index such that $E^d_j$ rings before $E^b_j$. This process is equivalent to waiting for the clocks of the sum in~\eqref{eq:Lexp} to ring. Indeed, $E_i\overset \dd=\min\{E^b_i,E^d_i\}$ and, independently of $E_i$, 
	\be 
	k+D_k\overset\dd=j:=\inf\{\ell\geq k: E^d_\ell<E^b_\ell\}.
	\ee 
	The sum in~\eqref{eq:Lexp} then is the total time that has passed, that is, $E^b_k+\cdots+ E^b_{j-1}+E^d_j$. 
	
	Since all death clocks have the same rate $d^*$, by the memoryless property of the exponential distribution, the above process is equivalent to running only one death clock $E^d_k$, and running the birth clocks $(E^b_i)_{i\geq k}$ until $E^d_k$ rings. That is, resampling the death clocks $(E^d_i)_{i\geq k}$ every time $E^b_i$ rings first is equivalent to letting $E^d_k$ continue to run and only resampling the clocks $(E^b_i)_{i\geq k}$. It then directly follows that the total time that has passed equals $E^d_k\sim \text{Exp}(d^*)$, from which~\eqref{eq:Lexp} follows. Assumption~\eqref{ass:A1} is required to ensure that $v$ does not produce an infinite number of children before $E^d_k$ amount of time has passed almost surely. That is, it ensures that $\sum_{i=k}^\infty E_i^b>E^d_k$ almost surely. 
	
	We then prove~\eqref{eq:remainstochdom}. Let $\wt D_k$ be a inhomogeneous geometric random variable characterised by the sequences $(b(i))_{i\geq k}$ and $\lambda$, that is, 
	\be 
	\P{\wt D_k\geq \ell}=\prod_{i=k}^{k+\ell-1}\frac{b(i)}{b(i)+\lambda} \qquad\text{for } \ell\in \N_0, 
	\ee 
	and let $\wt E_i\sim \text{Exp}(b(i)+\lambda)$, independent of $\wt D_k$. Then, as $d(i)\geq \lambda$ for all $i\geq k$, it follows that $D_k\preceq \wt D_k$ and $E_i\preceq \wt E_i$ for any $i\geq k$. The desired result thus follows by using~\eqref{eq:Lexp}. The converse result follows mutatis mutandis by reversing inequalities.
	
	Finally, we prove~\eqref{eq:infremlife}. Since Assumption~\eqref{ass:A2} is not satisfied, it follows from Lemma~\ref{lemma:Dlifedistr} that $\P{D=\infty}>0$. An analogous proof yields that $\P{D_k=\infty}>0$ for any $k\in\N_0$. Then, the case distinction in Table~\ref{table:deglife} yields that $S_\infty=\infty$ almost surely when Assumption~\eqref{ass:A1} is satisfied but Assumption~\eqref{ass:A2} is not. Equivalently, $\sum_{i=k}^\infty E_i=S_\infty-S_k=\infty$ almost surely. Hence, by the independence of the exponential random variables $E_i$ and $D_k$, 
	\be 
	\P{\sum_{i=k}^{k+D_k} E_i=\infty}\geq \P{\sum_{i=k}^\infty E_i=\infty}\P{D_k=\infty}=\P{D_k=\infty}>0\qquad \text{for any }k\in\N_0,
	\ee 
	which concludes the proof.
\end{proof}

\begin{proof}[Proof of Lemma~\ref{lemma:remlifediv}]
	Since $d$ is increasing, we observe that $d(i)\geq d(k)$ for all $i\geq k$. Hence, using Lemma~\ref{lemma:stochdomexp}, for any $t\geq 0$, 
	\be\label{eq:remlifeub} 
	\P{d(k)\sum_{i=k}^{k+D_k}E_i\geq t}\leq \e^{-t},
	\ee 
	so that, in particular, we obtain the desired upper bound for $t=t_k$. For a lower bound, we fix $\eps>0$ and set $\ell_k :=\lfloor (1+\eps)t_kb(k)/d(k)\rfloor$, to bound
	\be 
	\P{d(k)\sum_{i=k}^{k+D_k}E_i\geq t_k}\geq \mathbb P\bigg(d(k) \sum_{i=k}^{k+\ell_k} E_i\geq t_k\bigg)\P{D_k\geq \ell_k}.
	\ee 
	We show that the first term on the right-hand side tends to one for any $\eps >0$ and that the second term is at least $\e^{-(1+2\eps)t_k}$. As $\eps$ is arbitrary, the result then follows. Let us start with the first term. First, since $d$ is increasing, the rates of each of the exponential random variables $E_i$ in the sum are at most $\overline b_k+d(k+\ell_k)$, where $\overline b_k:=\sup\big\{b(i):k\leq i\leq k+\ell_k\big\}$. Hence, with $(E'_i)_{i\geq k}$ a sequence of i.i.d.\ rate-one exponential random variables, 
	\be 
	\sum_{i=k}^{k+\ell_k}E_i\succeq \frac{1}{\overline b_k+d(k+\ell_k)}\sum_{i=k}^{k+\ell_k}E'_i.
	\ee 
	By assuming that $t_k=o(kd(k)/b(k))$, it follows that $\ell_k=o(k)$, so that the index $i$ in the suprema in the definition of $\overline b_k$ ranges from $k$ to $k+o(k)$, and $d(k+\ell_k)=d(k+o(k))$. As a result, since $b$ and $d$ are either regularly varying or slowly varying, we obtain from~\cite[Theorems $1.2.1$ and $1.5.2$]{BinGolTeu89} that $\overline b_k=(1+o(1))b(k)$ and $d(k+\ell_k)=(1+o(1))d(k)$. For any $\xi>0$ and $k$ sufficiently large, we thus arrive at 
	\be \label{eq:boundbar}
	\mathbb P\bigg(d(k) \sum_{i=k}^{k+\ell_k} E_i\geq t_k\bigg)\geq \mathbb P\bigg(\frac{d(k)}{(1+\xi)(b(k)+d(k))} \sum_{i=k}^{k+\ell_k} E_i'\geq t_k\bigg).
	\ee 
	The sum on the right-hand side contains $\ell_k+1$ many i.i.d.\ random variables with mean one. Note that $\ell_k$ diverges, since $d=o(b)$ and $\liminf_{k\to\infty}t_k>0$. Hence, the definition of $\ell_k$ and the fact that $d=o(b)$,  combined with the strong law of large numbers, implies that
	\be 
	\frac{d(k)}{(1+\xi)(b(k)+d(k))t_k} \sum_{i=k}^{k+\ell_k} E_i'\toas \frac{1+\eps}{1+\xi}. 
	\ee 
	By taking $\xi<\eps$, the right-hand side of~\eqref{eq:boundbar} thus converges to one with $k$. 
	
	It remains to prove that 
	\be \label{eq:Dklim}
	\P{D_k\geq \ell_k}\geq \e^{-(1+2\eps)t_k}. 
	\ee
	By using the  distribution of $D_k$ in~\eqref{eq:Dk} together with the inequality $\log(1-x)\geq -x-x^2/2-x^3/(1-x)$ for $x\in (0,1)$,  we obtain  the lower bound
	\be \ba 
	\P{D_k\geq \ell_k}&=\prod_{i=k}^{k+\ell_k-1} \frac{b(i)}{b(i)+d(i)}\\ 
	&\geq \exp\bigg(- \sum_{i=k}^{k+\ell_k-1}\!\!\! \frac{d(i)}{b(i)+d(i)}-\frac12  \sum_{i=k}^{k+\ell_k-1}\!\!\! \Big(\frac{d(i)}{b(i)+d(i)}\Big)^2-\sum_{i=k}^{k+\ell_k-1}\!\!\! \Big(\frac{d(i)}{b(i)+d(i)}\Big)^2\frac{d(i)}{b(i)}\bigg).
	\ea \ee 
	Equivalent to $\overline b_k$, we also define $\underline b_k:=\inf\big\{b(i):k\leq i\leq k+\ell_k\big\}$. Also using that $d$ is increasing,  we obtain the lower bound
	\be 
	\exp\bigg(-  \ell_k  \frac{d(k+\ell_k)}{\underline b_k}-\frac12  \ell_k \Big(\frac{d(k+\ell_k)}{\underline b_k}\Big)^2-\ell_k \Big(\frac{d(k+\ell_k)}{\underline b_k}\Big)^3 \bigg).
	\ee 
	It again follows from~\cite[Theorems $1.2.1$ and $1.5.2$]{BinGolTeu89} that $\underline b_k=(1+o(1))b(k)$. Recalling that $d(k+\ell_k)=(1+o(1))d(k)$ and $\liminf_{k\to\infty}t_k>0$, we thus have
	\be 
	\ell_k\frac{d(k+\ell_k)}{\overline b_k}=(1+\eps+o(1))t_k, \qquad \ell_k\Big(\frac{d(k+\ell_k)}{\overline b_k}\Big)^2=o(t_k), \qquad\text{and}\qquad \ell_k\Big(\frac{d(k+\ell_k)}{\overline b_k}\Big)^3=o(t_k), 
	\ee 
	from which we arrive at the final lower bound
	\be
	\P{D_k\geq \ell_k}\geq \exp\big(-(1+\eps+o(1))t_k\big)\geq \e^{-(1+2\eps)t_k},  
	\ee 
	when $k$ is sufficiently large, which concludes the proof. 
\end{proof}

\begin{remark}[Linear birth rates]\label{rem:linb}
	We can weaken the assumption that $t_k=o(kd(k)/b(k))$ somewhat in the case when $b(k)=\cO(k)$. Namely, we assume that $t_k=\cO(d(k))$. This weaker assumption is required in the proof of Proposition~\ref{prop:divdeathhighdeg}. If we follow the proof, we then can no longer use that $\ell_k=o(k)$, but can only conclude that $\ell_k=\cO(k)$. As a result, we cannot use that $\underline b_k=(1+o(1))\overline b_k=(1+o(1))b(k)$ and $ d(k+\ell_k)=(1+o(1))d(k)$, but rather $\underline b_k, \overline b_k=\Theta(b(k))$ and $d(k+\ell_k)=\Theta(d(k))$. By altering the definition of $\ell_k$ to $\ell_k:=\floor{Ct_k k/d(k)}$ for some sufficiently large constant $C>0$, we can still obtain the lower bound
	\be 
	\P{d(k)\sum_{i=k}^{k+D_k}E_i\geq t_k}\geq \e^{-\wt C t_k}, 
	\ee 
	where $\wt C>0$ is a sufficiently large constant. Though this lower bound does not match with the upper bound provided in the proof, it is sufficient for our purposes when using this result in the proof of Proposition~\ref{prop:divdeathhighdeg}.\ensymboldremark
\end{remark} 

\begin{proof}[Proof of Lemma~\ref{lemma:survdeg}]
	On the event $\{D\geq k\}$ we can write $D=k+D_k$, where $D_k$ is defined in~\eqref{eq:Dk}. Using that the exponential random variables $(E_i)_{i\in\N_0}$  are independent of $D$, we obtain
	\be \ba \label{eq:probbound}
	\P{D\geq k, S_k\leq t, S_{D+1}>t' }&=\P{S_k\leq t,S_{k+1+D_k}>t'\,|\, D\geq k }\P{D\geq k }\\
	&=\P{S_k\leq t,S_{k+1+D_k}>t'}\P{D\geq k }\\
	&= \E{\ind_{\{S_k\leq t\}}\P{\sum_{i=k}^{k+D_k}E_i>t'-S_k\,\Bigg|\, S_k}}\P{D\geq k}. 
	\ea \ee 
	As $d(i)\geq x$ (resp.\ $d(i)\leq x$), Lemma~\ref{lemma:stochdomexp} with $\lambda=x$ yields the upper bound (resp.\ lower bound)
	\be 
	\e^{-x(t'-t)}\P{D\geq k}\E{\ind_{\{S_k\leq t\}}\e^{x(S_k-t)}},
	\ee
	which proves $(i)$ (resp.\ $(ii)$). To prove $(iii)$, we immediately arrive at the upper bound $\P{S_k\leq t}$ by omitting the events $\{D\geq k\}$ and $\{S_{D+1}>t'\}$. To obtain a lower bound, we use the same steps as in~\eqref{eq:probbound}  to arrive at 
	\be \ba 
	\P{D\geq k, S_k\leq t, S_{D+1}>t' } &=\E{\ind_{\{S_k\leq t\}}\P{\sum_{i=k}^{k+D_k}E_i>t'-S_k\,\Bigg|\, S_k}}\P{D\geq k}\\ 
	&\geq \P{S_k\leq t} \P{\sum_{i=k}^{k+D_k}E_i=\infty}\P{D=\infty}. 
	\ea \ee 
	As Assumption~\eqref{ass:A2} is not satisfied, Lemma~\ref{lemma:Dlifedistr} yields that the last probability is non-zero. As, additionally, Assumption~\eqref{ass:A1} is satisfied, (the proof of) Lemma~\ref{lemma:stochdomexp} yields that the second probability equals $\P{D_k=\infty}\geq \P{D_0=\infty}=\P{D=\infty}$, so that we arrive at the desired lower bound and conclude the proof.
\end{proof}

\begin{proof}[Proof of Proposition~\ref{prop:mdptilt}]
	We start by proving~\eqref{eq:mdptiltmin}. For $\theta=0$, the result directly follows from Lemma~\ref{lemma:mdp}, since $\Var(S_k)=\varphi_2(k)$, which tends to infinity with $k$ by Assumption~\eqref{ass:varphi2}. We then take $\theta>0$. First, we prove an upper bound. By directly bounding $S_k$ from above by $\varphi_1(k)-z\varphi_2(k)$ and applying Lemma~\ref{lemma:mdp}, we arrive at the upper bound
	\be \ba 
	\limsup_{k\to\infty}{}&\frac{1}{\varphi_2(k)}\log\E{\ind_{\{S_k\leq \varphi_1(k)-z\varphi_2(k)\}}\e^{\theta (S_k-(\varphi_1(k)-y\varphi_2(k)))}}\\ 
	&\leq  \theta (y-z)+\lim_{k\to\infty}\frac{1}{\varphi_2(k)}\log \P{S_k\leq \varphi_1(k)-z\varphi_2(k)}=\theta (y-z)-\frac{z^2}{2}.
	\ea\ee 
	We then prove a lower bound. We fix $\eps>0$ and bound
	\be \ba 
	\mathbb E{}&\Big[\ind_{\{S_k\leq \varphi_1(k)-z\varphi_2(k)\}}\e^{\theta(S_k-(\varphi_1(k)-y\varphi_2(k)))}\Big]\\ 
	&\geq \E{\ind_{\{(S_k-\varphi_1(k))/\varphi_2(k)\in(-(z+\eps),-z]\}}\e^{\theta(S_k-(\varphi_1(k)-y\varphi_2(k)))}}\\
	&\geq \P{\frac{S_k-\varphi_1(k)}{\varphi_2(k)}\in(-(z+\eps),-z]}\e^{\theta (y-(z+\eps))\varphi_2(k))}. 
	\ea\ee 
	We write the probability as 
	\be 
	\P{S_k-\varphi_1(k)\leq -z\varphi_2(k)}-\P{S_k-\varphi_1(k)\leq -(z+\eps)\varphi_2(k)}. 
	\ee 
	We then apply Lemma~\ref{lemma:mdp} to bound the first probability from below and the second probability from above, to obtain, for $\eps$ sufficiently small and all $k$ sufficiently large, the lower bound 
	\be 
	\exp\Big(-\frac{(z+\eps/3)^2}{2}\varphi_2(k)\Big)-\exp\Big(-\frac{(z+2\eps/3)^2}{2}\varphi_2(k)\Big)=(1-o(1))\exp\Big(-\frac{(z+\eps/3)^2}{2}\varphi_2(k)\Big).
	\ee 
	All combined, we thus arrive at 
	\be \ba 
	\liminf_{k\to\infty}{}&\frac{1}{\varphi_2(k)}\log\E{\ind_{\{S_k\leq \varphi_1(k)-z\varphi_2(k)\}}\e^{\theta(S_k-(\varphi_1(k)-y\varphi_2(k)))}}\geq \theta (y-(z+\eps))-\frac{(z+\eps/3)^2}{2}. 
	\ea \ee 
	Since $\eps>0$ is arbitrary, we arrive at the desired result. The case $\theta<0$ follows mutatis mutandis, which concludes the proof of~\eqref{eq:mdptiltmin}.
\end{proof}

\begin{proof}[Proof of Lemma~\ref{lemma:func}]
	The proofs are based on the following. Suppose that $\ell\colon\N_0\to\R$ and $f\colon\N_0\to(0,\infty)$ are functions such that $\ell$ tends to zero  and 
	\be \label{eq:fdiv}
	\sum_{i=0}^\infty \frac{1}{f(i)}=\infty. 
	\ee 
	Then, 
	\be \label{eq:sumdom}
	\sum_{i=0}^{k-1} \frac{\ell(i)}{f(i)}=o\bigg(\sum_{i=0}^{k-1}\frac{1}{f(i)}\bigg)\qquad \text{as }k\to\infty. 
	\ee 
	Indeed, fix $\eps>0$ and let $i_\eps\in\N_0$ be such that $|\ell(i)|<\eps$ for all $i\geq i_\eps$. Then, for $k i_\eps$, 
	\be 
	\bigg|\sum_{i=0}^{k-1}\frac{\ell(i)}{f(i)}\bigg|\leq \sum_{i=0}^{i_\eps-1}\frac{|\ell(i)|}{f(i)}+\eps \sum_{i=i_\eps}^{k-1} \frac{1}{f(i)}\leq \sum_{i=0}^{i_\eps-1}\frac{|\ell(i)|}{f(i)}+\eps \sum_{i=0}^{k-1}\frac{1}{f(i)}.  
	\ee 
	As a result, by~\eqref{eq:fdiv} and since $\eps$ is arbitrary we arrive at~\eqref{eq:sumdom}. Similarly, when $\ell(i)$ tends to infinity with $i$, 
	\be \label{eq:sumdomflip}
	\sum_{i=0}^{k-1}\frac{1}{f(i)}=o\bigg(\sum_{i=0}^{k-1}\frac{\ell(i)}{f(i)}\bigg)\qquad\text{as }k\to\infty.
	\ee 
	We now prove the lemma.
	
	$(a)$ By Assumption~\eqref{ass:A1}, we have that $\varphi_1(x)$ and $\varphi_1^{-1}(x)$ both tend to infinity with $x$. If $\varphi_2$ converges, both claims are immediate. If $\varphi_2$ tends to infinity, we use that $b$ diverges and apply~\eqref{eq:sumdom} with $\ell(i):=1/(b(i)+d(i))$ and $f(i):=b(i)+d(i)$. For the second part, we apply~\eqref{eq:sumdom} with $\ell(i):=d(i)/(b(i)+d(i))$ and $f(i):=(b(i)+d(i))/d(i)$.
	
	$(b)$ Again, we have that $\varphi_1(x)$ tends to infinity with $x$. Fix $\eps>0$. There exists $i_\eps\in\N_0$ such that $d(i)\leq \overline d+\eps$ for all $i\geq i_\eps$. Hence, for $k> i_\eps$, 
	\be 
	\rho_1(k)=\sum_{i=0}^{i_\eps-1}\frac{d(i)}{b(i)+d(i)}+(\overline d+\eps)\sum_{i=i_\eps}^{k-1}\frac{1}{b(i)+d(i)}\leq \sum_{i=0}^{i_\eps-1}\frac{d(i)}{b(i)+d(i)}+(\overline d+\eps)\varphi_1(k). 
	\ee 
	As $\eps$ is arbitrary, we arrive at the first result. The second is immediate, since $\rho_1(k)/\varphi_1(k)\geq 0$. The second part of $(b)$ follows in an analogous way.
	
	$(c)$ By Assumption~\eqref{ass:varphi2}, we have that $\varphi_2(x)$ tends to infinity with $x$. We then apply~\eqref{eq:sumdom} with $\ell(i)=d(i)^2-(d^*)^2$ and $f(i):=1/(b(i)+d(i))^2$. 
	
	$(d)$ By Assumption~\eqref{ass:A1}, we have that $\varphi_1(x)$ and $\varphi_1^{-1}(x)$ both tend to infinity with $x$. Further, we have $\alpha=o(\varphi_1)$, which follows from~\eqref{eq:sumdom} with $\ell(i):=d(i)-d^*$ and $f(i):=1/(b(i)+d(i))$. As $\cK_\alpha(t)=\alpha(\varphi_1^{-1}(t))$ we thus obtain the first result. The second result follows as for any $\eps>0$ there exists $t_\eps>0$ such that for any $t\geq t_\eps$, 
	\be\ba 
	|\cK_\alpha(t)-\cK_\alpha(t-s(t))|\leq\int_{\varphi_1^{-1}(t-s(t))}^{\varphi_1^{-1}(t)}\!\frac{|d(\lfloor x\rfloor )-d^*|}{b(\lfloor x\rfloor )+d(\lfloor x\rfloor )}\,\dd x\leq \eps \int_{\varphi_1^{-1}(t-s(t))}^{\varphi_1^{-1}(t)}\!\frac{1}{b(\lfloor x\rfloor )+d(\lfloor x\rfloor )}\,\dd x=\eps s(t). 
	\ea\ee 
	The last step follows from the definition of $\varphi_1$ (and the extension of its domain to $(0,\infty)$ by linear interpolation). As $\eps$ arbitrary, we arrive at the desired result.
\end{proof}

\end{document}